\documentclass{amsart}

\usepackage{amssymb,amsmath,amsfonts,euscript,amscd,mathrsfs,amsthm,graphicx}

\numberwithin{equation}{section}
\numberwithin{figure}{section}

\newtheorem{lemma}{Lemma}[section]
\newtheorem{theorem}[lemma]{Theorem}
\newtheorem{corollary}[lemma]{Corollary}
\newtheorem{proposition}[lemma]{Proposition}

\newtheorem{definition}{Definition}[section]
\theoremstyle{remark}

\newcommand{\D}{\mathbb{D}}
\newcommand{\half}{\mathbb{H}}
\newcommand{\R}{\mathbb{R}}
\newcommand{\N}{\mathbb{N}}
\newcommand{\Z}{\mathbb{Z}}

\newcommand{\prob}[2]{\mathbf{P}^{#1} \left \{ #2 \right \}}
\newcommand{\abs}[1]{\left | #1  \right |}
\newcommand{\paren}[1]{\left ( #1 \right )}
\newcommand{\bracket}[1]{\left [ #1 \right ]}
\newcommand{\set}[1]{\left \{ #1 \right \}}
\newcommand{\ev}[2]{\mathbf{E}^{#1} \left[ #2 \right ] }

\DeclareMathOperator{\diam}{diam}
\DeclareMathOperator{\hcap}{hcap}
\DeclareMathOperator{\hm}{hm}
\DeclareMathOperator{\imag}{Im}

\DeclareMathOperator{\osc}{osc}
\DeclareMathOperator{\rad}{rad}

\DeclareMathOperator{\SLE}{SLE}

\DeclareMathOperator{\dist}{dist}
\newcommand{\C}{\mathbb{C}}

\begin{document}
\author{Shawn Drenning}
\title [ERBM and Loewner Equations in Multiply Connected Domains]{Excursion Reflected Brownian Motion and Loewner Equations in Multiply Connected Domains}
\date{\today}

\begin{abstract}
Excursion reflected Brownian motion (ERBM) is a strong Markov process defined in a finitely connected domain $D \subset \C$ that behaves like a Brownian motion away from the boundary of $D$ and picks a point according to harmonic measure from infinity to reflect from every time it hits a boundary component.  We give a new construction of ERBM using its conformal invariance and discuss the relationship between the Poisson kernel and Green's function for ERBM and conformal maps into certain classes of finitely connected domains.  One important reason for studying ERBM is the hope that it will be a useful tool in the study of SLE in finitely connected domains.  To this end, we show how the Poisson kernel for ERBM can be used to derive a Loewner equation for simple curves growing in a certain class of finitely connected domains.
\end{abstract}
\maketitle

\section{Introduction}

\subsection{Motivation and Results}
Oded Schramm \cite{MR1776084} introduced a one parameter family of random processes now called \emph{Schramm-Loewner evolution} ($\SLE$) as a proposed scaling limit for many discrete models arising in statistical mechanics which were expected to be conformally invariant in the limit.  Since then, $\SLE$ has been extensively studied and has proven to be an important tool in providing mathematical rigor to a number of predictions in statistical mechanics.  The definition of $\SLE$ in simply connected domains uses a classical result of Charles Loewner \cite{MR1512136}.  Namely, if $\gamma\paren{t}:\paren{0,\infty} \rightarrow \half$ is a simple curve parametrized such that $\gamma\paren{t}$ has half-plane capacity $a\paren{t}$ and $g\paren{t}:\half \backslash \gamma\paren{t} \rightarrow \half$ is the unique conformal map satisfying $\displaystyle \lim_{z \rightarrow \infty}g_t\paren{z}-z=0$, then $g_t\paren{z}$ satisfies the initial value problem
\begin{equation}\label{LE}
\dot{g}_t\paren{z}=\frac{\dot{a}\paren{t}}{g_t\paren{z}-U_t},~~~g_0\paren{z}=z,
\end{equation}
where $U_t$ is a real-valued function called the \emph{driving function}.  The differential equation in \eqref{LE} is one example of a \emph{Loewner equation}.  If \eqref{LE} is solved with $\dot{a}\paren{t}=2$ and $U_t=\kappa B_t$, where $B_t$ is a standard one-dimensional Brownian motion, then the random family of maps $g_t$ is \emph{generated by a curve} \cite{MR2153402} in the sense that the domain of $g_t$ is equal to the unbounded component of $\half \backslash \gamma\paren{t}$ for a random family of curves $\gamma\paren{t}$.  If $0 \leq \kappa \leq 4$, then $\gamma\paren{t}$ is almost surely a simple curve.  When $\kappa>4$, $\gamma\paren{t}$ almost surely has self-intersections and when $\kappa \geq 8$ it is almost surely a space-filling curve. Chordal $\SLE_{\kappa}$ in $\half$ from $0$ to $\infty$ is defined to be the random family of curves $\gamma\paren{t}$.  

It is natural to ask whether an $\SLE$ process can be defined in multiply connected domains $D \subset \half$ in an analogous way.  That is, is it possible to find an analog of \eqref{LE} and an appropriate driving function so that the solution of the resulting initial value problem is generated by a curve with the properties one would expect of $\SLE_{\kappa}$?  An added difficulty of the multiply connected case is that not all $n$-connected domains are conformally equivalent.  In the simply connected case, Schramm was able to show that any stochastic process that satisfies the \emph{domain Markov property} and is conformally invariant must come from \eqref{LE} with $U_t$ a Brownian motion.  Part of what makes this work is that $\half$ with a simple curve removed is conformally equivalent to $\half$.  In the multiply connected case, requiring that $\SLE$ have the domain Markov property and be conformally invariant is not enough to uniquely determine the driving function.  Bauer and Friedrich (\cite{MR2111355},\cite{MR2230350}, \cite{MR2357634}) defined a candidate for $\SLE$ in multiply connected domains by solving a Loewner equation for a curve growing in a multiply connected domain.  They did not determine the ``right'' driving function for the process to be $\SLE$, but they were able to narrow down the possible choices.  In separate work (\cite{MR2393989},\cite{ZhanWholePlane}), Zhan took a similar approach.  He showed that, in the case of the annulus, if in addition to satisfying the domain Markov property and conformal invariance, $\SLE$ is also assumed to be reversible, then the driving function is uniquely determined. 

Recent work of Lawler \cite{Lawler2011} takes another approach to defining $\SLE$ in multiply connected domains.  His definition is motivated by work \cite{MR1992830} of Lawler, Schramm, and Werner.  They showed that if $D\subset \half$ is a simply connected domain such that $\half \backslash D$ is bounded and $\half$ and $D$ agree in a neighborhood of $0$, then there is a local martingale $M_t$ with the property that $\SLE$ in $D$ is $\SLE$ in $\half$ weighted by $M_t$.  In \cite{MR2045953}, Lawler and Werner showed that this local martingale is related to the \emph{Brownian loop measure}.  This led Lawler to suggest in \cite{MR2518970} that $\SLE$ in a multiply connected domain could be defined by using the Brownian loop measure to specify its Radon-Nikodym derivative with respect to $\SLE$ in a simply connected domain.  In \cite{Lawler2011}, Lawler defines $\SLE$ in multiply connected domains in this way and in the case of the annulus, shows that the resulting process agrees with the one found by Zhan in \cite{ZhanWholePlane}.  Even though Lawler does not use a Loewner equation in a multiply connected domain to define $\SLE$, the analysis of a Loewner equation in a multiply connected domain is still an important aspect of his work.  The goal of this paper is to better understand the Loewner equations appearing in work on $\SLE$ in multiply connected domains.

The study of Loewner equations in multiply connected domains is not new and goes as far back as 1950 in work by Komatu \cite{MR0046437}. These equations frequently feature special functions that make the calculation work, but are introduced without motivation.  In most cases, we expect such functions can be given a probabilistic interpretation.  For instance, the special function in \eqref{LE} is $\paren{z,x} \mapsto \frac{1}{z-x}$ and the probabilistic interpretation is that the imaginary part of $\frac{1}{z-x}$ is equal to $-\pi H_{\half}\paren{z,x}$, where $H_{\half}\paren{z, \cdot}$ is the Poisson kernel for Brownian motion in $\half$.  We call a domain $D\subset \half$ a \emph{chordal standard domain} if it is obtained by removing a finite number of horizontal line segments from the upper half plane.  In this paper, we study a strong Markov process, \emph{excursion reflected Brownian motion} (ERBM), whose Poisson kernel can be used to prove a Loewner equation for a simple curve growing in a chordal standard domain.  Excursion reflected Brownian motion in a simply connected domain is just Brownian motion and, in this case, the Loewner equation we get is just \eqref{LE}

Roughly speaking, if $D \subset \C$ is a domain with $n$ ``holes,'' ERBM is a strong Markov process that has the distribution of a Brownian motion away from $\partial D$ and picks a point according to \emph{harmonic measure from $\infty$} to reflect from every time it hits $\partial D$.  To understand the behavior of ERBM, we consider the case that $D=\C \backslash \D$. In this case, every time ERBM hits $\partial \D$, it picks a point uniformly on $\partial \D$ to reflect from.  ERBM has what Walsh (\cite{MR509476}, pg. 37) has called a ``roundhouse singularity'' in a neighborhood of $\D$.  That is, in any neighborhood of a time that it hits $\partial \D$, it will hit $\partial \D$ uncountably many times and jump randomly from point to point on $\partial \D$.  Finally, an important property of ERBM is that it is conformally invariant.  This will be clear once we more precisely define what it means to ``pick a point according to harmonic measure from $\infty$ to reflect from.''

The existence of ERBM follows from more general work of Fukushima and Tanaka in \cite{MR2139028}.  Their work uses the theory of Dirichlet forms and does not take advantage of the conformal invariance of ERBM.  An alternative construction making explicit use of the conformal invariance of ERBM was proposed by Lawler in \cite{MR2247843}.  He proposed that ERBM could be defined in any domain with ``one hole'' by first constructing the process in $\C \backslash \D$ using excursion theory and then defining it in any domain conformally equivalent to $\C \backslash \D$ via conformal invariance.  To define ERBM in a domain with ``$n$ holes,'' multiple copies of the process defined in a domain with ``one hole'' can be pieced together.  We take this basic approach and give a new construction of ERBM.
 
A function is ER-harmonic if it satisfies the mean value property with respect to ERBM.  More precisely, a function $u$ is ER-harmonic if it is harmonic on $D$ and, for any curve $\eta$ surrounding a boundary component $A$ of $D$,
\[u(A)=\int_{\eta} u(z) \frac{H_{\partial U}(A,z)}{\mathcal{E}_U(A, \eta)}\abs{dz},\]
where $\displaystyle \frac{H_{\partial U}(A_i,z)}{\mathcal{E}_U(A_i, \eta)}$ is the density for the distribution of the first time ERBM started at $A$ hits $\eta$. It turns out that a harmonic function $u$ on $D$ that is constant on each connected component of $\partial D$ is ER-harmonic if and only if it is the imaginary part of a holomorphic function on $D$.  For this reason, the study of ER-harmonic functions is a useful tool in the study of conformal maps into certain classes of finitely connected domains.

Two important ER-harmonic functions are the Poisson kernel $H_D^{ER}\paren{z,w}$ and Green's function $G_D^{ER}\paren{z,w}$ for ERBM. In order to define these functions, it is necessary to choose at least one boundary component of $D$ at which to kill the ERBM.  Once this is done, the definitions and many of the properties of the Poisson kernel and Green's function for ERBM are similar to those for usual Brownian motion.  The Poisson kernel for ERBM was first considered by Lawler in \cite{MR2247843} as a way of understanding a classical theorem \cite{MR510197} of complex analysis stating that any $n$-connected domain $D \subset \C$ is conformally equivalent to a chordal standard domain.  He sketched a proof showing that the imaginary part of any such map is equal to a real multiple of the Poisson kernel for ERBM.  We give a complete proof here.  Furthermore, we use the Green's function for ERBM to give an interpretation of two other classical conformal mapping theorems.

We finish this section by stating our main result.  This result was conjectured by Lawler in \cite{MR2247843}.  Let $D$ be a chordal standard domain, $\gamma:\paren{0,\infty} \rightarrow D$ be a simple curve with $\gamma \paren{0}=0$, $D_t=D \backslash \gamma\paren{0,t}$, and $b\paren{t}$ be the \emph{excursion reflected half-plane capacity} of $\gamma\paren{0,t}$.  Excursion reflected half-plane capacity is a generalization of half-plane capacity to multiply connected domains and parameterizing our curves so the excursion reflected half-plane capacity is a differentiable function of time is the most convenient parametrization. 

\begin{theorem}\label{MainThm}
For each $t$, there is a unique conformal map $h_t:D_t \rightarrow h_t\paren{D_t}$ such that $h_t\paren{D_t}$ is a chordal standard domain and $\displaystyle \lim_{z \rightarrow \infty} h_t\paren{z}-z=0$.  Furthermore, this map satisfies the initial value problem
\[\dot{h}\paren{t}\paren{z}=-\dot{b}\paren{t}\mathcal{H}^{ER}_{h_t\paren{D_t}}\paren{h_t\paren{z},\tilde{U}_t},~~~h_0\paren{z}=z,\]
where $\tilde{U}_t=h_t\paren{\gamma\paren{t}}$ and $\mathcal{H}^{ER}_{h_t\paren{D_t}}\paren{\cdot,\tilde{U}_t}$ is a conformal map with imaginary part a real multiple of $H_D^{ER}\paren{\cdot, \tilde{U}_t}$.
\end{theorem}
If $D=\half$, then this theorem is just a restatement of \eqref{LE}.

\subsection{Outline of the Paper}

Section \ref{sectBG} sets notation and contains necessary background material.  In particular, we outline the proof of the chordal Loewner equation for a simple curve growing in $\half$ as in \cite{MR2129588}.  The basic structure of the proof of the analogous result for chordal standard domains is the same and uses some of the same preliminary results.

In Section \ref{sectERBM} we define and construct ERBM in finitely connected domains $D \subset \C$. First, we construct the process in $\C \backslash \D$ by explicitly defining a transition kernel for ERBM in terms of the transition kernels for Brownian motion and reflected Brownian motion and then using general theory to show that there actually is a strong Markov process with this transition kernel.  Finally, we check that the strong Markov process we obtain satisfies our definition of ERBM.  Our construction is motivated by a similar construction of Walsh's Brownian motion in \cite{MR1022917}.  Once we have ERBM in $\C \backslash \D$, we define ERBM in any domain conformally equivalent to $\C \backslash \D$ via conformal invariance.  In Section \ref{chapERBMFC} we show how countably many independent ERBMs in $1$-connected domains can be pieced together to construct an ERBM in an $n$-connected domain.  ERBM in $D$ induces a discrete time Markov chain on the connected components of the boundary of $D$, which we discuss in Section \ref{chapMC}.  This chain was observed by Lawler in \cite{MR2247843} and appears implicitly in classical work on conformal mapping of multiply connected domains.  We conclude the section with a brief discussion of ER-harmonic functions.  We prove a maximal principle for ER-harmonic functions and show how ERBM can be used to construct ER-harmonic functions.

Section \ref{sectPK} introduces the Poisson kernel $H_D^{ER}\paren{z,w}$ for ERBM and proves some of its basic properties.  We gather a number of estimates for $H_D^{ER}\paren{z,w}$ that are used in Section \ref{sectLE} and discuss the connection between $H_D^{ER}\paren{\cdot,w}$ and conformal maps into chordal standard domains.


Section \ref{sectGF} introduces the Green's function $G_D^{ER}\paren{z,w}$ for ERBM in $D$ and proves some of its basic properties.  In Section \ref{sectGFPOF} we use the theory of Green's function for ERBM to prove two formulas from Section \ref{sectERBM} necessary to show our construction of ERBM is well-defined.  We conclude the section by showing how $G_D^{ER}\paren{z,\cdot}$ can be used to construct conformal maps into circularly-slit annuli.

In Section \ref{sectLE} we prove Theorem \ref{MainThm}.  We start by defining $\mathcal{H}^{ER}\paren{\cdot,x}$, the \emph{complex Poisson kernel} for ERBM, for any finitely-connected domain $D \subset \half$ with $\R \subset \partial D$.  The complex Poisson kernel for ERBM, initially considered by Lawler in \cite{MR2247843}, has the property that for any $x \in \R$, $\mathcal{H}_D^{ER}\paren{\cdot,x}$ is a conformal map into a chordal standard domain with imaginary part equal to $\pi H_D^{ER}\paren{\cdot,x}$.  We use $\mathcal{H}_D^{ER}\paren{\cdot,x}$ to show that there is a unique conformal map $\varphi_D$ from $D$ into a chordal standard domain satisfying $\displaystyle \lim_{z \rightarrow \infty}\varphi_D\paren{z}-z=0$.  The map $h_t$ in Theorem \ref{MainThm} is equal to $\varphi_{g_t\paren{D_t}}\circ g_t$, where $g_t$ is as in \eqref{LE}.  There are two main steps to the proof of Theorem \ref{MainThm}.  First we prove the result at $t=0$.  This proof is similar in spirit to the proof of the analogous result for $g_t$.  The second main step is to show that $\tilde{U}_t=h_t\paren{\gamma\paren{t}}$ is a well-defined continuous function.  We do this by combining the analogous fact for $U_t=g_t\paren{\gamma\paren{t}}$ with derivative estimates for $\varphi_D\paren{x}$ restricted to $\R$.  The key observation is that $\varphi_D'\paren{x}=\pi H_D^{ER}\paren{\infty,x}$, where $H_D^{ER}\paren{\infty,x}$ is the ``normal derivative'' of $H_D^{ER}\paren{\cdot,x}$ at $\infty$.  This allows us to use appropriate Poisson kernel estimates to provide the necessary estimates for $\varphi_D'\paren{x}$.  Since $g_t\paren{D_t}$ varies with $t$, we need Poisson kernel estimates that are uniform over certain classes of domains.

I would like to thank my advisor Greg Lawler for suggesting this line of research and for many useful conversations pertaining to it.



\section{Background} \label{sectBG}
\subsection{Some Notation}

We denote the unit disk in $\C$ centered at the origin by $\D$ and the upper half-plane by $\half$. We let $\mathcal{Y}_n$ consist of all subdomains of $\C$ with $n$ ``holes.''  More precisely, let $\mathcal{Y}_n$ consist of all connected domains of the form 
\[D=\C \backslash \bracket{A_0 \cup A_1 \cup \cdots \cup A_n},\]
where $A_0, A_1, \ldots ,A_n$ are closed disjoint subsets of $\C$ such that $A_i$ is simply connected, bounded, and larger than a single point for $1 \leq i \leq n$ (we allow $A_0$ to be empty) and $\C \backslash A_0$ is simply connected.  We denote $\displaystyle \bigcup_{i=0}^{\infty} \mathcal{Y}_i$ by $\mathcal{Y}$.

We call $A \subset \half$ a \emph{compact $\half$-hull} if $A=\half \cap \overline{A}$ and $\half \backslash A$ is simply connected and denote the set of all compact $\half$-hulls by $\mathcal{Q}$.  We denote the subsets of $\mathcal{Y}$ and $\mathcal{Y}_n$ consisting of domains such that $A_0$ is the union of $\C \backslash \half$ and a compact $\half$-hull by $\mathcal{Y}^*$ and $\mathcal{Y}_n^*$ respectively. We call $D \in \mathcal{Y}$ a \emph{chordal standard domain} if $D$ is the upper half-plane with a finite number of horizontal line segments removed.  We denote the set of chordal standard domains and $n$-connected chordal standard domains by $\mathcal{CY}$ and $\mathcal{CY}_n$ respectively.  If $D$ is a domain with $\C \backslash A_0=\half$, then if $r>0$, we let
\[D^r=\set{z\in D:\abs{z}\geq r}.\]
We also let
\[H^r=\set{z\in \half:\abs{z}\leq r} \]
and
\[\D_+=\D \cap \half.\]
We denote the open annulus centered at $x\in \R$ with inner radius $r$ and outer radius $R$ by $A_{r,R}\paren{x}$ and $A_{r,R} \paren{x} \cap \half$ by $A^+_{r,R}\paren{x}$.  We write $A_{r,R}$ and $A^+_{r,R}$ for $A_{r,R}\paren{0}$ and $A^+_{r,R}\paren{0}$ respectively.  Finally, we denote the open ball of radius $r$ centered at $z$ by $B_r\paren{z}$ and, if $x \in \R$, $B_r\paren{x}\cap \half$ by $B_r^+\paren{x}$. 

If $A$ is a subset of $\C$, we let the radius of $A$, denoted $\rad \paren{A}$, be the infimum over all $r>0$ such that $A\subset r\D$ and the diameter of $A$, denoted $\diam \paren{A}$, be the supremum over all $x,y \in A$ of $\abs{x-y}$.

We will use $c$ to denote a real constant that is allowed to change from one line to the next.  We write $f\paren{z} \sim g\paren{z}$ as $z \rightarrow a$ if $\lim_{z \rightarrow a}\frac{f\paren{z}}{g\paren{z}}=1$ and $f\paren{z} \asymp g\paren{z}$ as $z \rightarrow a$ if $f\paren{z}=O\paren{\abs{g\paren{z}}}$ and $g\paren{z}=O\paren{\abs{f\paren{z}}}$ as $z \rightarrow a$.

\subsection{Poisson Kernel for Brownian Motion}
Let $D \in \mathcal{Y}$ and let $\tau_D$ be the first time that a Brownian motion $B_t$ leaves $D$.  If $\partial D$ has at least one regular point for Brownian motion, then for each $z \in D$, the distribution of $B_{\tau_D}$ defines a measure $\hm_D\paren{z,\cdot}$ on $\partial D$ (with the $\sigma$-algebra generated by Borel subsets of $\partial D$) called \emph{harmonic measure in $D$ from $z$}.  We say $\partial D$ is \emph{locally analytic} at $w \in \partial D$ if $\partial D$ is an analytic curve in a neighborhood of $w$.  If $\partial D$ is locally analytic at $w$, then in a neighborhood of $w$, $\hm_D\paren{z,\cdot}$ is absolutely continuous with respect to arc length and the density of $\hm_D\paren{z,\cdot}$ at $w$ with respect to arc length is called the \emph{Poisson kernel for Brownian motion} and is denoted $H_D\paren{z,w}$.  

Some of the domains we consider will have two-sided boundary points (a recurring example is when $D$ is a chordal standard domain).  If $w$ is a two-sided boundary point, we should really think of it as being two distinct boundary points, $w^+$ and $w^-$.  In such cases, by abuse of notation, we will sometimes write $H_D\paren{z,w}$ when we should consider $H_D\paren{z,w^+}$ and $H_D\paren{z,w^-}$ separately.

Harmonic measure is conformally invariant.  That is, if $f: D \rightarrow D'$ is a conformal map, then 
\[\hm_D\paren{z,V}=\hm_{D'}\paren{f\paren{z},f\paren{V}}.\]
Using this, we see that if $\partial D$ is locally analytic at $w$ and $\partial D'$ is locally analytic at $f\paren{w}$, then
\begin{equation} \label{PKCI}
H_{D'}\paren{f\paren{z},f\paren{w}}=\abs{f'\paren{w}}^{-1}H_D\paren{z,w}.
\end{equation}
It is well-known that
\begin{equation}\label{PKhalf}
H_{\half}\paren{x+iy,x'}=\frac{1}{\pi} \frac{y}{\paren{x-x'}^2+y^2}.
\end{equation}
Using \eqref{PKhalf}, it is straightforward to compute that if $\abs{x}>\epsilon$, then $H_{\half}\paren{\cdot,x}$ restricted to the boundary of $B_{\epsilon}^+\paren{0}$ is bounded above by 
\begin{equation}\label{PKhalfMax}
\max \set{\frac{\epsilon}{\pi\paren{x-\epsilon}^2},\frac{\epsilon}{\pi\paren{x+\epsilon}^2}}.
\end{equation}

Using \eqref{PKCI}, it is sometimes possible to explicitly compute the Poisson kernel for a simply connected domain.  Often though, having an estimate is good enough.  One particularly important estimate (\cite{MR2129588}, pg. 50) that we will use extensively is that if $\abs{z}\geq 2 \epsilon$, then
\begin{equation}\label{PKHalfRMHalfDisk}
H_{H^{\epsilon}}\paren{z,\epsilon e^{i \theta}}=2 H_{\half}\paren{z,0}  \sin \theta \bracket{1+O\paren{\frac{\epsilon}{\abs{z}}}},
\end{equation}
as $\frac{\epsilon}{\abs{z}} \rightarrow 0$.  In particular, if we fix a radius $r>0$, then the $O\paren{\frac{\epsilon}{\abs{z}}}$ term can be replaced with an $O\paren{\epsilon}$ term that is uniform over all $z$ with $\abs{z}>r$.

Using \eqref{PKCI}, \eqref{PKHalfRMHalfDisk}, and the map $z \mapsto \frac{-1}{z}$, we see that if $\abs{z}<r/2$, then
\begin{align}
H_{B_r^+\paren{0}}\paren{z,re^{i\theta}}&= \frac{1}{r^2}H_{H^{1/r}}\paren{\frac{-1}{z},\frac{-e^{-i\theta}}{r}} \nonumber \\
&=\frac{2}{r^2}H_{\half}\paren{\frac{-1}{z},0} \sin\paren{\theta}\bracket{1+O\paren{\frac{\abs{z}}{r}}} \nonumber \\
&=\frac{2\imag \bracket{z}}{r^2} \sin\paren{\theta}\bracket{1+O\paren{\frac{\abs{z}}{r}}} \label{PKHD},
\end{align}
as $\frac{\abs{z}}{r} \rightarrow 0$.  In particular for fixed $r$, the probability that a Brownian motion started at $z$ exits $B_r^+\paren{0}$ on $\abs{z}=r$ is comparable to $\frac{\imag \bracket{z}}{r}$ as $z \rightarrow 0$.

The function $H_D\paren{\cdot,w}$ can be characterized up to a positive multiplicative constant as the unique positive harmonic function on $D$ that is ``equal to'' the Dirac delta function at $w$ on $\partial D$. 

\begin{proposition}\label{noSubLinHar}
Let $D \in \mathcal{Y}$ be such that $\partial D$ is locally analytic at $w\in \partial D$.  Then $H_D\paren{\cdot,w}$ is up to a real constant multiple the unique positive harmonic function on $D$ that satisfies $H_D\paren{z,w} \rightarrow 0$ as $z \rightarrow w'$ for any $w' \in \partial D$ not equal to $w$.
\end{proposition}

Next, we prove an estimate analogous to \eqref{PKHalfRMHalfDisk} for $D \in \mathcal{Y}^*$. A useful observation \cite{MR2247843} that we will use in the proof of this estimate is that if $D_2 \subset D_1$ and $D_1$ and $D_2$ agree in a neighborhood of $w \in \partial D_1$, then
\begin{equation} \label{twoDomainPK}
H_{D_2}\paren{z,w}=H_{D_1}\paren{z,w}-\ev {z}{H_{D_1}\paren{B_{\tau_{D_2}},w}}.
\end{equation}

\begin{lemma} \label{PKStandardChordalRMHalfDisk}
Let $D \in \mathcal{Y}^*$ be such that $\partial D$ is locally analytic and $\partial D$ and $\partial \half$ agree in a neighborhood of $0$. If $\abs{z}>2\epsilon$, then
\[H_{D^{\epsilon}}\paren{z,\epsilon e^{i\theta}}=2 \sin \theta ~H_D\paren{z,0} \bracket{1+O\paren{\epsilon}},\]
where for any $r>0$, $O\paren{\epsilon}$ is uniform over all $z$ with $\abs{z}>r$.
\end{lemma}

\begin{proof}
Using \eqref{twoDomainPK}, we have
\begin{equation} \label{ERPKEstimate2}
H_{D^{\epsilon}}\paren{z,\epsilon e^{i\theta}}=H_{H^{\epsilon}}\paren{z,\epsilon e^{i \theta}}-\ev {z} {H_{H^{\epsilon}}\paren{B_{\tau_{D_{\epsilon}}},\epsilon e^{i\theta}}}.
\end{equation}
We can rewrite $\ev {z} {H_{H^{\epsilon}}\paren{B_{\tau_{D_{\epsilon}}},\epsilon e^{i\theta}}}$ as (where by convention $H_{D^{\epsilon}}\paren{z,w}=0$ if $w \notin \partial D^{\epsilon}$)
\begin{align}
&\sum_{i=0}^n \int_{\partial A_i}H_{H^{\epsilon}}\paren{w,\epsilon e^{i\theta}} H_{D^{\epsilon}}\paren{z, w}~\abs{dw} \nonumber \\
=&2\sin \theta \sum_{i=0}^n \int_{\partial A_i} H_{\half}\paren{w,0}\bracket{1+O\paren{\epsilon}} H_{D^{\epsilon}}\paren{z,w} ~\abs{dw} \nonumber \\
=& 2\sin \theta \bracket{\sum_{i=0}^n \int_{\partial A_i} H_{\half}\paren{w,0} H_{D^{\epsilon}}\paren{z, w} ~\abs{dw}} \bracket{1+O\paren{\epsilon}} \label{ERPKEstimate3}.
\end{align}
Applying \eqref{twoDomainPK} again, we have
\begin{equation} \label{ERPKEstimate4}
H_{D^{\epsilon}}\paren{z, w}=H_{D}\paren{z, w}-\ev{z}{H_D\paren{B_{\tau_{D^{\epsilon}}},w}}.
\end{equation}
The probability that a Brownian motion in $D$ started at $z$ leaves $D^{\epsilon}$ on $\partial B^+_{\epsilon}\paren{0}$ is less than the probability that a Brownian motion in $H^{\epsilon}$ does the same thing, which is $O\paren{\epsilon}$ by \eqref{PKHalfRMHalfDisk}.  If $R$ is such that $B_R^+\paren{0} \subset D$, then in order for a Brownian motion in $D$ started on $\partial B^+_{\epsilon}\paren{0}$ to not exit $D$ on $\R$, it has to leave $B_R^+\paren{0}$ before hitting $\R$.  By \eqref{PKHD}, the probability of this event is $O\paren{\epsilon}$ and hence, the probability that a Brownian motion in $D$ started on $\partial B^+_{\epsilon}\paren{0}$ does not exit $D^{\epsilon}$ on $\R$ is $O\paren{\epsilon}$.  It follows that
\begin{equation}\label{ERPKEstimate5}
\ev{z}{H_D\paren{B_{\tau_{D^{\epsilon}}},w}}=O\paren{\epsilon^2}.
\end{equation}
Substituting \eqref{ERPKEstimate4} into \eqref{ERPKEstimate3} and using \eqref{ERPKEstimate5}, we see that \eqref{ERPKEstimate3} is equal to
\begin{align*}
& 2\sin \theta \bracket{\sum_{i=0}^n \int_{\partial A_i} H_{\half}\paren{w,0} H_D\paren{z, w} ~\abs{dw}} \bracket{1+O\paren{\epsilon}}\\
=& 2\sin \theta ~\ev {z}{H_{\half}\paren{B_{\tau_D},0}}\bracket{1+O\paren{\epsilon}}.
\end{align*}
Combining this with \eqref{PKHalfRMHalfDisk} and \eqref{twoDomainPK},  \eqref{ERPKEstimate2} becomes
\begin{equation} \label{ERPKEstimate6}
H_{D^{\epsilon}}\paren{z,\epsilon e^{i\theta}}=2\sin \theta ~H_D\paren{z,0}    \bracket{1+O\paren{\epsilon}}.
\end{equation}
\end{proof}

We will need an estimate for the derivative (in the second variable) of the Poisson kernel.

\begin{lemma}\label{PKDeriv}
Let $D \in \mathcal{Y}^*$ be such that $A_0=\C \backslash \half$ and $f_z\paren{x}:=H_D\paren{z,x}$.  If $x' \in \R$ and $r<\abs{z-x'}$ is such that $B_r^+\paren{x'}\subset D$, then $\abs{f_z'\paren{x'}} \leq \frac{4}{\pi r^2}.$
\end{lemma}

\begin{proof}
See \cite{DreThesis}.
\end{proof}

\subsection{Some Brownian Measures}
Let $D \in \mathcal{Y}$.  If $\partial D$ is locally analytic at $w$, then the \emph{boundary Poisson kernel} is defined by
\[H_{\partial D}\paren{w,z}=\frac{d}{dn} H_D\paren{w,z},\]
where $n$ is the inward pointing normal at $w$.  If $w$ is a two-sided boundary point, then we have two distinct boundary Poisson kernels, $H_{\partial D}\paren{w^+,z}$ and $H_{\partial D}\paren{w^-,z}$.  In such cases, by abuse of notation, we will sometimes write $H_{\partial D}\paren{w,z}$ when we should consider $H_{\partial D}\paren{w^+,z}$ and $H_{\partial D}\paren{w^-,z}$ separately.  If $f$ is a conformal map and $\partial f\paren{D}$ is locally analytic at $f\paren{w}$ and $f\paren{z}$, then
\begin{equation}\label{BdPKCI}
H_{\partial D}\paren{w,z}=\abs{f'\paren{w}}\abs{f'\paren{z}}H_{\partial f\paren{D}}\paren{f\paren{w},f\paren{z}}.
\end{equation}
If $D$ is as in Lemma \ref{PKStandardChordalRMHalfDisk}, then using Lemma \ref{PKStandardChordalRMHalfDisk}, we have that
\begin{equation} \label{BoundaryPKSCEstimate}
H_{\partial D^{\epsilon}}\paren{z,\epsilon e^{i\theta}}=2 \sin \theta ~H_{\partial D}\paren{z,0} \bracket{1+O\paren{\epsilon}},
\end{equation}
where for any $r>0$, $O\paren{\epsilon}$ is uniform over all $z \in \partial D$ with $\abs{z}>r$.

The definition of excursion reflected Brownian motion uses excursion measure.  Excursion measure is usually defined as a measure on paths between two boundary points of $D$.  Since we will only be interested in the norm of this measure, the definition we give of excursion measure is the norm of excursion measure as defined elsewhere (\cite{MR2129588}, \cite{MR2247843}).  

\begin{definition}
Suppose $D \subset \C$ is a domain with locally analytic boundary and $V$ and $V'$ are disjoint arcs in $\partial D$.  Then
\begin{equation*}
\mathcal{E}_D\paren{V,V'} :=\int_{V}\int_{V'} H_{\partial D}\paren{z,w} \abs{dz}\abs{dw}
\end{equation*}
is called \emph{excursion measure}. Excursion measure normalized to have total mass one is called \emph{normalized excursion measure} and is denoted $\overline{\mathcal{E}}_D\paren{V,\cdot}$.  
\end{definition}

Using \eqref{BdPKCI}, we can check that $\mathcal{E}_D$ is conformally invariant.  This allows us to define $\mathcal{E}_D\paren{V,V'}$ even if $D$ does not have locally analytic boundary.  We will often write $\mathcal{E}_D\paren{A,V}$ for $\mathcal{E}_D\paren{\partial A,V}$ and $\overline{\mathcal{E}}_D\paren{A,V}$ for $\overline{\mathcal{E}}_D\paren{\partial A,V}$.  We will also write $H_{\partial D}\paren{A,z}$ as shorthand for the quantity $\int_{\partial A} H_{\partial D}\paren{z,w} \abs{dz}$.  Using \eqref{BdPKCI}, we see that if $f: D \rightarrow D'$ is a conformal map, then 
\[H_{\partial D}\paren{A,z}=H_{\partial f\paren{D}}\paren{f\paren{A},f\paren{z}}\abs{f'\paren{z}}.\]
As a result, it is possible to define $H_{\partial D}\paren{A,z}$ even if $A$ does not have locally analytic boundary.

We conclude with a brief discussion of the Brownian bubble measure. Let $D \in \mathcal{Y}_n^*$ be such that $\partial D$ is locally analytic and $x\in \partial D \cap \R$.  Define the \emph{Brownian boundary bubble measure at $x$ of bubbles leaving D} by
\begin{equation} \label{BubDef}
\Gamma\paren{D; x}=\Gamma\paren{D;x~|~\half,x}=\pi\int_{\partial D} H_{\partial D}\paren{x,z}H_{\half}\paren{z,x}\abs{dz}.
\end{equation}
It is not hard to check (see \cite{DreThesis}) that
\begin{equation}\label{BubbleAt0}
\pi H_D\paren{z,0}=\pi H_{\half}\paren{z,0}-\imag\bracket{z}\Gamma\paren{D;0} \bracket{1+O\paren{\abs{z}}},~~~z \rightarrow 0.
\end{equation}

\subsection{Green's Function for Brownian Motion}

In what follows, let $D \in \mathcal{Y}$ be such that $\partial D$ has at least one regular point for Brownian motion.  In this setting, it is possible to define a (a.s. finite) Green's function for Brownian motion $G_D\paren{z,w}$ (see, for instance, \cite{MR2129588}).  By convention, we scale $G_D\paren{z,\cdot}$ so that it is the density for the occupation time of Brownian motion.  As a result, what we mean by $G_D$ may differ by a factor of $\pi$ from what appears elsewhere.

It is well-known that $G_D\paren{z,w}=G_D\paren{w,z}$ and that $G_D\paren{z,\cdot}$ can be characterized as the unique harmonic function on $D \backslash \set{z}$ such that $G_D\paren{z,w} \rightarrow 0$ as $w \rightarrow \partial D$ and 
\begin{equation} \label{GreenAsztow}
G_D\paren{z,w}=\frac{-\log \abs{z-w}}{\pi}+O\paren{1},
\end{equation}
as $z \rightarrow w$.  Another property of $G_D\paren{z,w}$ is that it is conformally invariant.  That is, if $f:D \rightarrow D'$ is a conformal map, then $G_{f\paren{D}}\paren{f\paren{z},f\paren{w}}=G_D\paren{z,w}$.  Finally, it is well-known that
\begin{equation} \label{GFUniDisk}
G_{r\D}\paren{0,z}=-\frac{\log r - \log \abs{z}}{\pi}
\end{equation}
and
\begin{equation}\label{Greenztoi}
G_{\half}\paren{x+iy,i}=\frac{1}{2\pi} \log \frac{x^2+\paren{y+1}^2}{x^2+\paren{y-1}^2}.
\end{equation}

The normal derivative of $G_D\paren{z,\cdot}$ at $w \in \partial D$ is equal to $2H_D\paren{z,w}$.  While we will not need this fact, we will need the following lemma that is used in the proof and that we will use when we prove a similar statement for the Green's function for ERBM.
\begin{lemma} \label{GreenInt}
\[\int_0^{\pi}G_{\half}\paren{ e^{i\theta},\epsilon i} \sin \theta ~d\theta =\epsilon + O\paren{\epsilon^2},\]
as $\epsilon \rightarrow 0$.
\end{lemma}

\begin{proof}
See \cite{DreThesis}.
\end{proof}

Finally, if $D_2 \subset D_1$, it is easy to check that
\begin{equation}
G_{D_2}\paren{z,w}=G_{D_1}\paren{z,w}-\ev{z}{G_{D_1}\paren{B_{\tau_{D_2}},w}} \label{twoDomainGF}.
\end{equation}

\subsection{Chordal Loewner Equation}
In this section, we outline the proof of the chordal Loewner equation for a simple curve growing in the upper half-plane as presented in \cite{MR2129588}.  Our purpose is both to motivate the proof we give of the analogous result in non-simply connected domains and to gather some preliminary results needed for that proof.  The omitted proofs can be found in one of \cite{MR2129588} or \cite{DreThesis}.

The Loewner equation we are interested in is a differential equation governing the behavior of the conformal map that maps the upper half-plane with a simple curve removed onto the upper half-plane.  The next proposition shows that there is a unique such conformal map with a specified asymptotic at infinity.

\begin{proposition} \label{HalfRMHullCon}
If $A \in \mathcal{Q}$, then there is a unique conformal map $g_A:\half \backslash A \rightarrow \half$ such that 
\[\displaystyle \lim_{z \rightarrow \infty}\bracket{g_A\paren{z}-z}=0.\]
In particular, $g_A$ has an expansion near infinity of the form
\[g_A\paren{z}=z+\frac{a_1}{z}+O\paren{{\abs{z}^{-2}}},~~~ z \rightarrow \infty.\]
\end{proposition}

The constant $a_1$ is called the \emph{half-plane capacity (from infinity)} for $A$ and is denoted $\hcap\paren{A}$.  There are several different ways to compute $\hcap\paren{A}$.

\begin{proposition}\label{HCapEqualities}
Suppose $A \in \mathcal{Q}$, $B_t$ is a Brownian motion in $\half$, and $\tau$ is the first time that $B_t$ leaves $\half \backslash A$.  Then for all $z \in \half \backslash A$,
\[\imag\bracket{z-g_A\paren{z}}=\ev {z}{\imag
\bracket{B_{\tau}}}.\]
Also, $\hcap\paren{A}$ is equal to each of the following.
\begin{enumerate}
\item $\displaystyle \lim_{y \rightarrow \infty}y\ev{iy}{\imag\bracket{B_{\tau}}}.$
\item $\displaystyle  \frac{2r}{\pi}\int_0^{\pi}\ev{re^{i\theta}}{\imag\bracket{B_{\tau}}}\sin \theta ~d\theta$ for any $r>0$ such that $\rad\paren{A}<r$.
\end{enumerate}
\end{proposition}

We give one more interpretation of $\hcap\paren{A}$.

\begin{proposition}\label{HCapMeas}
Let $A \in \mathcal{Q}$ and $D=\half \backslash A$ and for any Borel subset $V$ of $\partial A$, define
\[\mu_{A}\paren{V}=\lim_{y \rightarrow \infty}  y\hm_D\paren{iy,V}.\]
The following statements hold.
\begin{enumerate}
\item  If $\rad\paren{A}<R$, then $\displaystyle \mu_{A}\paren{V}=\frac{2R}{\pi} \int_0^{\pi} \hm_{D}\paren{Re^{i\theta},V} \sin \theta ~d\theta.$
\item  $\mu_A$ is a measure.
\item  $\displaystyle \hcap\paren{A}=\int \imag\bracket{z} d\mu_{A}\paren{z}.$
\end{enumerate}
\end{proposition}

Using Proposition \ref{HCapEqualities}, we can find a uniform bound on the difference between $g_A\paren{z}$ and $z+\frac{\hcap\paren{A}}{z}$ in terms of $\hcap\paren{A}$ and $\rad\paren{A}$.

\begin{proposition}\label{SCLEat0}
There is a $c < \infty$ such that for all $A \in \mathcal{Q}$ and $\abs{z} \geq 2\rad\paren{A}$,
\[ \left | z - g_A\paren{z} + \frac{\hcap\paren{A}}{z} \right | \leq c \frac{\hcap\paren{A}\rad\paren{A}}{\abs{z}^2}.\]
\end{proposition}

This result can be interpreted as a proof of the chordal Loewner equation at $t=0$.  That is, if we think of $A$ as being the trace of a simple curve at a small time $t$, this result shows that the time derivative of $g_A$ at $0$ is equal to the time derivative of $\hcap\paren{A}$ at $0$ divided by $z$.
 
In what follows, let $\gamma: \left[0, \infty \right) \rightarrow \C$ be a simple curve such that $\gamma\paren{0}\in \R$ and $\gamma\paren{0,\infty} \subset \half$.  Let $a\paren{t}=\hcap\paren{\gamma_t}$ and assume (reparametrizing if necessary) that $a\paren{t}$ is $C^1$.  For each $t \geq 0$, let $\gamma_t := \gamma\bracket{0,t}$ and $g_t : \half \backslash \gamma_t \rightarrow \half$ be the unique conformal transformation satisfying $\displaystyle \lim_{z \rightarrow \infty} g_t\paren{z}-z=0$. For each $s>0$, let $\gamma^s\paren{t} = g_s\paren{\gamma\paren{s+t}}$ and $g_{s,t}=g_{\gamma^{s}_{t-s}}$.  Observe that $g_t = g_{s,t} \circ g_s$.  The next proposition shows that $g_t$ maps the ``tip'' of $\gamma_t$ to a unique $U_t \in \R$ and that the resulting function $t \mapsto U_t$ is continuous.

\begin{proposition}\label{UtCont}
For all $t>0$, there is a unique $U_t \in \R$ such that
\begin{equation*}
\displaystyle \lim_{z \rightarrow \gamma\paren{t}} g_t\paren{z}=U_t,
\end{equation*}
where the limit is taken over $z \in \half \backslash \gamma_t$.  Furthermore,
\begin{equation*}
U_t=\lim_{s \rightarrow t^-} g_s\paren{\gamma\paren{t}}
\end{equation*}
and $t \mapsto U_t$ is continuous.
\end{proposition}

The main tool in the proof of Proposition \ref{UtCont} is the following technical lemma.

\begin{lemma}\label{gDiamCurve}
There exists a constant $c < \infty$ such that if $0 \leq s < t \leq t_0< \infty$, then
\[\diam \bracket{g_s\paren{\gamma\paren{s,t}}}\leq c \sqrt{\diam\paren{\gamma\bracket{0,t_0}}\osc\paren{\gamma,t-s,t_0}}\]
and
\[\left\|g_s-g_t\right\|_{\infty} \leq c \sqrt{\diam \paren{\gamma \bracket{0,t_0}}\osc\paren{\gamma,t-s,t_0}},\]
where
\[\osc\paren{\gamma,\delta,t_0}=\sup \set{\abs{\gamma\paren{s}-\gamma\paren{t}}: 0 \leq s<t \leq t_0; \abs{t-s}<\delta} \]
and
$g_s - g_t$ is considered as a function on $\half \backslash \gamma_t$.
\end{lemma}

The main tool in proving Lemma \ref{gDiamCurve} is the Beurling estimate, which we will not discuss here.  The proof also needs the following useful lemma.

\begin{lemma} \label{gDerivR}
Let $A \in \mathcal{Q}$ and $\rad\paren{A}=r$. Then for all $x>r$,
\[x \leq g_A\paren{x} \leq x+\frac{r^2}{x}\]
and for all $x<-r$,
\[x+\frac{r^2}{x} \leq g_A\paren{x} \leq x. \]
Furthermore, if $z \in \half \backslash A$, then
\[\abs{g_A\paren{z}-z} \leq 3\rad\paren{A}.\]
\end{lemma}

We also need the following lemma.

\begin{lemma}\label{conRDeriv}
Suppose $u : \left [0, t_0 \right) \rightarrow \C$ is a continuous function such that the right derivative
\[u_+'\paren{t}=\lim_{\epsilon \rightarrow 0^+} \frac{u\paren{t+\epsilon}-u\paren{t}}{\epsilon} \]
exists for all $t \in [0,t_0)$ and is a continuous function.  Then $u'\paren{t}=u_+'\paren{t}$ for all $t \in \paren{0,t_0}$.
\end{lemma}

We state and discuss the proof of the chordal Loewner equation.

\begin{theorem}\label{chordalLE}
For all $z \in \half \backslash \gamma_{t_0}$ and $0 \leq t \leq t_0$, $g_t\paren{z}$ is a solution to the initial value problem
\begin{equation}\label{chordalLE1}
\dot{g}_t\paren{z}=\frac{\dot{a}\paren{t}}{g_t\paren{z}-U_t}, ~~~g_0\paren{z}=z.
\end{equation}
\end{theorem}

An important observation that we will return to when we prove the analogous result for multiply connected domains is that the imaginary part of $\frac{\dot{a}\paren{t}}{g_t\paren{z}-U_t}$ is equal to $-\pi \dot{a}\paren{t} H_{\half}\paren{z,U_t}$.  Using the Schwarz reflection principle, we can show that \eqref{chordalLE1} holds for $x \in \R$ as well.  The idea behind the proof of Theorem \ref{chordalLE} is to apply Proposition \ref{SCLEat0} to
\[g_{s+\epsilon}\paren{z}-g_s\paren{z}=g_{s,s+\epsilon}\paren{g_s\paren{z}}-g_s\paren{z} \]
to conclude that $g_s\paren{z}$ has a right derivative equal to $\frac{\dot{a}\paren{s}}{g_s\paren{z}-U_s}$ and then to apply Proposition \ref{UtCont} and Lemma \ref{conRDeriv}.

If $t \mapsto U_t$ is a continuous function and $t \mapsto a\paren{t}$ is an increasing $C^1$ function, then a converse to Theorem \ref{chordalLE} holds.  More precisely, for each $t\geq 0$ it is possible to find a $K_t \in \mathcal{Q}$ and conformal map $g_t:K_t \rightarrow \half$ such that
\[\dot{g}_t\paren{z}=\frac{\dot{a}\paren{t}}{z-U_t},~~~g_0\paren{z}=z.\]
A family of maps $g_t$ arising in this way is called a \emph{generalized Loewner Chain} with \emph{driving function} $U_t$.  While we will not make use of this important fact, we will need some facts about generalized Loewner chains.

It can be checked that if $g_t$ is a generalized Loewner chain, then for all $z \in \half \backslash K_t$,
\begin{equation}\label{gderiv}
g_t'\paren{z}=\exp \set{-\int_0^t \frac{\dot{a}\paren{s}ds}{\paren{g_s\paren{z}-U_s}^2}}.
\end{equation}
Using this, we can derive an estimate for the spatial derivative of $g_t$ restricted to the real line.

\begin{lemma}\label{gDerivOnR}
Let $g_t$ be a generalized Loewner chain and let $r_t>0$ be such that $\gamma_t \subset B_{r_t}\paren{\gamma\paren{0}}$.  Then there is a $0<c \leq 9$ such that
\[1-\frac{c r_t^2}{x^2} \leq g_t'\paren{x}\leq 1, \]
for $x \in \R$ with $\abs{x}>3r_t.$ 
\end{lemma}
\begin{proof}
See \cite{DreThesis}.
\end{proof}

We will need the following well-known result, which provides bounds for the derivatives of harmonic functions.

\begin{lemma}\label{HarDerBd}
Let $u$ be a real-valued harmonic function on a domain $D \subset \C$.  For each $k \in \N$, there is a $c\paren{k}>0$ such that if $j\leq k$ is a non-negative integer, then
\[\abs{\partial_x^j \partial_y^{k-j}u\paren{z}} \leq c\paren{k} \dist \paren{z,\partial D}^{-k} \left\|u \right \|_{\infty}.  \]
\end{lemma}

We conclude the section by looking at the effect of applying a locally real conformal transformation on the time derivative of the half-plane capacity of a continuously increasing $K_t \in \mathcal{Q}$.  This result is stated in \cite{MR2129588} and a detailed proof can be found in \cite{DreThesis}.

\begin{proposition}\label{LRConHCap}
Let $F:B_R\paren{0} \rightarrow \C$ be a conformal map that maps reals to reals and $\gamma$ be as in Theorem \ref{chordalLE}.  Then
\begin{equation}\label{LRConHCap1}
\lim_{t \rightarrow 0^+} \frac{\hcap \paren{F\paren{\gamma_t}}}{t}=F'\paren{0}^2 \dot{a}\paren{0}.
\end{equation}
In particular, the limit in \eqref{LRConHCap1} exists if and only if $\dot{a}\paren{0}$ exists.
\end{proposition}

\section{Excursion Reflected Brownian Motion} \label{sectERBM} 

\subsection{Definition}
\label{ChapERBMDef}
We start this section by giving a precise definition of excursion reflected Brownian motion in $D \in \mathcal{Y}$.  Later we will see that for any $D \in \mathcal{Y}$, there is a unique process satisfying the conditions of our definition.

The Jordan curve theorem says that any Jordan curve $\eta$ separates $\C$ into exactly two connected components.  We will call the bounded connected component the \emph{interior} of $\eta$ and the unbounded connected component the \emph{exterior} of $\eta$. If $A \subset \C$ is in the interior of $\eta$, we will say $\eta$ \emph{surrounds} $A$.

\begin{definition}
\label{characterizationERBM}
Let $E=D\cup \set{A_1, \ldots, A_n}$ be equipped with the quotient topology and let $E_{\partial}=E \cup \set{A_0}$ be the one-point compactification of $E$. A stochastic process $B^{ER}_{D}$ with state space $E_{\partial}$ is called an \emph{excursion reflected Brownian motion} (ERBM) if it satisfies the following properties.

\begin{enumerate}
\item $B_{D}^{ER}$ has the strong Markov property.

\item If we start the process at $z \in D$ and let
\[T=\inf \set{t:  B_{D}^{ER}\paren{t} \in \partial D},\]
then for $0\leq t \leq T$, $B_{D}^{ER}\paren{t}$ is a Brownian motion in $D$ killed at $\partial D$.

\item Let $\eta_1, \ldots, \eta_n$ be pairwise disjoint smooth Jordan curves in $D$ such that $\eta_i$ surrounds $A_i$ and does not surround $A_j$ for $j\neq i$.  If
\[\sigma=\inf \set{t:B_{D}^{ER}\paren{t} \in \eta_i},\]
then $B_{D}^{ER}\paren{\sigma}$ has the distribution of $\overline{\mathcal{E}}_{U_i}\paren{A_i,\cdot}$, where $U_i$ is the region bounded by $\partial A_i$ and $\eta_i$.

\item  $B_{D}^{ER}$ is conformally invariant (this will be made more precise in Proposition \ref{ERBMconformalinvariance}) and the radial part of $B_{\C \backslash \D}^{ER}$ has the same distribution as the radial part of a reflected Brownian motion in $\C \backslash \D$.
\end{enumerate}
\end{definition}

We will often refer to ERBM in $D$ or $E$ when we really mean the process with the enlarged state space $E_{\partial}$.

\subsection{Excursion Reflected Brownian Motion in $\C \backslash \D$}
\label{ChapERBMCD}
The first step in constructing ERBM is to construct it in $E=\C \backslash \D \cup \set{\overline{\D}}$. We will mimic the construction of Walsh's Brownian motion given in \cite{MR1022917}.  The idea of the construction is that if a process exists that satisfies Definition \ref{characterizationERBM}, we can determine what its transition semigroup must be.  Once we know what its transition semigroup must be, we use general theory to show that there actually is a process with that semigroup.  Finally, once we have the process, we check that it actually satisfies Definition \ref{characterizationERBM}.  For the remainder of this section, let $A_0=\overline{\D}$.

Since the radial part of ERBM has the same distribution as the radial part of reflected Brownian motion, to describe the semigroup for ERBM, we need the Feller-Dynkin semigroup for reflected Brownian motion.  There is much in the literature on reflected Brownian motion and it is possible to define it in very general domains.  However, in $\C \backslash \D$ it is possible to give a simple construction.   Let $B_1$ and $B_2$ be independent one-dimensional Brownian motions and define reflected Brownian motion in $\half$ to be the process $B_1+i\abs{B_2}$.  We can then define reflected Brownian motion in $\C \backslash \D$ to be the image of reflected Brownian motion in $\half$ under the map $z \mapsto e^{-i z}$ with the appropriate time change (taking this approach, it is still necessary to check that the resulting process is Feller-Dynkin).
 
\begin{proposition}\label{semigroupERBMThm}
Let $T_t^+$ be the Feller-Dynkin semigroup for reflected Brownian motion in $\C\backslash \D$ and $T_t^0$ be the Feller-Dynkin semigroup for reflected Brownian motion killed when it hits $\D$.  If $f\in C_0\paren{E}$, define an operator $P_t$ by
\begin{equation}
P_t f\paren{r,\theta}=T^+_t \overline{f}\paren{r,\theta}+T_t^0 \paren{f-\overline{f}}\paren{r,\theta},
\label{semigroupERBM}
\end{equation}
where $\displaystyle \overline{f}\paren{r,\theta}=\frac{1}{2\pi} \int_0^{2\pi} f\paren{r,\theta}d\theta.$  If there is a Feller-Dynkin process $B_{\C \backslash \D}^{ER}$ in $\C \backslash \D$ satisfying Definition \ref{characterizationERBM}, its semigroup is $P_t.$
\end{proposition}

\begin{proof}
Let $B_t$ and $R_t$ be respectively Brownian motion and reflected Brownian motion in $\C \backslash \D$ and let $\mathbf{E}$, $\mathbf{E}_1$, and $\mathbf{E}_2$ be the expectations with respect to the probability measures induced by $B_{\C \backslash \D}^{ER}$, $B_t$, and $R_t$ respectively.  Let $\Omega$ be the underlying probability space $B_{\C \backslash \D}^{ER}$, $B_t$, and $R_t$ are defined on and let $\tau$ be the first time $B_{\C \backslash \D}^{ER}$ hits $\D$.  With respect to the appropriate filtration, $\tau$ is a stopping time \cite{MR1796539}.  Finally, let
\[A_t=\set{\omega \in \Omega: \tau\leq t}.\]
By abuse of a notation, we will also denote the set of $\omega$ such that $R_s$ has hit $\D$ by time $t$ and the set of $\omega$ such that $B_s$ has hit $\D$ by time $t$ by $A_t$.  Using (2) of Definition \ref{characterizationERBM}, we have that $B_{\C \backslash \D}^{ER}$ has the distribution of a Brownian motion up until time $\tau$.  Using (3) and (4) of Definition \ref{characterizationERBM}, we have that on the set $A_t$, the angular part of $B_{\C \backslash \D}^{ER}$ is uniformly distributed and the radial part is that of a reflected Brownian motion.  Combining these facts, we have that if $f\in C_0\paren{E}$, then
\begin{align*}\displaystyle
P_t f\paren{x}&= \ev{x}{f\paren{B_{\C \backslash \D}^{ER}\paren{t}}}\\
&= \ev{x}{\mathbf{1}_{A_t}f\paren{B_{\C \backslash \D}^{ER}\paren{t}}}+\ev{x}{\mathbf{1}_{A_t^c}f\paren{B_{\C \backslash \D}^{ER}\paren{t}}} \\
&= \mathbf{E}_2^x \bracket{\overline{f}\paren{R_t}}+\mathbf{E}_1^x \bracket{f\paren{B_t}}-\mathbf{E}_2^x \bracket{\mathbf{1}_{A_t^c}\overline{f}\paren{R_t}}-\mathbf{E}_1^x \bracket{\mathbf{1}_{A_t}f\paren{B_t}}\\
&= T_t^+ \overline{f}\paren{x}+T_t^0 f\paren{x}-\mathbf{E}_1^x \bracket{\mathbf{1}_{A_t^c}\overline{f}\paren{B_t}}-\mathbf{E}_1^x \bracket{\mathbf{1}_{A_t}\overline{f}\paren{B_t}}\\
&= T^+_t \overline{f}\paren{r,\theta}+T_t^0 \paren{f-\overline{f}}\paren{r,\theta}.
\end{align*}
In the second to last equality, we use the fact that by convention we define $\overline{f}\paren{\D}$ to be $f\paren{\D}$.
\end{proof}

Next we show that $P_t$, as defined in Proposition \ref{semigroupERBM}, is a Feller-Dynkin semigroup.  We will continue to use the setup given in the beginning of the proof of Proposition \ref{semigroupERBM}.
\begin{proposition}
$P_t$ is a Feller-Dynkin semigroup on $C_0\paren{E}$.  That is,
\begin{enumerate}
\item $P_t: C_0\paren{E} \rightarrow C_0\paren{E}$
\item If $f \in C_0\paren{E}$ and $0 \leq f \leq 1,$ then $0 \leq P_t f \leq 1$.
\item $P_0$ is the identity on $C_0\paren{E}$ and $P_t P_s=P_{t+s}$
\item $\displaystyle \lim_{t \rightarrow 0} \left\| P_t f - f\right\|_{\infty} =0$ for all $f \in C_0\paren{E}$.
\end{enumerate}
\end{proposition}

\begin{proof}
If $f \in C_0\paren{E}$, then $\overline{f}$ is also in $C_0\paren{E}$.  We have
\begin{align*}
\abs{P_t f\paren{r,\theta}-P_t f\paren{r',\theta'}} \leq& \abs{T_t^+ \overline{f}\paren{r,\theta} - T_t^+ \overline{f}\paren{r',\theta'}}\\
&  +\abs{T_t^0f\paren{r,\theta}-T_t^0f\paren{r',\theta'}}\\
&  + \abs{T_t^0\overline{f}\paren{r,\theta}-T_t^0 \overline{f}\paren{r',\theta'}}.
\end{align*}
The fact that $P_t f$ is continuous follows from the fact that $T_t^0 f$, $T_t^0 \overline{f}$, and $T_t^+ \overline{f}$ are all continuous.  Since both $f$ and $\overline{f}$ vanish at infinity, so does $P_t f$.  This proves (i).

From the proof of Proposition \ref{semigroupERBM} we have
\begin{align*}
P_t f\paren{x}&=\mathbf{E}_2^x \bracket{\mathbf{1}_{A_t}\overline{f}\paren{R_t}}+\mathbf{E}_1^x \bracket{\mathbf{1}_{A_t^c}f\paren{B_t}}.
\end{align*}
If $0 \leq f \leq 1$, then
\begin{align*}
0 &\leq  \mathbf{E}_2^x \bracket{\mathbf{1}_{A_t}\overline{f}\paren{R_t}}+\mathbf{E}_1^x \bracket{\mathbf{1}_{A_t^c}f\paren{B_t}} \\
&\leq \mathbf{E}_2^x \bracket{\mathbf{1}_{A_t}}+\mathbf{E}_1^x \bracket{\mathbf{1}_{A_t^c}}\\
&= 1,
\end{align*}
from which (ii) follows.

It is clear that $P_0$ is the identity on $C_0\paren{E}$. Observe that
\begin{align*}
\overline{P_s f}\paren{r,\theta}&=\overline{T_s^+ \overline{f}}\paren{r,\theta} + \overline{T_s^0 f}\paren{r,\theta} -\overline{T_s^0\overline{f}}\paren{r,\theta}\\
&=T_s^+ \overline{f}\paren{r,\theta} + \int_0^{2\pi}T_s^0 f\paren{r,\theta}d\theta -T_s^0 \overline{f}\paren{r,\theta}\\
&=T_s^+ \overline{f}\paren{r,\theta} + T_s^0 \overline{f}\paren{r,\theta}-T_s^0 \overline{f}\paren{r,\theta}\\ 
&= T_s^+ \overline{f}\paren{r,\theta}
\end{align*}
and thus,
\begin{equation*}
P_s f\paren{r,\theta}- \overline{P_s f}\paren{r,\theta}=T_s^0\paren{f-\overline{f}}\paren{r,\theta}.
\end{equation*}
Using these two facts and the fact that (iii) holds for $T_t^+$ and $T_t^0$, we have
\begin{align*}
P_t P_s f\paren{r,\theta}&= T_t^+ \overline{P_s f}\paren{r,\theta}+T_t^0\paren{P_s f\paren{r,\theta}- \overline{P_s f}\paren{r,\theta}}\\
&= T_t^+ T_s^+ \overline{f}\paren{r,\theta}+T_t^0 T_s^0 \paren{f-\overline{f}}\paren{r,\theta}\\
&=T_{t+s}^+ \overline{f}\paren{r,\theta}+T_{t+s}^0\paren{f-\overline{f}}\paren{r,\theta}\\
&=P_{t+s} f\paren{r,\theta},
\end{align*}
from which (iii) follows.

Since (i)-(iii) hold, by (say) Lemma III.6.7 of \cite{MR1796539}, to prove (iv) it is enough to show that for all $f \in C_0\paren{E}$ and $z \in E$ we have
\[\displaystyle \lim_{t \rightarrow 0}P_t f\paren{z}=f\paren{z}.\]
Since $T_t^+$ and $T_t^0$ satisfy (i)-(iv), we have
\begin{align*}
\displaystyle \lim_{t \rightarrow 0} P_t f\paren{z}&= \lim_{t \rightarrow 0} T_t^+ \overline{f}\paren{z}+\lim_{t \rightarrow 0} T_t^0 \paren{f-\overline{f}}\paren{z}\\
&=\overline{f}\paren{z}+\paren{f\paren{z}-\overline{f}\paren{z}}\\
&= f\paren{z},
\end{align*}
which proves (iv).
\end{proof}

Using (say) Theorem III.7.1 of \cite{MR1796539}, given any measure $\mu$ on $E$, we can define a unique Feller-Dynkin process
 \begin{equation} \label{ERBMdef}
B^{ER}_{\C\backslash \D}:=\paren{\Omega,\mathscr{F},\{\mathscr{F}_t:t\geq 0\},\{B^{ER}_{\C\backslash \D}(t):t\geq 0\},\mathbf{P}^{\mu}}
\end{equation}
with semigroup $P_t$.  Furthermore, the filtration $\mathscr{F}_t$ is independent of the measure $\mu$ and $X$ has the strong Markov property with respect to $\mathscr{F}_t$.  We denote the angular and radial parts of $B^{ER}_{\C\backslash \D}$ at time $t$ by $\theta_t$ and $R_t$ respectively.

Next we check that the process $B^{ER}_{\C\backslash \D}$ defined in \eqref{ERBMdef} satisfies Definition \ref{characterizationERBM}.
\begin{proposition}
$B^{ER}_{\C \backslash \D}$ has the distribution of a Brownian motion up until the first time it hits $\partial D$.
\label{ERBMlikeBM}
\end{proposition}

\begin{proof}
This follows immediately from \eqref{semigroupERBM}.
\end{proof}

\begin{proposition} \label{radialpartERBM}
$R_t$ has the same distribution as the radial part of a reflected Brownian motion in $\C \backslash \D$.
\end{proposition}

\begin{proof}
We mimic the proof of Lemma 2.2 in \cite{MR1022917}.  Let $g \in C_0\paren{\left[1,\infty \right)}$ and define $f \in C_0\paren{E}$ by $f\paren{r,\theta}=g\paren{r}.$  Observe that $\overline{f}=f$ and $f\paren{X_t}=g\paren{R_t}$.  If $S$ is any $\mathscr{F}_t$-stopping time, then
\begin{align*}
\ev{\mu}{g\paren{R_{S+t}}|\mathscr{F}_S}&=\ev{\mu}{f\paren{X_{S+t}}|\mathscr{F}_S}\\
&= P_t f\paren{X_S}\\
&= T_t^+ \overline{f}\paren{R_S,\theta_S}+T_t^0\paren{f-\overline{f}}\paren{R_S,\theta_S}\\
&= T_t^+ f \paren{R_S,\theta_S}\\
&= R_t^+ g\paren{R_S},
\end{align*}
where $R_t^+$ is the semi-group for the radial part of reflected Brownian motion in $\C \backslash \D$. The result follows.
\end{proof}

\begin{proposition} \label{excursiondistERBM}
Let $\eta$ be a smooth Jordan curve surrounding $\D$, $U$ be the region bounded by $\eta$ and $\partial \D$, and $\tau$ be the first time $B^{ER}_{\C \backslash \D}$ hits $\eta$. If $V$ is a smooth arc in $\eta$, then
\begin{equation*}
\alpha:= \prob{\D}{B^{ER}_{\C \backslash \D}(\tau) \in V}=\overline{\mathcal{E}}_U\paren{\D,V}.
\end{equation*}
\end{proposition}

\begin{proof}
Let $C_{\epsilon}$ be the circle of radius $\epsilon$ centered at the origin.  Since it is clear from \eqref{semigroupERBM} that $B_{\C\backslash \D}^{ER}$ is rotationally invariant, the result follows in the case that $\eta=C_{\epsilon}$.

Let $p\paren{z}$ be the probability that a Brownian motion started at $z$ exits $U$ on $\eta$.  For small enough $\epsilon$, $C_{\epsilon}$ is in the interior of $\eta$.  For such an $\epsilon$, using the strong Markov property for ERBM and Proposition \ref{ERBMlikeBM}, we see that
\begin{align*}
2\pi \alpha &= \int_{C_{\epsilon}}\bracket{\int_V H_U\paren{z,w} \abs{dw}+\paren{ 1-p\paren{z}}\alpha}\abs{dz} \\
&= 2 \pi \alpha+ \int_{C_{\epsilon}} \bracket{\int_V H_U\paren{z,w} \abs{dw}-p\paren{z}\alpha}\abs{dz}\\
&= 2 \pi \alpha +\int_{C_{\epsilon}} \bracket{\int_V H_U\paren{z,w} \abs{dw}-\alpha\int_{\eta} H_U\paren{z,w}\abs{dw}}\abs{dz}.
\end{align*}
As a result, for small enough $\epsilon$, we have
\[\displaystyle\alpha=\frac{\int_{C_{\epsilon}}\int_V H_U\paren{z,w} \abs{dw}\abs{dz}}{\int_{C_{\epsilon}}\int_{\eta} H_U\paren{z,w}\abs{dw}\abs{dz}}.\]
Since the derivative of $H_U\paren{\cdot,w}$ is bounded in a neighborhood of $\D$, using the dominated convergence theorem, we see that
\begin{align*}
\alpha &=\lim_{\epsilon \rightarrow 0} \frac{\int_{C_{\epsilon}}\int_V \frac{H_U\paren{z,w}}{\epsilon} \abs{dw}\abs{dz}}{\int_{C_{\epsilon}}\int_{\eta} \frac{H_U\paren{z,w} }{\epsilon}\abs{dw}\abs{dz}}\\
&= \frac{\int_{C_{\epsilon}}\int_V H_{\partial U}\paren{z,w} \abs{dw}\abs{dz}}{\int_{C_{\epsilon}}\int_{\eta} H_{\partial U}\paren{z,w} \abs{dw}\abs{dz}}\\
&= \frac{\mathcal{E}_U\paren{\D,V}}{\mathcal{E}_U\paren{\D,\eta}}.
\end{align*}
\end{proof}

\begin{proposition}
There is a unique process Feller-Dynkin process with state space $E=\C\backslash D\cup \set{\D}$ satisfying Definition \ref{characterizationERBM}.  Furthermore, this process can be defined so as to have continuous sample paths in the topology of $E$.
\end{proposition}
\begin{proof}
Propositions \ref{ERBMlikeBM}, \ref{radialpartERBM}, and \ref{excursiondistERBM} combine to show that the process defined in \eqref{ERBMdef} satisfies Definition \ref{characterizationERBM}.  The uniqueness statement follows from Proposition \ref{semigroupERBM}.  By construction, $B_{\C \backslash \D}^{ER}$ is an $R$-process.  Combining this with Propositions \ref{ERBMlikeBM} and \ref{radialpartERBM}, and the fact that Brownian motion and reflected Brownian motion have continuous sample paths, it is easy to check that the sample paths of $B_{\C \backslash \D}^{ER}$ are continuous.
\end{proof}

\subsection{Excursion Reflected Brownian Motion in Conformal Annuli}
\label{chapERBMCA}
Let $A$ be any compact, connected subset of $\C$ larger than a single point and
\[f: \C \backslash \D \rightarrow \C \backslash A\]
be a conformal map sending $\infty$ to $\infty$.  It is a straightforward exercise to verify that $f$ is unique up to an initial rotation.  Let $\sigma_t$ be the $\mathscr{F}_t$ stopping time given by
\[\int_0^{\sigma_t} \abs{f'\paren{B_{\C\backslash \D}^{ER}\paren{s}}}^2 ds=t\]
and define
\[B_{\C\backslash A}^{ER}\paren{t}=f\paren{B_{\C \backslash D}^{ER}\paren{\sigma_t}}\]
and $\tilde{\mathscr{F}}_t=\mathscr{F}_{\sigma_t}.$
We define ERBM in $\C \backslash A$ to be the process
\begin{equation*}
B^{ER}_{\C \backslash A}:=\paren{\Omega,\mathscr{F},\set{\tilde{\mathscr{F}}_t: t\geq 0}, \set{B^{ER}_{\C\backslash A}},\set{\mathbf{P}^x}}.
\end{equation*}
Since $B_{\C \backslash D}^{ER}$ is rotationally invariant and $f$ is unique up to an initial rotation, it is clear that the distribution of $B_{\C\backslash A}^{ER}$ does not depend on $f$.  It is also not hard to check that the strong Markov property is preserved (see the discussion on pg. 277 of \cite{MR1796539}).  To ensure that $B_{\C\backslash A}^{ER}\paren{t}$ exists for all $t<\infty$, we need to verify that
\begin{equation}
\int_0^{\infty} \abs{f'\paren{B_{\C\backslash \D}^{ER}\paren{s}}}^2 ds=\infty~~\rm{a.s.}.
\label{ERBMalltime}
\end{equation}
In order for $B_{\C\backslash A}^{ER}\paren{t}$ not to have a limit as $t \rightarrow \infty$, we need to verify that for all $t<\infty$,
\begin{equation}
\int_0^{t} \abs{f'\paren{B_{\C\backslash \D}^{ER}\paren{s}}}^2 ds<\infty~~\rm{a.s.}.
\label{ERBMnolimit}
\end{equation}
We temporarily put these considerations aside.

\begin{proposition}
Suppose $f: \C \backslash \D \rightarrow D_1$ and $g : D_1 \rightarrow D_2$ are conformal maps.  Then the process
\[B^{ER}_{D_2}\paren{t}=B^{ER}_{D_1}\paren{\sigma_t},\]
where
\[\int_0^{\sigma_t} \abs{g'\paren{B^{ER}_{D_1}\paren{s}}}^2 ds=t\]
is an ERBM in $D_2$. \label{ERBMconformalinvariance}
\end{proposition}

\begin{proof}
Let $\sigma_r$ satisfy
\[\int_0^{\sigma_r} \abs{g'\paren{f\paren{B^{ER}_{\C \backslash \D}\paren{s}}}f'\paren{B^{ER}_{\C \backslash \D}\paren{s}}}^2 ds=r\]
and define a map $T:\bracket{0,\sigma_r} \rightarrow \left[0, \infty \right)$ by
\begin{equation}
t \mapsto \int_0^t \abs{f'\paren{B^{ER}_{\C \backslash \D}\paren{s}}}^2 ds.
\end{equation}
It is straightforward to verify that $T$ is a bijection (we use \eqref{ERBMnolimit} here) onto $\bracket{0,T\paren{\sigma_r}}$ with derivative $\abs{f'\paren{B^{ER}_{\C \backslash \D}\paren{s}}}^2$.  Using the change of variables formula, we have
\begin{align*}
r&=\int_0^{\sigma_r} \abs{g'\paren{f\paren{B^{ER}_{\C \backslash \D}\paren{s}}}f'\paren{B^{ER}_{\C \backslash \D}\paren{s}}}^2 ds\\
&=\int_0^{\sigma_r} \abs{g'\paren{B^{ER}_{D_1}\paren{T\paren{s}}}f'\paren{B^{ER}_{\C \backslash \D}\paren{s}}}^2 ds\\
&=\int_0^{T\paren{\sigma_r}} \abs{g'\paren{B^{ER}_{D_1}\paren{s}}}^2 ds.
\end{align*}
As a result, $B^{ER}_{D_2}\paren{r}=g\paren{B^{ER}_{D_1}\paren{T\paren{\sigma_r}}}=g\paren{f\paren{B^{ER}_{\C \backslash \D}\paren{\sigma_r}}}$ and thus, the process in $D_2$ defined by $g$ is the same as the process defined by $g\circ f$.  The result follows.
\end{proof}

\begin{proposition}
The process $B_{\C \backslash A}^{ER}$ satisfies Definition \ref{characterizationERBM}.
\end{proposition}
\begin{proof}
The 1st and 4th property have already been discussed.  The 2nd property follows from the conformal invariance of Brownian motion and the 3rd property follows from the conformal invariance of excursion measure.
\end{proof}

If $K$ is a closed subset of $\C \backslash A$ it makes sense to discuss ERBM in $\C \backslash A$ killed at $K$.  Most often we will do this when $K$ is a simple, closed curve $\eta$ surrounding $A$ and refer to the corresponding process as ERBM in $D$, where $D$ is the region bounded by $\eta$ and $\partial A$.

\subsection{Excursion Reflected Brownian Motion in Finitely Connected Domains} \label{ERBMGeneralDef}
\label{chapERBMFC}
Let $D\in \mathcal{Y}_n$ and $\eta_i$, for $1\leq i \leq n$, be as in Definition \ref{characterizationERBM}.  Denote the domain bounded by $\eta_i$ and $\partial A_i$ by $U_i$. We will now define a process $B^{ER}_D$ in $D$ satisfying Definition \ref{characterizationERBM}.  Intuitively, we define $B^{ER}_D\paren{t}$ pathwise to be a Brownian motion up until the first time it hits an $A_i$, then let it be an ERBM in $U_i$ until it hits $\eta_i$, then let it be a Brownian motion until it hits another $A_i$ and so on.  Adding rigor to this intuition is not hard, but is notationally cumbersome. For $i=1,2, \ldots, n$ and $j=1,2, \ldots$ let $B_{U_i}^{\paren{ j}}$ be an ERBM in $U_i$ started at $A_i$ and $B_j$ be a Brownian motion in $\C$ started at the origin.  We can construct these processes on the same probability space $\Omega$ so that they are all independent.  Let $z \in E= D\cup \set{A_0, \ldots ,A_n}$ and define
\begin{equation*}
B^{ER}_D\paren{ t}=\begin{cases}
z & \text{if } t=0\\
A_0 & \text{if } t > \tau \\
B^{ER}_D\paren{ \sigma_j} + B_i\paren{ t-\sigma_i} & \text{if } \sigma_j < t \leq \tau_j \\
B^{\paren{ j}}_{U_i}\paren{ t-\tau_{j}} \text{ where } B^{ER}_D\paren{ \tau_j}=A_i & \text{if } \tau_j < t \leq \sigma_{j+1}
\end{cases}
\end{equation*}
where
\begin{align*}
\tau &=\inf \set{t: B^{ER}_D\paren{ t} \in A_0}, \\
\sigma_1 &=0, \\
\sigma_j &= \inf \set{t \geq \tau_{j-1}: B_D^{ER}\paren{ t} \in \eta_i} \text{ for } j\geq 2, \\
\tau_j &=  \inf \set{t \geq \sigma_j: B_D^{ER}\paren{ t} \in A_i \text{ for some } i}.
\end{align*}
It is not hard to check that the distribution of $B_D^{ER}\paren{t}$ does not depend on the choice of $\eta_i$ and that $B_D^{ER}\paren{t}$ satisfies Definition \ref{characterizationERBM}.


\subsection{A Markov Chain Associated with ERBM} \label{MarkovChainERBM}
\label{chapMC}
Let $h_j$ be the unique bounded harmonic function on $D$ that is equal to $1$ on $\partial A_j$ and $0$ on $\partial A_i$ for $i \neq j$ (note that $h_j\paren{z}$ is the probability that a Brownian motion started at $z$ exits $D$ at $A_j$).  ERBM on $D$ induces a discrete time Markov chain $X$ with state space $\set{A_0, \ldots A_n}$ (see \cite{MR2247843} pg. 37).  The probability that the chain moves from $A_i$ to $A_j$ is equal to the probability that $A_j$ is the first boundary component of $D$ that $B_D^{ER}$ started at $A_i$ hits after the first time it hits $\eta_i$.  That is, the chain has transition probabilities $p_{00}=1$ and
\begin{equation*}
p_{i j}=\int_{\eta_i} h_j\paren{z} \frac{H_{\partial U_i}\paren{A_i,z}}{\mathcal{E}_{U_i}\paren{A_i, \eta_i}} \abs{dz},
\end{equation*}
for $i\neq 0$.  This Markov chain is not entirely satisfactory since it is highly dependent on the particular choice of $\eta_1, \ldots ,\eta_n$.  However, this chain does induce another chain $Y$ with transition probabilities
\begin{equation*}
q_{i j}=\frac{p_{i j}}{1-p_{ii}},
\end{equation*}
for $i \neq j$.  $Y$ is obtained from $X$ by erasing all of the loops and it is not hard to see that its transition probabilities are independent of the choice of $\eta_1, \ldots ,\eta_n$.  Since $q_{j0}>0$ for all $1 \leq j \leq n$, the eigenvalues of the transition matrix, $\mathbf{Q}$, for $Y$ restricted to $A_1, \ldots, A_n$ have absolute value strictly less than one and, using standard results from Markov chain theory, we have that the Green's matrix
\begin{equation}\label{GreenMatrixFormula}
\mathbf{I}+\mathbf{Q}+\mathbf{Q}^2+\cdots + \mathbf{Q}^n+ \cdots=\paren{\mathbf{I}-\mathbf{Q}}^{-1}.
\end{equation}
is well-defined.

\subsection{Excursion Reflected Harmonic Functions}
\begin{definition}\label{ERHarmonicDef}
A function
\[v:E \rightarrow \R\]
is called \emph{ER-harmonic} if it satisfies
\begin{enumerate}
\item $v$ is continuous on $E$ and is harmonic when restricted to $D$
\item For $1 \leq i \leq n$, if $\eta$ is a Jordan curve surrounding $A_i$, then
\begin{equation} \label{ERBMMeanvalue}
v\paren{A_i}=\int_{\eta} v\paren{z} \frac{H_{\partial U_i}\paren{A_i,z}}{\mathcal{E}_{U_i}\paren{A_i, \eta}} \abs{dz},
\end{equation}
where $U_i$ is the region bounded by $\eta$ and $\partial A_i.$
\end{enumerate}
\end{definition}

If it is clear what is meant, we will sometimes speak of the ER-harmonicity of a function with domain $D$ rather than $E$.  By an ER-harmonic function on $D-\set{z}$ or $D-\set{A_i}$ we mean a function that satisfies Definition \ref{ERHarmonicDef} except that (2) is not necessarily satisfied for curves surrounding $z$ and $A_i$ respectively.

The following is a useful criterion for a function to be ER-harmonic.

\begin{lemma}\label{ERHarmonicCondition}
Let $\eta$ and $\eta'$ be smooth Jordan curves surrounding $A_j$ and not surrounding $A_i$ for $i\neq j$. Then for any harmonic function $v$ on $D$ we have
\begin{enumerate}
\item
\begin{equation*}
\int_{\eta}\frac{d}{dn}~ v\paren{z}\abs{dz}=\int_{\eta'}\frac{d}{dn}~v\paren{z}\abs{dz},
\end{equation*}
where $n$ is the outward pointing normal
\item
\begin{equation*}
\int_{\eta} v\paren{z} H_{\partial U_j}\paren{A_j,z}\abs{dz}=v\paren{A_j}\mathcal{E}_{U_j}\paren{A_j,\eta}+ \int_{\eta} \frac{d}{dn}~ v\paren{z}\abs{dz},
\end{equation*}
where $U_j$ is the region bounded by $\eta$ and $\partial A_j$.  In particular, if $v$ is continuous on $E$, then $v$ is ER-harmonic if and only if for each $i$ there is an $\eta_i$ surrounding $A_i$ with
\[\int_{\eta_i} \frac{d}{dn}~ v\paren{z}\abs{dz}=0.\]
\end{enumerate}
\end{lemma}

\begin{proof}
See \cite{MR2247843} pg. 17.
\end{proof}

As with harmonic functions, if we specify suitable boundary conditions, there is a unique ER-harmonic function with these boundary conditions.  The key to proving this uniqueness is a maximal principle for ER-harmonic functions.

\begin{lemma}[Maximal principle for ER-harmonic functions] \label{ERBMMaximalprinciple}
Let $v: E \cup \partial A_0 \rightarrow \R$ be a bounded, continuous function that is ER-harmonic when restricted to $E$.  Then
\begin{enumerate}
\item  The maximum of value of $v$ is equal to the maximum value of $v$ restricted to $\partial A_0$.
\item  If there is a $z \in E$ such that $v$ attains its maximum at $z$, then $v$ is constant.
\end{enumerate}
\end{lemma}

\begin{proof}
It is clear that (2) implies (1), so it is enough to prove (2).  Let $z$ be a point where $v$ attains its maximum.  If $z \in D$, then by the strong maximal principle for harmonic functions \cite{MR1625845}, $v$ is constant.  If $z = A_i$, then using \eqref{ERBMMeanvalue} it is clear there is some $z' \in D$ where $v$ also attains its maximum and thus, $v$ is constant.
\end{proof}

\begin{proposition} \label{ERHarmonicFromERBM}
Suppose that  $\partial A_0$ has at least one regular point for Brownian motion and let $F: \partial A_0 \rightarrow \R$ be a bounded, measurable function.  Define
\[v: \overline{D} \rightarrow \R\]
by
\[v \paren{z}=\ev{z}{F\paren{B^{ER}_D\paren{\tau_D}}},\]
where $\tau_D$ is the first time an ERBM hits $A_0$.  Then $v$ is a bounded $ER$-harmonic function on $D$ that is continuous at all regular points of $\partial A_0$ at which $F$ is continuous.  Furthermore, if every point of $\partial A_0$ is regular and $F$ is continuous, then $v$ is the unique ER-harmonic function that is equal to $F$ on $\partial A_0$.
\end{proposition}

\begin{proof}
It is clear from the fact that $F$ is bounded that $v$ is also bounded.  The proof that $v$ is harmonic and continuous at the regular points of $A_0$ at which $F$ is continuous is similar to the proof of the corresponding result for Brownian motion (see \cite{MR2604525}).  The fact that \eqref{ERBMMeanvalue} holds follows from the strong Markov property for ERBM and (3) of Definition \ref{characterizationERBM}.  The uniqueness statement follows from a straightforward application of Lemma \ref{ERBMMaximalprinciple}.
\end{proof}

\section{The Poisson Kernel for ERBM}\label{sectPK}
\subsection{Definition and Basic Properties}
Throughout this chapter let $D \in \mathcal{Y}_n$ be such that $A_0 \neq \emptyset$ and let
\begin{equation*}
\tau_D=\inf\set{t\in \R^+ : B_D^{ER}\paren{t} \in \partial A_0 }.
\end{equation*}
The distribution of $B_D^{ER}\paren{\tau_D}$ defines a measure $\hm_D^{ER}\paren{z, \cdot}$ on $\partial A_0$ (with the $\sigma$-algebra generated by Borel subsets of $\partial A_0$) that we call \emph{ER-harmonic measure in $D$ from $z$}.  Using the analogous result for harmonic measure and the construction of ERBM, it is easy to check that if $\partial D$ is locally analytic at $w$, then $\hm_D^{ER}\paren{z, \cdot}$ is absolutely continuous with respect to arc length in a neighborhood of $w$.  The density of $\hm_D^{ER}\paren{z, \cdot}$ at $w$ with respect to arc length is called the \emph{Poisson kernel} for ERBM and is denoted $H_D^{ER}\paren{z,w}$.

If $\gamma:\paren{-\delta,\delta}\rightarrow \partial A_0,$ $\gamma\paren{0}=w$ is an analytic curve, then we can explicitly define a version of $H_D^{ER}\paren{z,w}$ by 
\begin{equation}\label{explicitPoisson}
\displaystyle H_D^{ER}\paren{z,w}=\lim_{\epsilon \rightarrow 0} \frac{\hm_D^{ER}\paren{z,\gamma \paren{ -\epsilon,\epsilon}} } {\int_{-\epsilon}^{\epsilon}\abs{\gamma'\paren{x}}~dx}.
\end{equation}
It is clear that this definition is independent of $\gamma$.  In what follows, when we refer to $H_D^{ER}\paren{z,w}$, we will mean the version given by \eqref{explicitPoisson}.

An analog of \eqref{PKCI} holds for $H_D^{ER}\paren{z,w}$.

\begin{proposition} \label{ERPKCI}
If $f: D \rightarrow D'$ is a conformal transformation such that $\partial D$ is locally analytic at $w$ and $\partial D'$ is locally analytic at $f\paren{w}$, then
\begin{equation*}
H_{D'}^{ER}\paren{f \paren{z},f \paren{w}}=\abs{f'\paren{w}}^{-1}H_D^{ER}\paren{z,w}.
\end{equation*}
\end{proposition}
\begin{proof}
Since ERBM is conformally invariant, $\hm_D^{ER}\paren{z, \cdot}$ is conformally invariant.  Combining this with the change of variables formula, the result follows.
\end{proof}

Recall that $h_i\paren{z}$ is the unique bounded harmonic function on $D$ that is $1$ on $\partial A_i$ and $0$ on $\partial A_j$ for $j \neq i$.  If $V$ is a Borel subset of $\partial A_0$, then using the strong Markov property for ERBM, we see that
\begin{equation*} 
\hm_D^{ER}\paren{z,V}=\hm_D\paren{z,V}+\sum_{i=1}^n h_i\paren{z}\hm_D^{ER}\paren{A_i,V}.
\end{equation*}
Combining this with \eqref{explicitPoisson}, we see that
\begin{equation}\label{ERBMPKD}
H_D^{ER}\paren{z,w}=H_D\paren{z,w}+\sum_{i=1}^n h_i\paren{z} H_D^{ER}\paren{A_i,w}.
\end{equation}
Using \eqref{ERBMPKD}, it is sometimes possible to explicitly compute $H_D^{ER}\paren{z,w}$.  We do this calculation in the case that $D$ is an annulus.

\begin{proposition} \label{AnnulusERBMPK}
If $r>1$ and $A_{e^{-r},1}$ is the annulus with $\partial A_0=\partial \D$ and $\partial A_1=\partial \paren{e^{-r}\D}$, then
\[ H_{A_{e^{-r},1}}^{ER}\paren{e^{i\paren{x+iy}},1}=\frac{-\log\abs{z}}{2\pi r}+\sum_{k \in \Z} \frac{\sin\paren{\frac{\pi y}{r}}}{2r\bracket{\cosh \paren{\frac{\pi \paren{x+2\pi k}}{r}}-\cos\paren{\frac{\pi y}{r}}}}.\]
\end{proposition}

\begin{proof}
See \cite{DreThesis}.
\end{proof}

\subsection{Some Poisson Kernel Estimates}

In this section, we gather some estimates for the Poisson kernel for ERBM that we will need in Chapter \ref{sectLE}.

It is possible to compute $H_D^{ER}\paren{A_i,\cdot}$ in terms of the boundary Poisson kernel and excursion measure.  To do this, we will need an analog of \eqref{twoDomainPK} for $H_D^{ER}\paren{z,\cdot}$.  Namely, if $D_2 \subset D_1$ are domains in $\mathcal{Y}$ such that $\partial D_2$ and $\partial D_1$ agree in a neighborhood of $x$, the strong Markov property for ERBM gives
\begin{equation}\label{twoDomainERPK}
H^{ER}_{D_2}\paren{w,x}=H^{ER}_{D_1}\paren{w,x}-\ev {w}{H^{ER}_{D_1}\paren{B_{\tau_{D_2}},x}}.
\end{equation}  

Let  $\tau^i$ be the first time an ERBM hits a boundary component of $D$ other than $\partial A_i$.  The distribution of $B_D^{ER}\paren{\tau^i}$ defines a measure on $\partial D$.  The next lemma computes the density of this measure restricted to $\partial A_0$.  In Lemma \ref{ERPKOneTime} and Proposition \ref{ERPKfromAI}, we suppose that $D$ has locally analytic boundary.

\begin{lemma}\label{ERPKOneTime}
If $1\leq i \leq n$, then
\[\displaystyle T_i\paren{w}:=\frac{H_{\partial D}\paren{A_i,w}}{\sum_{j \neq i} \mathcal{E}_D\paren{A_i,A_j}}\]
is a density for the distribution of $B_D^{ER}\paren{\tau^i}$ restricted to $\partial A_0$.
\end{lemma}

\begin{proof}
Let $\eta_i$ be a smooth Jordan curve surrounding $A_i$ and not surrounding $A_j$ for $j \neq i$ and let $U_i$ be the region bounded by $\partial A_i$ and $\eta_i$.  Using (1), (2), and (3) of Definition \ref{characterizationERBM} and the definition of $p_{i i}$,  we see that
\[\tilde{T}_i\paren{w}:=\frac{\int_{\eta_i}  H_D\paren{z,w}H_{\partial U_i}\paren{A_i,z}\abs{dz}}{\paren{1-p_{ii}} \mathcal{E}_{U_i}\paren{A_i,\eta_i}},\]
is a density for the distribution of $B_D^{ER}\paren{\tau^i}$ restricted to $\partial A_0$.  The strong Markov property for ERBM implies
\begin{equation}\label{ERPKOneTime1}
\int_{\eta_i}  H_D\paren{z,w}H_{\partial U_i}\paren{A_i,z}\abs{dz}=H_{\partial D}\paren{A_i,w}.
\end{equation}
Using \eqref{ERPKOneTime1}, we have
\begin{align*}
{\paren{1-p_{ii}} \mathcal{E}_{U_i}\paren{A_i,\eta_i}}&= \int_{\eta_i} \paren{1-h_i \paren{z}} H_{\partial U_i}\paren{A_i,z}\abs{dz}\\
&= \int_{\eta_i} \paren{\sum_{j \neq i} h_j\paren{z}} H_{\partial U_i}\paren{A_i,z}\abs{dz}\\
&= \sum_{j \neq i} \int_{\eta_i} \int_{A_j} H_D\paren{z,w} H_{\partial U_i}\paren{A_i,z}\abs{dw}\abs{dz}\\
&= \sum_{j \neq i} \int_{A_j} \int_{\eta_i} H_D\paren{z,w} H_{\partial U_i}\paren{A_i,z}\abs{dz}\abs{dw}\\
&= \sum_{j \neq i} \int_{A_j} H_{\partial D}\paren{A_i,w}\abs{dw}\\
&= \sum_{j \neq i}\mathcal{E}_D\paren{A_i,A_j}
\end{align*}
The result follows.
\end{proof}

The following proposition follows immediately.
\begin{proposition} \label{ERPKfromAI}
Fix $w \in \partial A_0$ and let $\mathbf{H}$ be the $n \times 1$ vector with $i$th component equal to $H_D^{ER}\paren{A_i,w}$, $\mathbf{T}$ be the $n \times 1$ vector with $i$th component equal to $T_i\paren{w}$, and $\mathbf{Q}$ be as in Section \ref{MarkovChainERBM}.  Then
\[\mathbf{H}=\mathbf{T}+\mathbf{Q} \mathbf{T}+ \mathbf{Q}^2\mathbf{T} + \ldots = \paren{\mathbf{I}-\mathbf{Q}}^{-1} \mathbf{T}.\]
\end{proposition}

Next we prove the analog of Lemma \ref{PKStandardChordalRMHalfDisk} for $H_D^{ER}\paren{z,\cdot}$.

\begin{proposition} \label{ERPKEstimate}
Let $D \in \mathcal{Y}^*$ be a domain with piecewise analytic boundary and such that $\partial D$ and $\partial \half$ agree in a neighborhood of $0$. If $\abs{z}>2\epsilon$, then
\begin{equation}
H^{ER}_{D^{\epsilon}}\paren{z,\epsilon e^{i\theta}}=2 H^{ER}_{D}\paren{z,0}  \sin \theta \bracket{1+O\paren{\epsilon}},~~~\epsilon \rightarrow 0.
\end{equation}
Furthermore, for any $r>0$, $O\paren{\epsilon}$ is uniform over all $z\in \half$ such that $\abs{z}>r.$
\end{proposition}

\begin{proof}
For $1 \leq i \leq n$, let $h_i^{\epsilon}\paren{z}$ be the unique bounded harmonic function on $D^{\epsilon}$ that is $1$ on $\partial A_i$ and $0$ on the other boundary components of $D^{\epsilon}$ and let $q_{ij}^{\epsilon}$ be the probability that the Markov chain induced by $B_{D^{\epsilon}}^{ER}$ moves from $A_i$ to $A_j$. Let $T_i^{\epsilon}\paren{w}$ be the density introduced in Lemma \ref{ERPKOneTime} for $B_{D^{\epsilon}}^{ER}$, $\mathbf{T}^{\epsilon}$ be the $n \times 1$ vector with $i$th component $T_i^{\epsilon}\paren{\epsilon e^{i\theta}}$, $\mathbf{Q}^{\epsilon}$ be the matrix with $ij$ entry $q_{ij}^{\epsilon}$, and $\mathbf{H}^{\epsilon}$ be the $n \times 1$ vector with $i$th component $H^{ER}_{D^{\epsilon}}\paren{A_i,\epsilon e^{i\theta}}$.  Proposition \ref{ERPKfromAI} implies
\begin{equation}\label{ERPKEstimate7}
\mathbf{H}^{\epsilon}=\paren{\mathbf{I}-\mathbf{Q}^{\epsilon}}^{-1} \mathbf{T}^{\epsilon}.
\end{equation}

Let $\mathbf{Q}=\mathbf{Q}^0$ and $\mathbf{H}=\mathbf{H}^0$.  Using \eqref{ERPKEstimate4} and \eqref{ERPKEstimate5} we see
\begin{equation}\label{ERPKEstimate71}
h_i\paren{z}=h_i^{\epsilon}\paren{z}+O\paren{\epsilon^2},
\end{equation}
where $O\paren{\epsilon}$ is uniform over all $z \in \half$ such that $\abs{z}>r$.  It follows that $\mathbf{Q}^{\epsilon}=\mathbf{Q}+O\paren{\epsilon^2}$.  Since inversion of matrices is a smooth operation (and in particular, Lipschitz), we conclude
\begin{equation}\label{ERPKEstimate8}
\paren{\mathbf{I}-\mathbf{Q}^{\epsilon}}^{-1}=\paren{\mathbf{I}-\mathbf{Q}}^{-1}+O\paren{\epsilon^2}.
\end{equation}
Substituting \eqref{ERPKEstimate8} into \eqref{ERPKEstimate7} and using \eqref{BoundaryPKSCEstimate}, we obtain
\begin{align}
\mathbf{H}^{\epsilon}&=\paren{\mathbf{I}-\mathbf{Q}^{\epsilon}}^{-1} \mathbf{T}^{\epsilon}\nonumber \\
&=2 \sin \theta \bracket{ \paren{\mathbf{I}-\mathbf{Q}}^{-1}+O\paren{\epsilon^2}}\mathbf{T} \bracket{1+O\paren{\epsilon}} \nonumber\\
&=2 \sin \theta  \paren{\mathbf{I}-\mathbf{Q}}^{-1}\mathbf{T} \bracket{1+O\paren{\epsilon}}\nonumber\\
&=2 \sin \theta  ~\mathbf{H} \bracket{1+O\paren{\epsilon}} \label{ERPKEstimate10}.
\end{align}
Finally, using \eqref{ERPKEstimate71}, \eqref{ERPKEstimate10}, and Lemma \ref{PKStandardChordalRMHalfDisk}, we see
\begin{align*}
H^{ER}_{D^{\epsilon}}\paren{z,\epsilon e^{i\theta}}&=H_{D^{\epsilon}}\paren{z,\epsilon e^{i\theta}}+\sum_{i=1}^n h_i^{\epsilon}\paren{z}H^{ER}_{D^{\epsilon}}\paren{A_i,\epsilon e^{i\theta}}\\
&=2\sin \theta \bracket{H_D\paren{z,0}    + \sum_{i=1}^n h_i^{\epsilon}\paren{z}H_D^{ER}\paren{A_i,0}}\bracket{1+O\paren{\epsilon}}\\
&=2\sin \theta \bracket{H_D\paren{z,0}    + \sum_{i=1}^n h_i\paren{z}H_D^{ER}\paren{A_i,0}}\bracket{1+O\paren{\epsilon}}\\
&=2 \sin \theta ~H_D^{ER}\paren{z,0}\bracket{1+O\paren{\epsilon}}.
\end{align*}
\end{proof}

For any $0<r<R$, there are bounds for $H_D^{ER}\paren{z,0}$ and $H_D^{ER}\paren{z,x}$, for $\abs{x}>2R$, that are uniform over all $D\in \mathcal{Y}^*$ that agree with $\half$ outside of $A_{r,R}$.

\begin{lemma}\label{uniPKAiBd}
Let $D \in \mathcal{Y}_n^*$ be such that $B_r^+\paren{0} \subset D$ and $z \in D$ for all $z \notin B_R^+\paren{0}.$  Then there are constants $c_r,c_R<\infty$ such that for each $1 \leq i \leq n$, $H_D^{ER}\paren{A_i,0} \leq c_r$ and, if $\abs{x} \geq 2R$, $H_D^{ER}\paren{A_i,x} \leq c_Rx^{-2}$.
\end{lemma}

\begin{proof}
Fix $j$ and let $\sigma_1$ be the first time $B_D^{ER}$ started at $A_j$ hits $\partial B^+_{r/2}\paren{0}$ and, for $k>1$, let $\sigma_k$ be the first time after $\sigma_{k-1}$ that $B_D^{ER}$ has hit an $A_i$ with $i\geq 1$ and then returned to $\partial  B^+_{r/2}\paren{0}$.  Let $p_n\paren{\theta}$ be the density for the distribution of $B_{D}^{ER}\paren{\sigma_n}$ conditioned on $\sigma_n<\infty$ and $q_n$ be the probability that $\sigma_n< \infty$.  Using the strong Markov property for $B_D^{ER}$ and \eqref{PKhalf}, we see
\begin{align}
H_{D}^{ER}\paren{A_j,0} &=  \sum_{n=1}^{\infty} q_n \bracket{\frac{r}{2}\int_0^{\pi} H_{D}\paren{\paren{r/2}e^{i\theta},0}  p_n\paren{\theta}~d\theta} \label{uniPKBound1}\\
&\leq \sum_{n=1}^{\infty} q_n \bracket{ \frac{r}{2}\int_0^{\pi} H_{\half}\paren{\paren{r/2}e^{i\theta},0}  p_n\paren{\theta}~d\theta } \nonumber \\
&\leq \sum_{n=1}^{\infty} q_n \bracket{\frac{r}{2}\int_0^{\pi} \frac{2}{\pi r}  p_n\paren{\theta}~d\theta} \nonumber \\
&= \frac{2}{\pi r}\sum_{n=1}^{\infty} q_n. \nonumber
\end{align}
To complete the proof, it is enough to show that $\sum_{n=1}^{\infty} q_n$ is less than infinity.  If $\sigma_n< \infty$, in order for $\sigma_{n+1}$ to be less than infinity, a Brownian motion started on $\partial B_{r/2}^+\paren{0}$ will have to hit $\partial B_R^+\paren{0}$ before it hits the real line.  It is easy to verify that there is a $p<1$ uniformly bounding the probability of this event.  It follows that $q_{n} \leq p^{n-1} q_1$ and hence, $\sum_{n=1}^{\infty} q_n \leq \frac{q_1}{1-p}$. This proves the first statement.  

Observe that \eqref{PKhalfMax} implies $H_{\half}\paren{\frac{3R}{2}e^{i\theta},x}<\frac{24R}{\pi x^2}$ for all $\theta$ and $x$ such that $\abs{x}>2R$.  Using this fact, the proof of the second statement is similar to the proof of the first.


\end{proof}

\begin{lemma}\label{uniPKBound}
Let $D \in \mathcal{Y}_n^*$ and suppose that $B_r^+\paren{0}\subset D$.  Then there is a constant $c_r>0$ such that $H_D^{ER}\paren{z,0}\leq c_r$ for all $z$ with $\abs{z}>r$.  If $w \in D$ for all $w \in \half$ such that $\imag\bracket{w}<r'$, then there is a constant $c_{r'}>0$ such that $H_D\paren{z,x}<c_{r'}$ for all $z$ with $\imag\bracket{z}>r'$ and $x \in \R$.
\end{lemma}

\begin{proof}
An ERBM started at $z \in D$ with $\abs{z}>r$ has to hit $\partial B_r^+\paren{0}$ before it can hit $0$.  As a result, the strong Markov property for ERBM implies that to prove the first statement, it is enough to find a bound for $H_D^{ER}\paren{\cdot,0}$ restricted to $\partial B_r^+\paren{0}$.  Since
\begin{equation*}
H_{D}^{ER}\paren{re^{i\theta},0}=H_{D}\paren{re^{i\theta},0}+\sum_{j=1}^n h_j\paren{re^{i\theta}} H_{D}^{ER}\paren{A_j,0},
\end{equation*}
the necessary bound follows from Lemma \ref{uniPKAiBd} and the fact that 
\begin{equation*}
H_{D}\paren{re^{i\theta},0} \leq H_{\half}\paren{re^{i\theta},0}\leq \frac{1}{\pi r}.
\end{equation*}
The proof of the second statement is similar.
\end{proof}

There is an analog to Lemma \ref{PKDeriv} for ERBM.

\begin{lemma}\label{ERPKDeriv}
Let $D \in \mathcal{Y}^*$ and define $f_z\paren{x}:=H_D^{ER}\paren{z,x}$.  There is a $c>0$ such that if $x \in \R$ and $r<\abs{z-x}$ are such that $B_r^+\paren{x}\subset D$, then $\abs{f_z'\paren{x}}\leq c r^{-2}$.
\end{lemma}

\begin{proof}
We may assume without loss of generality that $x=0$.  Define $f_i\paren{x}$ to be equal to $H_D^{ER}\paren{A_i,x}$.  Since
\[f_z\paren{x}=H_D\paren{z,x}+\sum_{i=1}^n h_i\paren{z}f_i\paren{x},\]
using Lemma \ref{PKDeriv}, to complete the proof it is enough to show there is a $c>0$ such that $f_i'\paren{0}<c r^{-2}$ for all $1 \leq i \leq n$.
Using \eqref{uniPKBound1} (and the notation preceding it), we have
\[ f_i\paren{x} =  \sum_{n=1}^{\infty} q_n \bracket{\frac{r}{2}\int_0^{\pi} H_{D}\paren{\paren{r/2}e^{i\theta},0}  p_n\paren{\theta}~d\theta}.\]
Differentiating both sides of this equation and using the bounded convergence theorem, Lemma \ref{PKDeriv}, and the computation following \eqref{uniPKBound1}, the result follows.
\end{proof}

The next lemma gives an estimate on the effect on the Poisson kernel of removing a compact $\half$-hull from a domain $D$.

\begin{lemma}\label{ERPKEffHull}
Let $D \in \mathcal{Y}$ be such that $A_0=\C \backslash \half$ and suppose that there are real constants $0<r<R$ such that $w \in D$ for all $w \notin A_{r,R}$ and a constant $r'$ such that $w \in D$ for all $w \in \half$ with $\imag \bracket{w}<r'$.  Let $A$ be a compact $\half$-hull contained in $B_{r/2}^+\paren{0}$.  If $\abs{x}>\rad\paren{A}+\sqrt{\rad\paren{A}}$ and $\abs{z}>r$, then there is a $c>0$ depending only on $r$, $r'$, and $R$ such that
\[H_D^{ER}\paren{z,x}-H_{D\backslash A}^{ER}\paren{z,x} \leq c H_D^{ER}\paren{z,0}\rad\paren{A}.\]
Furthermore, if $\abs{x}>2R$, there is a $c>0$ depending only on $r$ and $R$ such that
\[H_D^{ER}\paren{z,x}-H_{D\backslash A}^{ER}\paren{z,x} \leq c H_D^{ER}\paren{z,0}x^{-2}\rad \paren{A}^2.\]
\end{lemma}

\begin{proof}
Using \eqref{twoDomainERPK}, we see that
\[H_D^{ER}\paren{z,x}-H_{D\backslash A}^{ER}\paren{z,x}=\ev{z}{H^{ER}_{D}\paren{B_{\tau_{D\backslash A}}^{ER},x}}.\]
Let $\epsilon=\rad\paren{A}$.  We can bound $\ev{z}{H^{ER}_{D}\paren{B_{\tau_{D\backslash A}}^{ER},x}}$ by the probability that an ERBM started at $z$ hits $\partial B_{\epsilon}^+\paren{0}$ before leaving $D$ multiplied by the maximum value of $H_D^{ER}\paren{\cdot,x}$ restricted to $\partial B_{\epsilon}^+\paren{0}$.  Proposition \ref{ERPKEstimate} implies the probability an ERBM started at $z$ hits $\partial B_{\epsilon}^+\paren{0}$ before leaving $D$ is
\begin{equation}\label{ERPKEffHull1}
4 \epsilon H_D^{ER}\paren{z,0}\bracket{1+O\paren{\epsilon}}.
\end{equation}
Next, recall that
\begin{equation}\label{ERPKEffHull2}
H_D^{ER}\paren{\epsilon e^{i\theta},x}=H_D\paren{\epsilon e^{i\theta},x}+\sum_{i=1}^n h_i\paren{\epsilon e^{i\theta}} H_D^{ER}\paren{A_i,x}.
\end{equation}
Since $H_D\paren{\epsilon e^{i\theta},x}<H_{\half}\paren{\epsilon e^{i\theta},x}$, \eqref{PKhalfMax} implies $H_D\paren{\epsilon e^{i\theta},x}$ is uniformly bounded for $\abs{x}>\epsilon+\sqrt{\epsilon}$ and less than a constant depending only on $R$ multiplied by $\epsilon x^{-2}$ for $\abs{x}>2R$.  The remark following \eqref{PKHD} implies $\sum_{i=1}^n h_i\paren{\epsilon e^{i\theta}}=O\paren{\epsilon}$ as $\epsilon \rightarrow 0$.  Lemma \ref{uniPKBound} implies $H_D^{ER}\paren{A_i,x}$ is bounded by a constant depending only on $r'$ for all $x$ and Lemma \ref{uniPKAiBd} implies $H_D^{ER}\paren{A_i,x}$ is bounded by a constant depending only on $R$ multiplied by $x^{-2}$ for $\abs{x}>2R$.  Combining these facts with \eqref{ERPKEffHull1} and \eqref{ERPKEffHull2}, the results follow.
\end{proof}

\subsection{Conformal Mapping Using $H_D^{ER}(\cdot,w)$} \label{sectPKconMap}
\label{chapPKconMap}
Recall that a domain is called a chordal standard domain if it obtained by removing a finite number of horizontal line segments from the upper half-plane.  It is a classical theorem of complex analysis \cite{MR510197} that every $D \in \mathcal{Y}_n$ is conformally equivalent to a chordal standard domain.  Furthermore, this equivalence is unique up to a scaling and real translation.  In this section, we discuss the relationship between this conformal equivalence and $H_D^{ER}\paren{\cdot,w}$.  In what follows, assume that $\partial A_0$ is locally analytic at $w \in \partial A_0$.  

There is an analytic characterization of $H_D^{ER}\paren{\cdot, w}.$ 

\begin{proposition} \label{ERBMPoissonKernelCharacterization}
$H_D^{ER}\paren{\cdot,w}$ is up to a real constant multiple the unique positive ER-harmonic function that satisfies $H_D^{ER}\paren{z,w} \rightarrow 0$ as $z \rightarrow w'$ for any $w' \in \partial A_0$ not equal to $w$.
\end{proposition}

\begin{proof}
Using \eqref{ERBMPKD}, we see that $H_D^{ER}\paren{\cdot,w}$ is harmonic on $D$.  If $V$ is a Borel subset of $\partial A_0$, then it follows from the strong Markov property for ERBM and (3) of Definition \ref{characterizationERBM} that $\hm_D^{ER}\paren{\cdot,V}$ is ER-harmonic.  As a result, if $\gamma$ is as in \eqref{explicitPoisson} and $\eta$ and $U_i$ are as in Definition \ref{ERHarmonicDef}, then
\begin{align*}
H_D^{ER}\paren{A_i,w} &= \lim_{\epsilon \rightarrow 0} \frac{\hm_D^{ER}\paren{A_i,\gamma\paren{-\epsilon,\epsilon}}} {\int_{-\epsilon}^{\epsilon}\abs{\gamma'\paren{x}}~dx}\\
&=  \lim_{\epsilon \rightarrow 0} \int_{\eta} \frac{\hm_D^{ER}\paren{z,\gamma\paren{-\epsilon,\epsilon}}}{\int_{-\epsilon}^{\epsilon}\abs{\gamma'\paren{x}}~dx} \cdot \frac{H_{\partial U_i}\paren{A_i,z}}{\mathcal{E}_{U_i}\paren{A_i,\eta}} ~\abs{dz}\\
&=\int_{\eta_i} H_D^{ER}\paren{z,w} \cdot \frac{H_{\partial U_i}\paren{A_i,z}}{\mathcal{E}_{U_i}\paren{A_i,\eta}} ~\abs{dz},
\end{align*}  
where the last equality follows from the Harnack inequality and dominated convergence.  This proves that $H_D^{ER}\paren{\cdot,w}$ is ER-harmonic.  It is clear from \eqref{ERBMPKD} and Proposition \ref{noSubLinHar} that $H_D^{ER}\paren{\cdot,w}$ has the required asymptotics at $\partial A_0$.

Suppose $f$ is another positive ER-harmonic function that satisfies $f\paren{z} \rightarrow 0$ as $z \rightarrow w'$ for any $w' \in \partial A_0$ not equal to $w$.  The function
\[g\paren{z}:=f\paren{z}-\sum_{i=1}^n h_i\paren{z}f\paren{A_i}\]
is a harmonic function with $g\paren{A_i}=0$ for each $1 \leq i \leq n$ that has the same boundary conditions as $f$ at $ \partial A_0$.  It follows from Proposition \ref{noSubLinHar} that there is a $c>0$ such that $g\paren{z}=c H_D\paren{z,w}$.  As a result, \eqref{ERBMPKD} implies $f\paren{z}-c H^{ER}_D\paren{z,w}$ is an ER-harmonic function that is $0$ on $A_0$ and thus, by the maximal principle for ER-harmonic functions, $f\paren{z}=c H^{ER}_D\paren{z,w}$ for all $z \in D$.
\end{proof}

\begin{theorem} \label{PKConformalMapTheorem}
Let $D \in \mathcal{Y}_n$ and suppose $\partial A_0$ is a smooth Jordan curve (in the topology of $E$) such that there is no Jordan curve in $D$ with $A_0$ in its interior.  If $w \in \partial A_0$, then there is a $D' \in \mathcal{CY}_n$ and conformal map $f: D \rightarrow D'$ with $f\paren{w}=\infty$ such that $\imag\bracket{f\paren{z}}=H_D^{ER}\paren{z,w}$.  Furthermore, if $g$ is another such map, then there are real constants $r,x$ such that $g=r f+x$.
\end{theorem}

\begin{proof}
Using Lemma \ref{ERHarmonicCondition} and (say) Proposition 13.3.5 of \cite{MR1344449}, we see that a harmonic function $h$ that is continuous on $E$ is the imaginary part of a holomorphic function if and only if it is ER-harmonic.  It follows that if $D'\in \mathcal{CY}_n$ and $f:D \rightarrow D'$ is a conformal map with $f\paren{w}=\infty$, then the imaginary part of $f$ is a positive ER-harmonic function such that $f\paren{z} \rightarrow 0$ as $z \rightarrow w'$ for any $w'\neq w$.  By Proposition \ref{ERBMPoissonKernelCharacterization}, this implies that the imaginary part of $f$ is a real constant multiple of $H_D^{ER}\paren{\cdot,w}$.  Combining this with the fact the imaginary part of a holomorphic function determines the real part up to a real additive constant, we obtain the uniqueness statement.

The existence of $f$ is a classical result of complex analysis.  The map can also be explicitly constructed using $H_D^{ER}\paren{\cdot,w}$ (see \cite{DreThesis} for details).
\end{proof}

\section{The Green's Function for ERBM} \label{sectGF}

\subsection{Definition and Basic Properties}

Throughout this section, let $D \in \mathcal{Y}$ be such that it is possible to define a Green's function $G_D\paren{z,w}$ for Brownian motion.  Recall that we normalize $G_D\paren{z,\cdot}$ so that it is a density for the expected amount of time a Brownian motion started at $z$ spends in a set before exiting $D$.

\begin{definition}
\[G^{ER}_{D}\paren{z,\cdot}:E \rightarrow \R\]
is a \emph{Green's function for ERBM} if for any Borel subset $V \subset D$
\begin{equation}\label{GrDef}
\mu_z\paren{V}:=\ev{z}{\int_0^{\tau_{D}} \mathbf{1}_{V}\paren{B_{D}^{ER}\paren{t}} dt}=\int_V G^{ER}_{D}\paren{z,w} dw,
\end{equation}
where $\tau_{D}=\inf \set{t: B_{D}^{ER}\paren{t} \in \partial A_0}$.
\end{definition}

Using the definition of ERBM and the analogous fact for Brownian motion, it is easy to prove that the probability that ERBM started at $z$ is in a set of Lebesgue measure zero at some fixed time is $0$.  Combining this fact with Fubini's theorem, we see that if $V$ has Lebesgue measure zero, then
\[\mu_z\paren{V}=\int_0^{\tau_{D}} \prob{z}{B_{D}^{ER}\paren{t} \in V} dt =0.\]
As a result, we can define $G^{ER}_{D}\paren{z,\cdot}$ as a Radon-Nikodym derivative.  Furthermore, we have 
\begin{equation}\label{RNGreenERBM}
\displaystyle G_D^{ER}\paren{z,w}=\lim_{\epsilon \rightarrow 0} \frac{\mu_z\paren{B\paren{w,\epsilon}}}{m\paren{B\paren{w,\epsilon}}}
\end{equation}
is a Green's function for ERBM, where $m$ is Lebesgue measure.  A priori, there is no reason the Green's function as defined cannot be infinite on a set of positive measure.  This potential issue will be resolved by Proposition \ref{ERBMGreenDecomp} and \eqref{GreenMatrixGreenFunction}.

We have only given a probabilistic definition of $G_D^{ER}\paren{z,\cdot}$ and our definition is unique only as an element of $L^1\paren{D}$.  It is also possible to give an analytic characterization of $G_D^{ER}\paren{z,\cdot}$.  More specifically, we will prove that there is a version of $G_D^{ER}\paren{z,\cdot}$ that is the unique ER-harmonic function on $D-\set{z}$ satisfying certain boundary conditions (that depend on whether or not $z$ is equal to some $A_i$).  In particular, this will allow us to talk about ``the'' Green's function for ERBM rather than ``a'' Green's function.  We start by proving an analog of \eqref{ERBMPKD} for $G_D^{ER}\paren{z,\cdot}.$

\begin{proposition}\label{ERBMGreenDecomp}
\[G_{D}\paren{z,w}+\sum_{i=1}^n h_i\paren{z}G_{D}^{ER}\paren{A_i,w}\] is a version of $G_D^{ER}\paren{z, \cdot}$.
\end{proposition}

\begin{proof}
This follows easily using the strong Markov property for ERBM and the fact that up until the first time it hits $\partial D$, ERBM has the distribution of a Brownian motion.
\end{proof}

As we expect, $G_D^{ER}\paren{z, \cdot}$ is conformally invariant.  To prove this we need the following lemma, which is a straightforward exercise in measure theory.

\begin{lemma}
If $g\in L^1\paren{D}$, then for all Borel $V \subset D$ we have
\begin{equation*}
\ev{z}{\int_0^{\tau_{D}} \mathbf{1}_{V}\paren{B^{ER}_{D}\paren{t}}g\paren{B^{ER}_{D}\paren{t}} dt}=\int_V G^{ER}_{D}\paren{z,w}g\paren{w} dw.
\end{equation*}
\label{greendefExtension}
\end{lemma}



\begin{proposition}\label{ERBMGreenCI}
If $f: D \rightarrow D'$ is a conformal map, then
\[G^{ER}_{D}\paren{f^{-1}\paren{z},f^{-1}\paren{\cdot}}\]
is a version of $G^{ER}_{D'}\paren{z,\cdot}$.
\end{proposition}

\begin{proof}
It is enough to show that $G^{ER}_{D}\paren{f^{-1}\paren{z},f^{-1}\paren{\cdot}}$ satisfies \eqref{GrDef} for all open subsets of $D'$.  Let $V'$ be an open subset of $D'$ and $V=f^{-1}\paren{V'}$. Using Lemma \ref{greendefExtension} and the change of variables formula, we have
\begin{align}
\int_{V'} G^{ER}_{D}\paren{z,f^{-1}\paren{w}} dw &= \int_{V}G^{ER}_{D}\paren{z,w} \abs{f'\paren{w}}^2 dw\nonumber\\ \label{proofCI2}
&=\ev{z}{\int_0^{\tau_D} \abs{\mathbf{1}_{V}\paren{B^{ER}_{D}\paren{t}}f'\paren{B^{ER}_{D}\paren{t}}}^2 dt}.
\end{align}
Let
\begin{equation}
u\paren{t}=\int_0^t \abs{f'\paren{B^{ER}_{D}\paren{s}}}^2 ds.
\label{proofCI3}
\end{equation}
Substituting $u^{-1}\paren{r}$ for $t$ and using the conformal invariance of ERBM, we see that \eqref{proofCI2} is equal to
\begin{equation}
\ev{z}{\int_0^{\tau_{D'}} \mathbf{1}_{V'}\paren{B^{ER}_{D}\paren{t}} dt},
\label{proofCI4}
\end{equation}
which completes the proof.
\end{proof}
In the proof of Proposition \ref{ERBMGreenCI}, observe that we can only conclude that \eqref{proofCI2} is equal to \eqref{proofCI4} if \eqref{proofCI3} is almost surely finite for all $t<\infty$.  This will be addressed when we prove \eqref{ERBMnolimit}.

In order to prove $G_D^{ER}\paren{z,\cdot}$ is ER-harmonic, we will need to be able to compute $G_{A_{1,r}}\paren{z,\cdot}$.  Using Proposition \ref{ERBMGreenDecomp}, to do this, it is enough to compute $G_{A_{1,r}}^{ER}\paren{A_1,\cdot}.$

\begin{lemma}  \label{GreensFunctionAnnulus} Let $A_{1,r} \in \mathcal{Y}_1$ be the annulus with $A_1=\overline{\D}$ and $\partial A_0=\partial B_r\paren{0}$ for some $r>1$ and $B_t$ be a Brownian motion in $r \D$. If $V$ is a Borel set bounded away from $A_1$, then
\begin{equation*}
\ev{A_1}{\int_0^{\tau_{A_{1,r}}} \mathbf{1}_{V}\paren{B_{A_{1,r}}^{ER}\paren{t}} dt}=\ev{0}{\int_0^{\tau_{r\D}} \mathbf{1}_{V}\paren{B_t}}dt,
\end{equation*}
where $\tau_{A_{1,r}}$ and $\tau_{r\D}$ are respectively the first time $B_{A_{1,r}}^{ER}$ leaves $A_{1,r}$ and $B_t$ leaves $r\D$.  Furthermore, we have
\begin{equation*}
G_{A_{1,r}}^{ER}\paren{A_1,z}=\frac{-\log \abs{z}+\log r}{\pi}.
\end{equation*}
\end{lemma}

\begin{proof}
Since $V$ is bounded away from $\overline{D}$, there exists an $\epsilon>0$ such that $V$ is contained in the region bounded by the circle
\[C_{\epsilon}=\set{z \in \C:\abs{z}=1+\epsilon} \]
and the outer boundary of $A_{1,r}$.  Let $\sigma_1=0$, $\tau_j$ be the first time after $\sigma_j$ that $B_{A_{1,r}}^{ER}$ hits $C_{\epsilon}$, and $\sigma_j$ for $j>1$ be the first time after $\tau_{j-1}$ that $B_{A_{1,r}}^{ER}$ hits $A_1$.  Similarly, let $\sigma_1'=0$, $\tau'_j$ be the first time after $\sigma'_j$ that $B_t$ hits $C_{\epsilon}$, and $\sigma'_j$ for $j>1$ be the first time after $\tau'_{j-1}$ that $B_t$ hits the circle of radius $1$, and $\sigma'_1=0$.  It follows from the strong Markov property for ERBM and (3) of Definition \ref{characterizationERBM} that given that $\tau_j<\infty$, the distribution of $B_{A_{1,r}}^{ER}\paren{\tau_j}$ is uniform on $C_{1+\epsilon}$.  It is an easy exercise to check that given that $\tau'_j<\infty$, the distribution of $B_{\tau_j}$ is uniform on $C_{1+\epsilon}$.  Using these two facts, the strong Markov property for ERBM, and the fact that an ERBM has the distribution of a Brownian motion up until the first time it hits the boundary of $A_{1,r}$, we see that
\begin{equation*}
\ev{B_{A_{1,r}}^{ER}\paren{\tau_j}}{\int_{\tau_j}^{\sigma_{j+1}} \mathbf{1}_{V}\paren{B_{A_{1,r}}^{ER}\paren{t}} dt}=\ev{B_{\tau'_j}}{\int_{\tau'_j}^{\sigma'_{j+1}} \mathbf{1}_{V}\paren{B_{r\D}\paren{t}} dt}.
\end{equation*}
Combined with the fact that
\begin{equation*}
\ev{B_{A_{1,r}}^{ER}\paren{\sigma_j}}{\int_{\sigma_j}^{\tau_j} \mathbf{1}_{V}\paren{B_{A_{1,r}}^{ER}\paren{t}} dt}=\ev{B_{\sigma'_j}}{\int_{\sigma'_j}^{\tau'_j} \mathbf{1}_{V}\paren{B_{r\D}\paren{t}} dt}=0,
\end{equation*}
the first result follows.

Using the first part of the proposition, we see that 
\[G_{A_{1,r}}^{ER}\paren{A_1,z}=G_{r\D}\paren{0,z}.\]
Combining this with \eqref{GFUniDisk}, the second part of the proposition follows.
\end{proof}

A quantity that will help us understand $G_D^{ER}\paren{z,\cdot}$ is the density for the amount of time ERBM started at $A_i$ spends in a set from the time it hits a curve $\eta_i$ surrounding $A_i$ until the time it hits $\partial D$ again.  The next lemma establishes the existence and some properties of this density.

\begin{lemma} \label{GreenERBMOneTime}
For $i=1, \ldots, n$, let $\eta_i$ and $U_i$ be as in Definition \ref{characterizationERBM}.  The function
\begin{equation*}
T_i\paren{w}:=G^{ER}_{U_i}\paren{A_i,w}+\int_{\eta_i}G_D\paren{z,w} \frac{H_{\partial U_i}\paren{A_i,z}}{\mathcal{E}_{U_i}\paren{A_i, \eta}} \abs{dz},
\end{equation*}
where by convention we let $G^{ER}_{U_i}\paren{A_i,w}=0$ for $w \notin U_i$, has the following properties.
\begin{enumerate}
\item $T_i\paren{w}$ is a density for the expected amount of time ERBM started at $A_i$ spends in a set up until $\tau_2$, where $\tau_2$ is as in Section \ref{ERBMGeneralDef}
\item $T_i\paren{w}$ is harmonic on $D \backslash \eta_i$
\item If $i\neq j$, then $\displaystyle \frac{1}{2}\int_{\eta_j}\frac{d}{dn}~ T_i\paren{w}   \abs{dw}=p_{i j}$, where $n$ is the outward-pointing normal and $p_{i j}$ is as in Section \ref{MarkovChainERBM}
\item If $\eta'_i$ is a smooth curve in the interior of $U_i$ that is homotopic to $\eta_i$, then
\[\frac{1}{2}\int_{\eta'_i}\frac{d}{dn}~ T_i\paren{w}   \abs{dw}=p_{ii}-1.\]
\end{enumerate}
\end{lemma}

\begin{proof}
It is clear using the strong Markov property for ERBM, the fact that ERBM has the distribution of a Brownian motion up until the first time it hits $\partial D$, and (3) of Definition \ref{characterizationERBM} that the first statement holds.

Denote the second summand in the definition of $T_i\paren{w}$ by $S_i\paren{w}$.  If $w \notin \eta_i$ and $\epsilon$ is small enough such that $B\paren{w,\epsilon}$ does not intersect $\eta_i$, then  using Fubini's theorem and the fact that $G_D\paren{z,\cdot}$ is harmonic, we have
\begin{align*}
\frac{1}{2\pi}\int_0^{2\pi} S_i\paren{w+\epsilon e^{i\theta}} d\theta &= \frac{1}{2\pi}\int_0^{2\pi}\bracket{\int_{\eta_i}G_D\paren{z,w+\epsilon e^{i\theta}} \frac{H_{\partial U_i}\paren{A_i,z}}{\mathcal{E}_{U_i}\paren{A_i, \eta}} \abs{dz}} d\theta\\
&=\int_{\eta_i} \frac{H_{\partial U_i}\paren{A_i,z}}{\mathcal{E}_{U_i}\paren{A_i, \eta}}\bracket{ \frac{1}{2\pi}\int_0^{2\pi} G_D\paren{z,w+\epsilon e^{i\theta}} d\theta} \abs{dz}\\
&=\int_{\eta_i} \frac{H_{\partial U_i}\paren{A_i,z}}{\mathcal{E}_{U_i}\paren{A_i, \eta}}G_D\paren{z,w} \abs{dz}\\
&=S_i\paren{w}.
\end{align*}
This shows that $S_i\paren{w}$ satisfies the spherical mean value property at $w$ and, thus, is harmonic on $D \backslash \eta_i$.  It follows that to finish the proof of the second statement, we just have to show that $G^{ER}_{U_i}\paren{A_i,\cdot}$ is harmonic away from $\eta_i$.  Let $f_i:A_{1,r_i} \rightarrow U_i$ be a conformal map mapping the outer boundary of $A_{1,r_i}$ to the outer boundary of $U_i$.  Using Proposition \ref{ERBMGreenCI} and Lemma \ref{GreensFunctionAnnulus}, we see that
\begin{equation}\label{GreenERBMAnnulus}
G^{ER}_{U_i}\paren{A_i,w}=G^{ER}_{D_{r_i}}\paren{A_1,f_i\paren{w}}=\frac{-\log \abs{f_i\paren{w}} + \log r_i}{\pi}.
\end{equation}
Since $\log\abs{z}$ is harmonic and precomposing a harmonic function with a conformal map yields a harmonic function, $G^{ER}_{U_i}\paren{A_i,\cdot}$ is harmonic away from $\eta_i$.

The proof of the third statement uses the fact that if $z$ is in the exterior of $\eta_j$, then
\begin{equation}\label{GreenNormalPoisson1}
\int_{\eta_j}\frac{d}{dn}G_D\paren{z,w} \abs{dw}=2 h_j\paren{z}.
\end{equation}
In the case that $\partial A_j$ is a smooth Jordan curve, this is true because the normal derivative of $G_D\paren{z,w}$ is $2 H_D\paren{z,w}$ on $\partial A_j$ and the integral of the normal derivative of a harmonic function is the same over any two homotopic curves.  If the boundary of $A_j$ is not a smooth Jordan curve, we can map $D$ conformally to a region where the image of $\partial A_j$ is a smooth Jordan curve \cite{MR1344449} and use the conformal invariance of the Green's function, the change of variables formula, the fact that conformal maps preserve angles and the result in the case that $\partial A_j$ is a smooth Jordan curve.  If $i\neq j$, using Fubini's theorem, the dominated convergence theorem, and \eqref{GreenNormalPoisson1}, we have
\begin{align*}
\int_{\eta_j}\frac{d}{dn}~T_i\paren{w} \abs{dw} &= \int_{\eta_j}\frac{d}{dn} \int_{\eta_i}G_D\paren{z,w} \frac{H_{\partial U_i}\paren{A_i,z}}{\mathcal{E}_{U_i}\paren{A_i, \eta}} \abs{dz} \abs{dw}\\
&= \int_{\eta_i}\frac{H_{\partial U_i}\paren{A_i,z}}{\mathcal{E}_{U_i}\paren{A_i, \eta}} \int_{\eta_j}\frac{d}{dn}~G_D\paren{z,w} \abs{dw} \abs{dz}\\
&= 2 \int_{\eta_i}\frac{H_{\partial U_i}\paren{A_i,z}}{\mathcal{E}_{U_i}\paren{A_i, \eta}}h_j\paren{z} \abs{dz}\\
&= 2 p_{i j}
\end{align*}

The proof of the fourth statement is similar to the proof of the third statement and will rely on calculating
$\displaystyle\int_{\eta'_i}\frac{d}{dn}G_D\paren{z,w} \abs{dw}.$  If $z$ is a point in the interior of $\eta'_i$ and $\eta^{''}_i$ is a smooth Jordan curve in the interior of $\eta'_i$ such that $z$ is in the exterior of $\eta^{''}_i$, then by setting up the appropriate contour integral and using Green's theorem, it is not hard to see that (with the normals appropriately oriented)
\[\int_{\eta'_i}\frac{d}{dn}~G_D\paren{z,w} \abs{dw}= \int_{\eta^{''}_i} \frac{d}{dn}~G_D\paren{z,w} \abs{dw} + \int_{B_{\epsilon}\paren{z}} \frac{d}{dn} G_D\paren{z,w} \abs{dw}.\]
Using \eqref{GreenNormalPoisson1} and the fact that $G_D\paren{z,w}=-\frac{\log \abs{z-w}}{\pi}+g_z\paren{w}$, where $g_z$ is harmonic on $D$, we have
\begin{align*}
\int_{\eta'_i}\frac{d}{dn}~G_D\paren{z,w} \abs{dw}&= \int_{\eta^{''}_i} \frac{d}{dn}~G_D\paren{z,w} \abs{dw}+\int_{B_{\epsilon}\paren{z}} \frac{d}{dn}~G_D\paren{z,w} \abs{dw}\\
&=2 h_i\paren{z}-\int_{B\paren{z,\epsilon}} \frac{d}{dn}~\frac{\log \abs{z-w}}{\pi} \abs{dw}\\
&=2 \paren{h_i\paren{z}-1}.
\end{align*}
Using this and arguing as in the proof of the third statement, we have
\begin{align*}
\int_{\eta'_i}\frac{d}{dn}~T_i\paren{w} &= \int_{\eta'_i}\frac{d}{dn} \int_{\eta_i}G_D\paren{z,w} \frac{H_{\partial U_i}\paren{A_i,z}}{\mathcal{E}_{U_i}\paren{A_i, \eta}} \abs{dz} \abs{dw}\\
&= \int_{\eta_i}\frac{H_{\partial U_i}\paren{A_i,z}}{\mathcal{E}_{U_i}\paren{A_i, \eta}} \int_{\eta'_i}\frac{d}{dn}~G_D\paren{z,w} \abs{dw} \abs{dz}\\
&= 2 \int_{\eta_i}\frac{H_{\partial U_i}\paren{A_i,z}}{\mathcal{E}_{U_i}\paren{A_i, \eta}}\paren{h_i\paren{z}-1} \abs{dz}\\
&= 2 \paren{p_{ii}-1}.
\end{align*}
\end{proof}

We have all of the tools necessary to prove that $G_{D}^{ER}\paren{z,\cdot}$ is ER-harmonic.  In what follows, we continue to use the set up of the previous lemma.

\begin{proposition} \label{ERBMGreenERHarmonic}
There are versions of $G_{D}^{ER}\paren{\cdot,z}$ and $G_{D}^{ER}\paren{z,\cdot}$ that are ER-harmonic on $D-\set{z}$.
\end{proposition}
\begin{proof}
Using Proposition \ref{ERBMGreenDecomp}, in order to show that there is a harmonic version of $G_{D}^{ER}\paren{z,\cdot}$, it is enough to show that there is a harmonic version of $G^{ER}_D\paren{A_i,\cdot}$ for each $1 \leq i \leq n$.  Let $\mathbf{T}$ be the vector function with $i$th component $T_i\paren{w}$ and let $\mathbf{D}$ be the diagonal matrix with $ii$ entry $\frac{1}{1-p_{ii}}$.  Using the strong Markov property for ERBM and Lemma \ref{GreenERBMOneTime}, we see that $\mathbf{D}\mathbf{T}$ is the vector function whose $i$th component is the density for the expected amount of time ERBM started at $A_i$ spends in a set up until the first time it hits an $A_j$ with $j\neq i$.  Using \eqref{GreenMatrixFormula} and the strong Markov property for ERBM, we see that the $i$th component of
\begin{equation} \label{GreenMatrixGreenFunction}
\mathbf{D} \mathbf{T}+\mathbf{Q} \mathbf{D} \mathbf{T}+\mathbf{Q}^2 \mathbf{D} \mathbf{T}+ \ldots  = \paren{\mathbf{I}-\mathbf{Q}}^{-1} \mathbf{D} \mathbf{T}
\end{equation}
is a version of $G^{ER}_D\paren{A_i,\cdot}$.  Since $T_i\paren{\cdot}$ is harmonic away from each $\eta_i$, it follows that there is a version of $G_D^{ER}\paren{A_i,\cdot}$ that is harmonic away from each $\eta_i$.  By choosing different $\eta_i$'s and repeating this procedure, we can get a version of $G_D^{ER}\paren{A_i,\cdot}$ that is harmonic away from a sequence of Jordan curves $\eta_1',\ldots, \eta_n'$ which are disjoint from each $\eta_i$.  Finally, since any two versions of $G_D^{ER}\paren{A_i,\cdot}$ are equal almost everywhere, we can find a version of $G_D^{ER}\paren{A_i,\cdot}$ that is harmonic everywhere.

Using (3) and (4) of Lemma \ref{GreenERBMOneTime} and the fact that the $i$th component of \eqref{GreenMatrixGreenFunction} is a version of $G_D^{ER}\paren{A_i,\cdot}$, we see that 
\begin{equation}\label{GreenERBMNormalInt} \int_{\eta_j} \frac{d}{dn}~G^{ER}_D\paren{A_i,w}\abs{dw}= \begin{cases}
0 & \text{ if }  j\neq i \\
-2 &\text{ if }  j=i
\end{cases}.
\end{equation}
It is easy to see using its definition and \eqref{GreenERBMAnnulus} that each $T_i\paren{\cdot}$, and thus each $G_D^{ER}\paren{A_i,\cdot}$, can be extended to a continuous function on $E$.  Combining this with Lemma \ref{ERHarmonicCondition}, we have that that there is a version of $G_D^{ER}\paren{A_i,\cdot}$ that is ER-harmonic on $D-\set{A_i}$.  Finally, using \eqref{GreenNormalPoisson1}, \eqref{GreenERBMNormalInt}, and Lemma \ref{ERHarmonicCondition}, it follows that the version of $G_D^{ER}\paren{z,\cdot}$ defined in Proposition \ref{ERBMGreenDecomp} is ER-harmonic on $D-\set{z}$.

An argument similar to the one showing that $H_D^{ER}\paren{\cdot,z}$ is ER-harmonic shows that $G_D^{ER}\paren{\cdot,z}$ is ER-harmonic.
\end{proof}

We can now give an analytic characterization of $G_D^{ER}\paren{z,\cdot}$.

\begin{proposition}\label{ERBMGreenUnique}
If $z \in D$, then $G_D^{ER}\paren{z,\cdot}$ is the unique ER-harmonic function on $D-\set{z}$ satisfying
\begin{itemize}
\item  $G_D^{ER}\paren{z,w}=-\frac{\log \abs{z-w}}{\pi}+O\paren{1}$ as $w \rightarrow z$
\item  $G_D^{ER}\paren{z,w} \rightarrow 0$ as $w \rightarrow w'$ for any $w' \in \partial A_0$.
\end{itemize}
Furthermore, $G_D^{ER}\paren{A_i,\cdot}$ is the unique ER-harmonic function on $D-\set{A_i}$ that is equal to $G_D^{ER}\paren{A_i,A_i}$ on $\partial A_i$ and $0$ on $\partial A_0$.

\end{proposition}
\begin{proof}
If $z \in D$, the asymptotics for $G_D^{ER}\paren{z,\cdot}$ at the boundary are clear and the asymptotic at $z$ follows from Proposition \ref{ERBMGreenDecomp} and the corresponding result for $G_D\paren{z,\cdot}$.  The uniqueness follows from a proof similar to the corresponding result for $G_D\paren{z,\cdot}$ (see \cite{MR2129588}, pg. 54).  The second statement follows from an extension of Proposition \ref{ERHarmonicFromERBM}.
\end{proof}

In what follows, when we write $G_D^{ER}\paren{z,\cdot}$ or $G_D^{ER}\paren{\cdot,w}$ we will mean a version that is ER-harmonic.

\begin{corollary}\label{GreenSym}
$G_D^{ER}\paren{z,w}=G_D^{ER}\paren{w,z}$ for all $z, w \in E$.
\end{corollary}
\begin{proof}
$G_D^{ER}\paren{\cdot,z}$ satisfies the conditions of Proposition \ref{ERBMGreenUnique} and thus, is the same function as $G_D^{ER}\paren{z,\cdot}$.
\end{proof}

We conclude this section by proving that the normal derivative on $\partial A_0$ of the Green's function for ERBM is a multiple of the Poisson kernel for ERBM.  We need a lemma.

\begin{lemma}
Let $D \in \mathcal{Y}_n^*$ be a domain such that $\partial D$ and $\partial \half$ agree in a neighborhood of $0$ and fix $R>0$ such that $B_R^+\paren{0}\subset D$.  Then if $\epsilon<r<R$,
\[\int_0^{\pi}G_D^{ER}\paren{re^{i \theta}, \epsilon i} \sin \theta  ~d\theta= \epsilon\bracket{\frac{1}{r}+O\paren{r}}, \]
as $\epsilon, r \rightarrow 0$.
\end{lemma}

\begin{proof}
Using Proposition \ref{ERBMGreenDecomp} and \eqref{twoDomainGF}, we see that
\[G_D^{ER}\paren{re^{i \theta}, \epsilon i}=G_{\half}\paren{re^{i \theta}, \epsilon i}-\ev{re^{i\theta}}{G_{\half}\paren{B_{\tau_D},\epsilon i}} +\sum_{i=1}^n h_i\paren{re^{i \theta}} G_D^{ER}\paren{A_i,\epsilon i}. \]
As a result, using Lemma \ref{GreenInt}, to complete the proof, it is enough to show that
\[\sum_{i=1}^n h_i\paren{re^{i \theta}} G_D^{ER}\paren{A_i,\epsilon i}-\ev{re^{i\theta}}{G_{\half}\paren{B_{\tau_D},\epsilon i}} =O\paren{r\epsilon},~~~\epsilon, r \rightarrow 0,\] where $O\paren{r \epsilon}$ is uniform over $\theta \in \paren{0,\pi}$.  

It is easy to check using Lemma \ref{GreenERBMOneTime}, \eqref{twoDomainGF}, \eqref{Greenztoi}, and \eqref{GreenMatrixGreenFunction} that $G_D^{ER}\paren{A_i, \epsilon i}=O\paren{\epsilon}$.  Using \eqref{PKHD}, it follows that for $1 \leq i \leq n$, $h_i\paren{re^{i\theta}}$ is $O\paren{r}$.  Since the probabilty a Brownian motion started at $re^{i\theta}$ does not exit $D$ on $\R$ is $O\paren{r}$, using \eqref{Greenztoi}, it follows that $\ev{re^{i\theta}}{G_{\half}\paren{B_{\tau_D},\epsilon i}}$ is $O\paren{r\epsilon}$.
\end{proof}


\begin{proposition} \label{GreenND}
Let $D \in \mathcal{Y}_n$ be such that $\partial D$ is locally analytic at $x \in \partial A_0$.  Then the (inner) normal derivative of $G_D^{ER}\paren{z,\cdot}$ at $x$ is $2H_D^{ER}\paren{z,x}$.
\end{proposition}

\begin{proof}
We start by proving the result when $D \in \mathcal{Y}_n$ is a domain such that $\C \backslash A_0 =\half$ and $x=0$.  Fix $R$ such that $B_R^+\paren{0}\subset D$.  Using Proposition \ref{ERPKEstimate}, we have that if $r<R$, then
\begin{align*}\displaystyle 
\lim_{\epsilon \rightarrow 0}\frac{G_D^{ER}\paren{z,\epsilon i}}{\epsilon}&=\lim_{\epsilon \rightarrow 0} \frac{r \int_0^{\pi} H_{D^{r}}\paren{z,r e^{i\theta}}G_D^{ER}\paren{r e^{i\theta},\epsilon i}~d\theta}{\epsilon} \\
&=\lim_{\epsilon \rightarrow 0} \frac{2rH_D^{ER}\paren{z,0} \bracket{\int_0^{\pi}G_D^{ER}\paren{re^{i \theta}, \epsilon i} \sin \theta  ~d\theta}\bracket{1+O\paren{r}}}{\epsilon}\\
&= 2H_D^{ER}\paren{z,0} \bracket{1+O\paren{r}}.
\end{align*}
Taking the limit as $r \rightarrow 0$, the result follows.

For arbitrary $D$, let $f:D \rightarrow D'$ be a conformal map such that $D' \in \mathcal{Y}_n$ is such that $\C \backslash A_0=\half$ and $x$ is mapped to $0$.  Using the Schwarz reflection principle, $f\paren{z}$ can be extended to a map conformal in a neighborhood of $x$ and $G_D^{ER}\paren{z,\cdot}$ can be extended to a function harmonic in a neighborhood of $x$.  Combining this with the fact that conformal maps preserve angles, the chain rule, \eqref{PKCI}, the conformal invariance of $G_D^{ER}\paren{z,\cdot}$, and the result for $D'$, the result follows.

\end{proof}

\subsection{Proofs of formulas (\ref{ERBMalltime}) and (\ref{ERBMnolimit})} \label{sectGFPOF}
\label{chapGFPOF}
The theory of Green's functions for ERBM can be used to prove formulas \eqref{ERBMalltime} and \eqref{ERBMnolimit}.  We start with a lemma.

\begin{lemma}
Let $A_{1,r} \in \mathcal{Y}_1$ be as in Lemma \ref{GreensFunctionAnnulus} and $\tau=\inf\set{t:B^{ER}_{A_{1,r}}\paren{t} \in \partial A_0}$.  If $f: A_{1,r} \rightarrow D$ is a conformal map and $D$ is bounded, then
\[\ev{z}{\int_0^{\tau} \abs{f'\paren{B^{ER}_{A_{1,r}}\paren{s}}}^2 ds} < \infty.\]
\label{allpathlemma}
\end{lemma}

\begin{proof}
Using Lemma \ref{greendefExtension}, we have that for sufficiently small $\epsilon$
\begin{align*}
\ev{}{\int_0^{\tau} \abs{f'\paren{B^{ER}_{A_{1,r}}\paren{s}}}^2 ds}=&\int_{A_{1,r}}G^{ER}_{A_{1,r}}\paren{z,w}\abs{f'\paren{w}}^2 dw\\
=&\int_{B_{\epsilon}\paren{z}}G^{ER}_{A_{1,r}}\paren{z,w}\abs{f'\paren{w}}^2 dw\\
&+\int_{A_{1,r} \backslash {B_{\epsilon}\paren{z}}}G^{ER}_{A_{1,r}}\paren{z,w}\abs{f'\paren{w}}^2 dw.
\end{align*}
Since $\abs{f'\paren{r}}$ is bounded on ${B_{\epsilon}\paren{z}}$ and the Green's function for ERBM is integrable, the first integral in the sum is finite.  Since $G^{ER}_{A_{1,r}}\paren{z,\cdot}$ is bounded on $A_{1,r} \backslash B\paren{z,\epsilon}$ and $\int_{A_{1,r}} \abs{f'\paren{w}}^2 dw$ is equal to the area of $D$ (by a straightforward change of variables), the second integral in the sum is also bounded.  
\end{proof}

\begin{proposition} \label{allpaththeoremERBM}
Let $f: \C\backslash\D \rightarrow D$ be a conformal map sending $\infty$ to $\infty$. Then a.s. we have
\begin{equation}
\int_0^t \abs{f'\paren{B_{\C \backslash \D}^{ER}\paren{s}}}^2 ds < \infty
\label{nolimitproofERBM}
\end{equation}
and
\begin{equation}
\int_0^{\infty} \abs{f'\paren{B_{\C \backslash \D}^{ER}\paren{s}}}^2ds = \infty.
\label{allpathproofERBM}
\end{equation}
\end{proposition}

\begin{proof}
For fixed $t$, let $W$ be the set of $\omega$ in the underlying probability space such that the left hand side of \eqref{nolimitproofERBM} is infinite and for each $n \in \N$, let $W_n$ be the set of $\omega$ such that $B^{ER}_{D}$ has not left $A_{1,n}$ by time $t$.  By Lemma \ref{allpathlemma}, the measure of $W_n \cap W$ is zero.  It follows that for almost every $\omega \in W$, the path of $B^{ER}_{\C\backslash \D}$ up to time $t$ is unbounded.  It is easy to see from the definition of ERBM that this implies that $W$ has measure $0$.

It is easy to see that $\abs{f'}$ is bounded below on the set
\[\set{z \in \C: \abs{z}>2}.\]
Since the set of $t$ such that $\abs{B^{ER}_{\C\backslash\D}\paren{t}}>2$ has infinite measure, \eqref{allpathproofERBM} follows.
\end{proof}

Proposition \ref{allpaththeoremERBM} clarifies the implicit use of \eqref{nolimitproofERBM} and its analogs.  The reader can verify that the proof of Proposition \ref{allpaththeoremERBM} does not rely on any of the results that used \eqref{nolimitproofERBM}.  For instance, in the proof of Proposition \ref{ERBMGreenCI} we used the fact that a.s. \eqref{nolimitproofERBM} holds for any finitely connected region $D$. Using the definition of ERBM, it is easy to see that to prove this, it is enough to prove it for any domain conformally equivalent to $\C\backslash \D$.  Notice, however, that the only property of $\C\backslash \D$ we used in the proof of Lemma \ref{allpathlemma} was that $G_{A_{1,r}}^{ER}\paren{z,\cdot}$ is bounded away from $z$ and integrable in a neighborhood of $z$.  Once we know Proposition \ref{allpaththeoremERBM} holds, the proof of Proposition \ref{ERBMGreenCI} works for $D=\C\backslash \D$ and we can use Proposition \ref{ERBMGreenCI} and Proposition  \ref{ERBMGreenDecomp} to conclude that $G_{D}^{ER}\paren{z,\cdot}$ is bounded away from $z$ and integrable in a neighborhood of $z$ for any region $D$ conformally equivalent to $\C \backslash \D$.  This allows us to prove an analog of Proposition \ref{allpaththeoremERBM} for any conformal annulus, which is what we needed.

\subsection{Conformal Mapping Using $G_D^{ER}(z,\cdot)$}
Analogous to the connection between $H_D^{ER}\paren{\cdot,w}$ and conformal maps, there is a connection between $G_D^{ER}\paren{z,\cdot}$ and conformal maps into certain classes of finitely connected domains.

Recall that a domain is a bilateral standard domain if it is an annulus of outer radius $1$ with a finite number of concentric arcs removed.  It is a classical theorem of complex analysis \cite{MR510197} that any finitely connected domain is conformally equivalent to a bilateral standard domain.  The conformal map giving this equivalence is closely related to $G_D^{ER}\paren{A_i,\cdot}$, where $\partial A_i$ is the boundary component mapped to the inner circle of the annulus.

\begin{theorem}\label{bilateralGreenERBM}
Let $D \in \mathcal{Y}_n$ and suppose that there is no Jordan curve in $D$ with $A_0$ in its interior.  If $u=\pi G_D^{ER}\paren{A_i,\cdot}$, then there is a bilateral standard domain $D'$ and a conformal map $f=e^{-\paren{u+iv}}$ from $D$ onto $D'$.  Furthermore, if $g$ is another conformal map from $D$ onto a bilateral standard domain $D''$ and $g$ maps $\partial A_i$ onto the inner radius of $D''$ and $\partial A_0$ onto the outer radius of $D''$, then $f$ and $g$ differ by a rotation.
\end{theorem}

\begin{proof}
The existence of a conformal map from $D$ onto a bilateral standard domain and the uniqueness of the map up to rotation are classical results of complex analysis.  It is also possible to explicitly construct the map $f$ using $G_D^{ER}\paren{A_i,\cdot}$, see \cite{DreThesis} for details.

Suppose $g=e^{-\paren{u+iv}}$ is a conformal map from $D$ onto a bilateral standard domain $D''$ and $g$ maps $\partial A_i$ onto the inner radius of $D''$ and $\partial A_0$ onto the outer radius of $D''$.  To complete the proof, we must show that $u\paren{z}=\pi G_D^{ER}\paren{A_i,z}$.  Observe that $-\log\paren{g}$ is a locally holomorphic, multi-valued function well-defined up to an integer multiple of $2\pi$.  As a result, $u$ is a well-defined harmonic function.  Let $\eta_j$ for $j\neq i$ be a Jordan curve surrounding $A_j$ whose interior contains no point of $A_k$ for $j\neq k$.  On the interior of $\eta_j$, $u+iv$ is a well-defined holomorphic map and as a result,
\[\int_{\eta_j'} \frac{d}{dn} u\paren{z} \abs{dz}=0\]
for any Jordan curve $\eta_j'$ surrounding $A_j$ and in the interior of $\eta_j$.  We conclude by Lemma \ref{ERHarmonicCondition} that $u$ is ER-harmonic on $D \backslash {A_i}$ and since it is equal to zero on $\partial A_0$, it must be a multiple of $G_D^{ER}\paren{A_i,\cdot}$.  Using \eqref{GreenERBMNormalInt}, it is easy to see that the only multiple that will work is $\pi$.
\end{proof}

Recall that a domain is a standard domain if it is the unit disk with a finite number of concentric arcs removed. There is a connection between $G_D^{ER}\paren{z,\cdot}$ for $z\in D$ and conformal maps from $D$ onto standard domains.

\begin{theorem}
Let $D \in \mathcal{Y}_n$ and suppose that there is no Jordan curve in $D$ with $A_0$ in its interior.  If $z \in D$ and $u=\pi G_D^{ER}\paren{z,\cdot}$, then there is a standard domain $D'$ and a conformal map $f=e^{-\paren{u+iv}}$ from $D$ onto $D'$.  Furthermore, if $g$ is another conformal map from $D$ onto a bilateral standard domain that sends $z$ to $0$, then $f$ and $g$ differ by a rotation.
\end{theorem}

\begin{proof}
The proof is similar to that of Theorem \ref{bilateralGreenERBM} and is omitted.
\end{proof}

\section{An Application to a Loewner Equation} \label{sectLE}

\subsection{The Complex Poisson Kernel for ERBM}

We want to prove the analog of Proposition \ref{HalfRMHullCon} for finitely connected domains.  We start with some preliminaries.

\begin{definition}
Let $D\in \mathcal{Y}^*$ and for each $x \in \R \cap \partial D$ define
\begin{equation*}
\displaystyle H_D^{ER}\paren{\infty,x}= \lim_{y \rightarrow \infty} y H_D^{ER}\paren{x+iy,x}.
\end{equation*} 
\end{definition}

When we write $H_D^{ER}\paren{\infty,x}$, it is assumed that $x \in \partial D \cap \R$ even if it is not explicitly stated.  $H_D^{ER}\paren{\infty,x}$ can be interpreted as the normal derivative of $H_D^{ER}\paren{\cdot,x}$ at $\infty$.

Let $R$ be such that $z \in D$ for all $z\in \half$ with $\abs{z}>R$.  Using the strong Markov property for ERBM and \eqref{PKHalfRMHalfDisk}, we have that if $\abs{z}>2R$, then
\begin{align}
H_D^{ER}\paren{z,x}&=R \int_0^{\pi} H_{H^R}\paren{z,Re^{i\theta}} H_D^{ER}\paren{Re^{i\theta},x}d\theta \nonumber \\
&=2 R H_{\half}\paren{z,0} \bracket{\int_0^{\pi} H_D^{ER}\paren{Re^{i\theta},x} \sin \theta~ d\theta}\bracket{1+O\paren{\abs{z^{-1}}}} \label{PKInfxInvar}.
\end{align}
Combining \eqref{PKInfxInvar} with \eqref{PKhalf}, we get the following lemma.

\begin{lemma}\label{PKInfinLem}
Let $D \in \mathcal{Y}^*$ and $R$ be such that $z \in D$ for all $z \in \half$ with $\abs{z}>R$. Then 
\begin{equation*}
H_D^{ER}\paren{\infty,0}=\frac{2R}{\pi} \int_0^{\pi} H_D^{ER}\paren{Re^{i\theta},0} \sin \theta ~d\theta.
\end{equation*}
\end{lemma}


An analog of Proposition \ref{ERPKCI} holds for $H_D^{ER}\paren{\infty,x}$.

\begin{proposition}\label{PKatInfinCI}
Let $D \in \mathcal{Y}^*$ and suppose that $f$ is a conformal map such that $f\paren{D} \in \mathcal{Y}^*$ and $f\paren{A_i}$ is bounded for $1 \leq i \leq n$.  If 
\[ f\paren{z}=a_0+a_1 z + O\paren{\abs{z}^{-1}},~~~z \rightarrow \infty,\]
 then
\[H_D^{ER}\paren{\infty,x}= \frac{\abs{f'\paren{x}}}{a} H_{f\paren{D}}^{ER}\paren{\infty,f\paren{x}}.\]
\end{proposition}

\begin{proof}
Using Proposition \ref{ERPKCI}, we have
\begin{align*}
H_D^{ER}\paren{\infty,x}&=\lim_{y \rightarrow \infty} y H_D^{ER}\paren{iy,x}\\
&=\lim_{y \rightarrow \infty} y H_{f\paren{D}}^{ER}\paren{f\paren{iy},f\paren{x}}\abs{f'\paren{x}}\\
&=\lim_{y \rightarrow \infty} a y H_{f\paren{D}}^{ER}\paren{i a y+O\paren{1},f\paren{x}} \frac{\abs{f'\paren{x}}}{a}\\
&= \frac{\abs{f'\paren{x}}}{a} H_{f\paren{D}}^{ER}\paren{\infty,f\paren{x}}.
\end{align*}
The last equality follows from \eqref{PKInfxInvar} combined with \eqref{PKhalf}.
\end{proof}

The function introduced in the next proposition is a key component in the proof of the analog of Proposition \ref{HalfRMHullCon} for finitely connected domains.
  
\begin{proposition} \label{complexPK}
Let $D \in \mathcal{Y}$ be such that $A_0=\C\backslash \half$.  Then there is a unique function $\mathcal{H}_D^{ER}: D \times \R \rightarrow \C$ satisfying the following.
\begin{enumerate}
\item  For each $x \in \R$, $z \mapsto \mathcal{H}_D^{ER}\paren{z,x}$ is a conformal map onto a chordal standard domain.
\item  The imaginary part of $\mathcal{H}_D^{ER}\paren{z,x}$ is $\pi H_D^{ER}\paren{z,x}$.
\item  For each $x \in \R$,
\[\mathcal{H}_D^{ER}\paren{z,x}=\frac{-\pi H_{D}^{ER}\paren{\infty,x}}{z-x}+O\paren{\abs{z}^{-2}},~~~z \rightarrow \infty.\]
\item  For each $x \in \R$, there is a constant $r\paren{D,x}>0$ such that
\[\mathcal{H}_D^{ER}\paren{z,x}=\frac{-1}{z-x}+r\paren{D,x}+O\paren{\abs{z-x}},~~~z \rightarrow x.\]
\end{enumerate}
\end{proposition}

\begin{proof}
Theorem \ref{PKConformalMapTheorem} implies that for each $x \in \R$ there is a conformal map $\mathcal{H}_D^{ER}\paren{\cdot,x}$ with imaginary part $\pi H_D^{ER}\paren{z,x}$ from $D$ onto a chordal standard domain.  Since a conformal map is uniquely determined up to a real translation by its imaginary part, the uniqueness of $\mathcal{H}_D^{ER}\paren{\cdot,x}$ will follow once we prove its asymptotic at $\infty$. For the remainder of the proof, we will assume $x=0$.  The $x \neq 0$ case can be handled by considering $\mathcal{H}_{D-x}^{ER}\paren{z-x,0}$.

Let $R>0$ be such that $z \in D$ for all $z \in \half$ with $\abs{z}>R$.  For any $z \in D$ with $\abs{z}>2R$, using \eqref{PKHalfRMHalfDisk} and Lemma \ref{PKInfinLem}, we see
\begin{align*}
H_D^{ER}\paren{z,0} &= \frac{\imag\bracket{z}}{\abs{z}^2}\bracket{\frac{2R}{\pi}\int_0^{\pi} H_D^{ER}\paren{Re^{i\theta},0}\sin \theta~d\theta}\bracket{1+O\paren{\frac{R}{\abs{z}}}} \\
 &= \frac{H_D^{ER}\paren{\infty,0}\imag\bracket{z}}{\abs{z}^2}\bracket{1+O\paren{\frac{R}{\abs{z}}}}.
\end{align*}
As a result, if we let 
\[f\paren{z}=\mathcal{H}_D^{ER}\paren{z,0}+\frac{H_D^{ER}\paren{\infty,0}}{z},\] 
then $\imag \bracket{f\paren{z}}=O\paren{\abs{z}^{-2}}$ as $z \rightarrow \infty$.  Combining this with Lemma \ref{HarDerBd}, we see that $f'\paren{z}=O\paren{\abs{z}^{-3}}$ as $z \rightarrow \infty$ and since $f\paren{\infty}=0$, for $\abs{x+iy}>2R$ we have
\begin{align*}
\abs{f\paren{x+iy}}&=\abs{\int_y^{\infty}f'\paren{x+iy'}~dy'}\\
&\leq \int_y^{\infty}\abs{f'\paren{x+iy'}}~dy' \\
&=O\paren{\abs{x+iy}^{-2}}.
\end{align*}
The third statement of the proposition follows.

Let $f\paren{z}=\mathcal{H}_D^{ER}\paren{z,0}+\frac{1}{z}$ and observe that to prove the fourth statement it is enough to show that $f\paren{z}=f\paren{0}+O\paren{\abs{z}}$ as $z \rightarrow 0$.  This will follow by the Schwarz reflection principle if we can show that $\abs{f'\paren{z}}$ (and hence $f\paren{z}$) is bounded in a neighborhood of $0$. The Cauchy-Riemann equations imply that to show this, it is enough to show that the partial derivatives of $\abs{\imag \bracket{f\paren{z}}}$ are bounded in a neighborhood of $0$.  This will follow from Proposition \ref{HarDerBd} if we can show that $\imag\bracket{f\paren{z}}=O\paren{\imag \bracket{z}}$ as $z \rightarrow 0$.

Observe that \eqref{ERBMPKD} and \eqref{BubbleAt0} imply that
\begin{align*}
 \pi H_D^{ER}\paren{z,0}&= \pi H_D\paren{z,0}+\pi \sum_{i=1}^n h_i\paren{z}H_D^{ER}\paren{A_i,0}\\
&= \pi H_{\half}\paren{z,0}-\imag\bracket{z}\Gamma\paren{D;0}\bracket{1+O\paren{\abs{z}}}+ \pi \sum_{i=1}^n h_i\paren{z}H_D^{ER}\paren{A_i,0},
\end{align*}
as $z \rightarrow 0$.  It follows that
\begin{equation}\label{complexPK1}
 \imag\bracket{f}=-\imag\bracket{z}\Gamma\paren{D;0}\bracket{1+O\paren{\abs{z}}}+ \pi \sum_{i=1}^n h_i\paren{z}H_D^{ER}\paren{A_i,0},
\end{equation}
 as $z \rightarrow 0$.  Since $h_i\paren{z}$ is zero on $\R$, we can extend $h_i\paren{z}$ to a function that is harmonic in a neighborhood of $0$.  As a result, letting $z=x+iy$, we can write 
\[h_i\paren{x+iy}=\frac{\partial h_i}{\partial y}\paren{x}y+O\paren{y^2},\]
as $y \rightarrow 0$.  Substituting this into \eqref{complexPK1} and using the fact that $\frac{\partial h_i}{\partial y}\paren{x}$ is bounded in a neighborhood of $0$, the result follows.

\end{proof}

We call the map $\mathcal{H}_D^{ER}$ the \emph{complex Poisson kernel} for ERBM.

\begin{proposition}\label{phiMap}
Let $D \in \mathcal{Y}$ be such that $A_0=\C \backslash \half$ and denote the image of $D$ under the map $z \mapsto \frac{-1}{z-x}$ by $D_x^*$. Then there is a unique conformal map $\varphi_D$ satisfying
\[\lim_{z \rightarrow \infty} \varphi_D\paren{z}-z=0 \]
that maps $D$ onto a chordal standard domain.  Furthermore, for each $x \in \R$, we have
\begin{equation}\label{phiALT}
\varphi_D\paren{z}=\mathcal{H}_{D_x^*}\paren{\frac{-1}{z-x},0}+x-r\paren{D_x^{*},0}.
\end{equation}
\end{proposition}

\begin{proof}
Using Proposition \ref{complexPK}, it is straightforward to verify that the the map in \eqref{phiALT} has the required properties.  The uniqueness of $\varphi_D$ is easy to check using Theorem \ref{PKConformalMapTheorem}.
\end{proof}

A quantity that will be of particular interest to us is $\varphi_D'\paren{x}$ for $x \in \R$.  In what follows, we continue to use the setup of Proposition \ref{phiMap}.

\begin{lemma}
$\varphi_D$ can be extended to a map that is conformal in a neighborhood of any $x \in \R$.  Furthermore, we have
\begin{equation}\label{phiALTDeriv}
\varphi_D'\paren{x}=\pi H^{ER}_{D_x^*}\paren{\infty,0}.
\end{equation}
\end{lemma}
\begin{proof}
The first statement follows from the Schwarz reflection principle.  The formula for $\varphi_D'\paren{x}$ can be computed from \eqref{phiALT} using Proposition \ref{complexPK}.
\end{proof}

An important fact is that $\varphi_D'\paren{x}=\pi H_D^{ER}\paren{\infty,x}$.  This will follow from \eqref{phiALTDeriv} if we can show that $H_D^{ER}\paren{\infty,x}=H^{ER}_{D_x^*}\paren{\infty,0}$.

\begin{lemma}\label{PKinvarCPK}
For any $x\in \R$, $H_{D}^{ER}\paren{\infty,x}=H^{ER}_{D_x^*}\paren{\infty,0}$.
\end{lemma}

\begin{proof}
Using Proposition \ref{GreenND}, Corollary \ref{GreenSym}, and the conformal invariance of $G_D^{ER}$, we have
\begin{align*}
H_{D_x^*}^{ER}\paren{\infty,0}&=\lim_{y \rightarrow \infty}y H_{D_x^*}^{ER}\paren{iy,0}\\
&=\lim_{y \rightarrow \infty} \lim_{\epsilon \rightarrow 0}\frac{y G_{D_x^*}^{ER}\paren{iy,i\epsilon}}{2\epsilon}\\
&=\lim_{y \rightarrow \infty} \lim_{\epsilon \rightarrow 0}\frac{y G_{D}^{ER}\paren{\frac{i}{y}+x,\frac{i}{\epsilon}+x}}{2\epsilon}\\
&=\lim_{y \rightarrow \infty} \lim_{\epsilon \rightarrow 0}\frac{y G_{D}^{ER}\paren{\frac{i}{\epsilon}+x,\frac{i}{y}+x}}{2\epsilon}\\
&=\lim_{\epsilon \rightarrow 0}\lim_{y \rightarrow \infty}\frac{y G_{D}^{ER}\paren{\frac{i}{\epsilon}+x,\frac{i}{y}+x}}{2\epsilon}\\
&= H_D^{ER}\paren{\infty,x}.
\end{align*}
The interchange of limits in the second to last equality is justified by Proposition \ref{ERBMGreenDecomp} and the fact that the interchange is valid when $G_D^{ER}$ is replaced by $G_D$.
\end{proof}

\begin{proposition}\label{phiDerivFin}
For any $x \in \R$, $\varphi_D'\paren{x}=\pi H_D^{ER}\paren{\infty,x}$.
\end{proposition}

\begin{corollary}\label{ERPKInfin1}
If $D$ is a chordal standard domain, then for any $x \in \R$ 
\[\pi H_D^{ER}\paren{\infty,x}=1.\]  
\end{corollary}

\begin{proof}
Since $\varphi_D\paren{z}=z$ when $D$ is a chordal standard domain, the result follows from Proposition \ref{phiDerivFin}.
\end{proof}

Using $\varphi_D\paren{z}$, we can prove an analog of Proposition \ref{HalfRMHullCon} for finitely connected domains.

\begin{proposition}\label{SCRMHullCon}
Let $D \in \mathcal{Y}$ be such that $A_0=\C \backslash \half$ and let $A \in \mathcal{Q}$ be such that $A \cap A_i=\emptyset$ for $1 \leq i \leq n$.  Then there is a unique conformal map $h_A^D$ satisfying
\[\displaystyle \lim_{z \rightarrow \infty}  h^D_A\paren{z}-z=0\]
that maps $D$ onto a chordal standard domain.
\end{proposition}

\begin{proof}
Proposition \ref{HalfRMHullCon} implies there exists a unique conformal map $g_A:\half \backslash A \rightarrow \half$ satisfying
\[\displaystyle \lim_{z \rightarrow \infty}  g_A\paren{z}-z=0.\]  
The map $h_A^D\paren{z}:=\varphi_{g_A\paren{D\backslash A}} \circ g_A\paren{z}$ satisfies the conditions of the proposition.  The uniqueness of $h_A^D\paren{z}$ is easy to check using Theorem \ref{PKConformalMapTheorem}.
\end{proof}

$h_A^D\paren{z}$ has an expansion at infinity
\begin{equation*}
h_A^D\paren{z}=z+\frac{a_1}{z}+O\paren{\abs{z}^{-2}},~~~ z \rightarrow \infty.
\end{equation*}
We call the constant $a_1$ the \emph{excursion reflected half-plane capacity (from infinity)} for $A$ in $D$ and denote it $\hcap^{ER}\paren{A}$.  Since $\hcap^{ER}\paren{A}$ depends not only on $A$, but also on the domain $D$, our notation is misleading, but it will usually be clear from context what domain we mean when we write $\hcap^{ER}\paren{A}$.  We continue to use the setup of Proposition \ref{SCRMHullCon}.

\begin{proposition}\label{ERHCapEqualities}
Let $\tau$ be the smallest $t$ such that $B^{ER}_D\paren{t} \in \partial \half \cup A.$   Then for all $z \in D \backslash A$, we have
\[\imag\bracket{z-h^D_A\paren{z}}=\ev{z}{\imag\bracket{B^{ER}_D\paren{\tau}}}.\]
Also, $\hcap^{ER}\paren{A}$ is equal to each of the following.
\begin{enumerate}
\item $\displaystyle \lim_{y \rightarrow \infty}y\ev{iy}{\imag\bracket{B^{ER}_D\paren{\tau}}}$
\item $\displaystyle \frac{2R}{\pi}\int_0^{\pi}\ev{Re^{i\theta}}{\imag\bracket{B^{ER}_D\paren{\tau}}}\sin \theta ~d\theta$ for any $R$ such that $z \in D$ for all $z \in \half$ with $\abs{z}>R$.
\end{enumerate}
\end{proposition}

\begin{proof}
Since $\imag\bracket{z-h^D_A\paren{z}}$ is a bounded ER-harmonic function and the imaginary part of $h_A\paren{z}$ is equal to $0$ on $\partial \half \cup A$, the first statement follows from Proposition \ref{ERHarmonicFromERBM}.

Using the first part of the proposition, we have
\begin{align*}
\displaystyle \lim_{y \rightarrow \infty} y \ev{iy}{\imag\bracket{B_D^{ER}\paren{\tau}}} &=\lim_{y \rightarrow \infty} y \imag\bracket{iy-h^D_A\paren{iy}}\\
&= \lim_{y \rightarrow \infty} y \imag\bracket{iy-\bracket{iy+\frac{\hcap^{ER}\paren{A}}{iy}+O\paren{y^{-2}}}}\\
&= \hcap^{ER}\paren{A}.
\end{align*}
This proves the first equality for $\hcap^{ER}\paren{A}$.

Using the first equality for $\hcap^{ER}\paren{A}$ and \eqref{PKHalfRMHalfDisk}, we have
\begin{align*}
\hcap^{ER}\paren{A}&=\lim_{y \rightarrow \infty}y \ev{i y}{\imag\bracket{B_D^{ER}\paren{\tau}}}\\
&= \lim_{y \rightarrow \infty}R y\int_0^{\pi} \ev{Re^{i\theta}}{\imag\bracket{B_D^{ER}\paren{\tau}}} H_{H^R}\paren{i y,Re^{i \theta}}~d\theta\\
&=\lim_{y \rightarrow \infty}R y\int_0^{\pi} \ev{Re^{i\theta}}{\imag\bracket{B_D^{ER}\paren{\tau}}} \bracket{\frac{2}{\pi y}\sin \theta \bracket{1+O\paren{y^{-1}}}}~d\theta \\
&=\frac{2R}{\pi}\int_0^{\pi}\ev{Re^{i\theta}}{\imag\bracket{B_D^{ER}\paren{\tau}}}\sin \theta ~d\theta.
\end{align*}
This proves the second equality for $\hcap^{ER}\paren{A}$.
\end{proof}

\begin{lemma}\label{hcapRadAEst}
Let $r=\rad\paren{A}$ and $\tau$ be as in Proposition \ref{ERHCapEqualities}.  Then
\[\hcap^{ER}\paren{A}=2r H_D^{ER}\paren{\infty,0}\bracket{\int_0^{\pi}  \ev{r e^{i\theta}}{\imag\bracket{B^{ER}_D\paren{\tau}}} \sin \theta~d\theta}\bracket{1+O\paren{r}}, \]
as $r \rightarrow 0$.
\end{lemma}

\begin{proof}
Let $R$ be such that $z \in D$ for all $z \in \half$ with $\abs{z}>R$.  Using Proposition \ref{ERPKEstimate}, Proposition \ref{ERHCapEqualities}, and Lemma \ref{PKInfinLem}, we have
\begin{align*}
\displaystyle \hcap^{ER}\paren{A}=& \frac{2R}{\pi}\int_0^{\pi}\ev{Re^{i\theta_1}}{\imag\bracket{B^{ER}_D\paren{\tau}}}\sin \theta_1 ~d\theta_1\\
=& \frac{2 R r}{\pi}\int_0^{\pi} \int_0^{\pi} \ev{re^{i\theta_2}}{\imag\bracket{B^{ER}_D\paren{\tau}}}  H_{D^r}^{ER}\paren{Re^{i\theta_1},re^{i\theta_2}}\sin \theta_1 ~d\theta_2 ~ d\theta_1\\
=& \frac{2 R}{\pi} \bracket{\int_0^{\pi}  H_D^{ER}\paren{Re^{i\theta_1},0}\sin \theta_1 ~d\theta_1}  \\
&  2r\bracket{\int_0^{\pi} \ev{r e^{i\theta_2}}{\imag\bracket{B^{ER}_D\paren{\tau}}} \sin \theta_2  ~d\theta_2} \bracket{1+O\paren{r}}\\
=& 2r H_D^{ER}\paren{\infty,0} \bracket{\int_0^{\pi} \ev{r e^{i\theta_2}}{\imag\bracket{B^{ER}_D\paren{\tau_2}}} \sin \theta_2  ~d\theta_2} \bracket{1+O\paren{r}}
\end{align*}
\end{proof}

The next result gives a uniform bound on the difference between $h_A\paren{z}$ and $z- \frac{\hcap^{ER}\paren{A}\mathcal{H}_D^{ER}\paren{z,0}}{\pi H_D^{ER}\paren{\infty,0}}$ in terms of $\hcap\paren{A}$ and $\rad{A}$.  This can be interpreted as a proof of a Loewner equation for chordal standard domains at $t=0$.

\begin{proposition} \label{LEat0}
Let $D \in \mathcal{Y}$ be such that $A_0=\C \backslash \half$.  There is a constant $c< \infty$ that depends only on $D$ and $z$ such that for all $A \in \mathcal{Q}$ and $z \in D$ with $\abs{z}\geq 2\rad\paren{A}$
\[ \abs{ z - h_A^D\paren{z}- \frac{\hcap^{ER}\paren{A}\mathcal{H}_D^{ER}\paren{z,0}}{\pi H_D^{ER}\paren{\infty,0}}} \leq c \hcap^{ER}\paren{A}\rad\paren{A}.\]
\end{proposition}

\begin{proof}
Let $r=\rad\paren{A},$
\[h\paren{z}=z - h_A^D\paren{z}- \frac{\hcap^{ER}\paren{A}\mathcal{H}_D^{ER}\paren{z,0}}{\pi H_D^{ER}\paren{\infty,0}},\]
and
\[v\paren{z} = \imag \bracket{h\paren{z}}=\imag \bracket{z-h_A^D\paren{z}}-\frac{ \hcap^{ER}\paren{A} H_D^{ER}\paren{z,0}} {H_D^{ER}\paren{\infty,0}}.\]
If $\abs{z} > 2\rad\paren{A}$, using Proposition \ref{ERPKEstimate}, Proposition \ref{ERHCapEqualities}, and Lemma \ref{hcapRadAEst}, we have
\begin{align*}
\imag \bracket{z-h_A^D\paren{z}}&= \ev {z}{\imag\bracket{B^{ER}_D\paren{\tau}}}\\
&=r\int_0^{\pi}  \ev{re^{i \theta}}{\imag\bracket{B^{ER}_D\paren{\tau}}} H_{D^r}^{ER}\paren{z,re^{i\theta}}~d\theta\\
&=2r H_D^{ER}\paren{z,0}\bracket{\int_0^{\pi}  \ev{r e^{i\theta}}{\imag\bracket{B^{ER}_D\paren{\tau}}}\sin \theta~d\theta } \bracket{1+O\paren{r}}\\
&=\bracket{\frac{ H_D^{ER}\paren{z,0}\hcap^{ER}\paren{A}}{ H_D^{ER}\paren{\infty,0}}}\bracket{1+O\paren{r}},
\end{align*}
as $r \rightarrow 0$.  It follows that there exists a $c>0$ such that
\begin{equation}\label{LEAt01} 
\abs{v\paren{z}} \leq  c H_D^{ER}\paren{z,0}\hcap^{ER}\paren{A}\rad\paren{A}.
\end{equation}

Let $R>r$ be such that $z \in D$ for all $z \in \half$ satisfying $\abs{z}>R$, $\tilde{\gamma}$ be a curve from $z$ to $i2R$ (that avoids $B_{2r}^+\paren{0}$), and $M_{\tilde{\gamma}}$ be the maximum value of $H_D^{ER}\paren{\cdot,0}$ restricted to $\tilde{\gamma}$.  Using Lemma \ref{HarDerBd} and \eqref{LEAt01}, we see that there is a $c>0$ such that the partial derivatives of $v$ restricted to $\tilde{\gamma}$ are bounded in absolute value by 
\begin{equation} \label{LEAt011}
c \frac{M_{\tilde{\gamma}}\hcap^{ER}\paren{A}\rad\paren{A}}{d},
\end{equation}
where $d$ is the distance from $\tilde{\gamma}$ to $\partial D$.  It follows that
\begin{equation}\label{LEAt02}
\abs{h\paren{z}-h\paren{i2R}} < c \frac{l M_{\tilde{\gamma}}\hcap^{ER}\paren{A}\rad\paren{A}}{d},
\end{equation}
where $l$ is the length of $\tilde{\gamma}$.  

Using \eqref{PKhalf}, it is easy to check that $H_D^{ER}\paren{iy,0}=O\paren{y^{-1}}$ as $iy \rightarrow \infty$.  Using this fact along with an argument similar to the one used to obtain \eqref{LEAt011}, we see that there is a $c>0$ such that if $y\geq 2R$, then
\begin{equation}\label{LEAt03}
\abs{h'\paren{iy}} \leq \frac{c H_D^{ER}\paren{iy,0}\hcap^{ER}\paren{A}\rad\paren{A}}{y}.
\end{equation}
Combining this with the fact that $h\paren{iy} \rightarrow 0$ as $y \rightarrow \infty$, we have
\begin{align*}
\abs{h\paren{i 2R}}&= \abs{\int_{2R}^{\infty} h'\paren{i y'}~dy'}\\
&\leq \int_{2R}^{\infty} \abs{h'\paren{i y'}}~dy'\\
&\leq  c\hcap^{ER}\paren{A}\rad\paren{A}\int_{2R}^{\infty} \frac{H_D^{ER}\paren{i y',0}}{y'}~dy'\\
&\leq  \frac{c\hcap^{ER}\paren{A}\rad\paren{A}}{2R}.
\end{align*}
Combining this with \eqref{LEAt02}, the result follows.
\end{proof}
With a little more work, it is possible to show that the constant $c$ can be chosen so as to depend only on $D$ and the distance from $z$ to $\partial D$.  We can also get an improved bound for $z \in D$ with $\abs{z}>2R$.

\subsection{The Chordal Loewner Equation in Standard Chordal Domains}
In what follows, let $D$ be a chordal standard domain and $\gamma:[0, \infty) \rightarrow D$ be a simple curve with $\gamma\paren{0}\in \R$.  Denote $\gamma\bracket{0,t}$ by $\gamma_t$ and, for each $t \geq 0$, let $D_t := D \backslash \gamma_t$, $g_t :\half \backslash \gamma_t \rightarrow \half$ be the unique conformal transformation satisfying $\displaystyle \lim_{z \rightarrow \infty} g_t\paren{z}-z=0$, $h_t$ be the unique conformal transformation satisfying $\displaystyle \lim_{z \rightarrow \infty} h_t\paren{z}-z=0$ that maps $D_t$ onto a chordal standard domain, and $\varphi_t$ be the unique map on $g_t\paren{D_t}$ such that $h_t=\varphi_t \circ g_t$.  For each $s>0$, let $\gamma^s\paren{t}=h_s\paren{\gamma\paren{s+t}}$ and $h_{s,t}=h_{\gamma^{s}_{t-s}}^{h_s\paren{D_s}}$.  Observe that $h_t = h_{s,t} \circ h_s$.

Let $b\paren{t}=\hcap^{ER}\paren{\gamma_t}$ and $a\paren{t}=\hcap\paren{\gamma_t}$.  Recall that $h_t$ has an expansion
\begin{equation*}
h_t\paren{z}=z+\frac{b\paren{t}}{z}+O\paren{\abs{z}^{-2}},~~~z \rightarrow \infty.
\end{equation*}
Reparametrizing if necessary, we may assume that $b\paren{t}$ is $C^1$. A priori, we do not know that $\dot{a}\paren{t}$ exists, but later we will show that, in fact, $\dot{a}\paren{t}$ exists if and only if $\dot{b}\paren{t}$ exists and give a formula relating the two quantities.  

Let
\begin{equation*}
\tilde{U}_t=\lim_{s \rightarrow t^-} h_s\paren{\gamma \paren{t}}.
\end{equation*}
Assuming for the moment that $\tilde{U}_t$ is well-defined, we can state our main theorem.

\begin{theorem}\label{SCLE}
For any $z \in D_t$, $h_t\paren{z}$ satisfies the initial value problem
\begin{equation*}
\dot{h}_t\paren{z}=-\dot{b}\paren{t}\mathcal{H}_{h_t\paren{D_t}}^{ER}\paren{h_t\paren{z},\tilde{U}_t},~~~h_0\paren{z}=z.
\end{equation*}
\end{theorem}

The first main step in the proof of Theorem \ref{SCLE} is Proposition \ref{LEat0}, which essentially establishes the theorem for $t=0$.  The second main step is proving that $\tilde{U}_t$ is a continuous function.  We know that
\begin{equation} \label{UtCont2}
U_t=\lim_{s \rightarrow t^-} g_s\paren{\gamma \paren{t}}
\end{equation}
is a well-defined continuous function.  To prove $\tilde{U}_t$ is continuous, the basic idea is to use the continuity of $U_t$ along with estimates for $\varphi'_t\paren{x}$.  Proposition \ref{phiDerivFin} implies that finding estimates for $\varphi'_t\paren{x}$ is equivalent to finding estimates for $H_{g_s\paren{D_s}}^{ER}\paren{\infty,x}$.  The following lemmas provide the necessary estimates.

\begin{lemma}\label{CPKDeriv0}
Let $D\in \mathcal{Y}^*$ be such that there exist constants $0<r<R$ such that $A_i \subset A^+_{r,R}$ for $1 \leq i \leq n$ and $x \in \partial A_0$ for all $x \in \R$ with $\abs{x}>R$.  Then there is a constant $C_{r,R} < \infty$ such that $H_{D}^{ER}\paren{\infty,0}<C_{r,R}$.
\end{lemma}

\begin{proof}
Lemma \ref{PKInfinLem} implies 
\[H_{D}^{ER}\paren{\infty,0}=\frac{2R}{\pi}\int_0^{\pi}H_{D}^{ER}\paren{Re^{i\theta},0}\sin \theta~d\theta.\]  
Since by Lemma \ref{uniPKBound} there is a uniform bound depending only on $r$ for $H_D^{ER}\paren{Re^{i\theta},0}$, the result follows.
\end{proof}

\begin{lemma}\label{CPKDeriv1}
Let $D \in \mathcal{Y}$ be such that $A_0=\C \backslash \half$ and such that $A_i \subset B_R^+\paren{0}$ for $1 \leq i \leq n$.  Then there is a $C_R<\infty$ such that for all $x \in \R$ with $\abs{x}>2R$ we have 
\[\abs{H_{D}^{ER}\paren{\infty,x}-\frac{1}{\pi}}<C_R x^{-2}.\]
\end{lemma}
 
\begin{proof}
Since
\begin{equation*}
\displaystyle H_{\half}\paren{\infty,x}:= \lim_{y \rightarrow \infty} y H_{\half}\paren{x+iy,x}=1/\pi,
\end{equation*}
to complete the proof, it is enough to find a constant $C_R$ such that 
\[\abs{H_{\half}\paren{\infty,x}-H_D^{ER}\paren{\infty,x}}<C_R x^{-2}.\]  
Let $\tau$ be the first time a Brownian motion in $\half$ hits $\partial H^R$.  Since $B_D^{ER}$ has the distribution of a Brownian motion up until the first time it hits $\partial D$, we have
\begin{align*}
\displaystyle \abs{H_{\half}\paren{\infty,x} - H_D^{ER}\paren{\infty,x}} \leq & \lim_{y \rightarrow \infty} y \prob{x+iy}{\abs{B_{\tau}}= R}\\
& \sup_{\theta \in \paren{0,\pi}} \set{\abs{H_D^{ER}\paren{Re^{i\theta},x}+H_{\half}\paren{Re^{i\theta},x}}}.
\end{align*}
Using \eqref{PKHalfRMHalfDisk}, we see that $\displaystyle \lim_{y \rightarrow \infty} y \prob{x+iy}{\abs{B_{\tau}}= R}$ is bounded by a constant depending only on $R$.  Using \eqref{PKhalf}, we see that there is a $c>0$ depending only on $R$ such that 
\begin{equation}\label{CPKDeriv11}\abs{H_{\half}\paren{Re^{i\theta},x}}<c x^{-2}. \end{equation}
Finally, using \eqref{ERBMPKD} and \eqref{twoDomainPK}, we have
\begin{align*} \displaystyle
H_D^{ER}\paren{Re^{i\theta},x} &= H_D\paren{Re^{i\theta},x}+ \sum_{i=1}^n h_i\paren{Re^{i\theta}} H_D^{ER}\paren{A_i,x} \\
& \leq H_{\half}\paren{Re^{i\theta},x}+ \sup_{1 \leq i \leq n} H_D^{ER}\paren{A_i,x}.
\end{align*}
As a result, \eqref{CPKDeriv11} and Lemma \ref{uniPKAiBd} imply that there is a $c>0$ depending only on $R$ such that 
\[\abs{H_D^{ER}\paren{Re^{i\theta},x}}<c x^{-2}.\]
The result follows.
\end{proof}

\begin{lemma}\label{ERPKInfRHull}
Let $D \in \mathcal{Y}$ be such that $A_0=\C \backslash \half$ and suppose that there are constants $r'>0$ and $0<r<R$ such that $w \in D$ for all $w \in \half$ with $\imag\bracket{w}<r'$  and $A_i \subset A_{r,R}^+$ for $1 \leq i \leq n$.  If $A$ is a compact $\half$-hull contained in $B_{r/2}^+\paren{0}$ and $\epsilon=\rad \paren{A}$, then there is a $c>0$ depending only on $r$, $R$, and $r'$ such that
\[H_D^{ER}\paren{\infty,x}-H_{D\backslash A}^{ER}\paren{\infty,x} < c\epsilon,\]
for all $x \in \R$ with $\abs{x}>\epsilon+\sqrt{\epsilon}$.  Furthermore, if $\abs{x}>2R,$ then there is a $c>0$ depending only on $r$, $R$, and $r'$ such that
\[H_D^{ER}\paren{\infty,x}-H_{D\backslash A}^{ER}\paren{\infty,x} < \frac{c\epsilon^2}{x^2}.\]
\end{lemma}

\begin{proof}
Lemma \ref{ERPKEffHull} implies that there is a $c>0$ depending only on $r$, $R$, and $r'$ such that if $\abs{x}>\epsilon+\sqrt{\epsilon}$, then
\[H_D^{ER}\paren{\infty,x}-H_{D\backslash A}^{ER}\paren{\infty,x} \leq c H_D^{ER}\paren{\infty,0}\epsilon.\]
Since, by Lemma \ref{CPKDeriv0}, $H_D^{ER}\paren{\infty,0}$ is bounded by a constant depending only on $r$ and $R$, the first statement follows.  Using the second part of Lemma \ref{ERPKEffHull}, the second statement follows similarly.
\end{proof}

Using the Koebe distortion theorem, we can extend bounds for $\abs{\varphi'\paren{x}}$ on $\R$ to bounds for $\abs{\varphi'\paren{z}}$ restricted to a compact $\half$-hull.

\begin{lemma}\label{CPKDeriv2}
Let $D \in \mathcal{Y}_*$, $A \in \mathcal{Q}$, and $x \in A \cap \R$.  If $\delta>0$ is such that $\dist\paren{A,A_i}>\delta$ for $1 \leq i \leq n$ and $\dist\paren{A,z}>\delta$ for all $z \in \paren{\partial A_0} \backslash \R$, then there is a bound for $\abs{\varphi_D'\paren{z}}$ restricted to $A$ that depends only on $\delta$, $\rad \paren{A}$, and $\varphi'_D\paren{x}.$
\end{lemma}

\begin{proof}
It is easy to see that we can find open balls $B_1, B_2, \ldots B_m$ satisfying the following. 
\begin{enumerate}
\item $B_i$ is a ball of radius $\delta/4$ centered at $c_i$ and $\displaystyle A \subset \bigcup_{i=1}^m B_i$
\item For all $1 \leq i \leq m$, $B_{\delta}\paren{c_i}$ does not intersect $A$ or $A_j$ for any $1 \leq j \leq n$
\item There is an upper bound for $m$ depending only on $\delta$ and $\rad\paren{A}$.
\end{enumerate}
Using the Koebe distortion theorem, we can find a bound depending only on $m$ and $\abs{\varphi_D'\paren{x}}$ for $\abs{\varphi_{s}'\paren{z}}$ restricted to $\displaystyle \bigcup_{i=1}^m B_i$.  Since $A \subset \displaystyle \bigcup_{i=1}^m B_i$, the result follows.
\end{proof}

In what follows, we once again let $D$ and $\gamma$ be as in Theorem \ref{SCLE} and fix $t_0>0$. In order to be able to apply the previous lemmas, we need to find estimates for the distance between $g_s\paren{A_i}$ and $g_s\paren{\gamma}$ that are uniform over all $0 \leq s \leq t_0$ and $1 \leq i \leq n$.

\begin{lemma} \label{gAvoid}
There exist positive constants $r_1<R_1$, $r_2$, and $d$ that depend only on $D$, $\gamma$, and $t_0$ such that for each $0 \leq s \leq t_0$ and $1 \leq i \leq n$ the following hold.
\begin{enumerate}
\item $g_s\paren{A_i} \subset A^+_{r_1,R_1} \paren{U_s}$
\item $g_s\paren{A_i} \subset \set{z \in \half : \imag\bracket{z}>r_2}$
\item The distance between $g_s\paren{\gamma}$ and $g_s\paren{A_i}$ is greater than $d$
\item If $i \neq j$, then the distance between $g_s\paren{A_i}$ and $g_s\paren{A_j}$ is greater than $d$.
\end{enumerate}
Finally, there is a uniform bound on $\diam \bracket{g_s\paren{\gamma\paren{s,t}}}$ over all $0 \leq s < t \leq t_0.$
\end{lemma}

\begin{proof}
Let $\displaystyle R_i\paren{s}=\sup_{z \in A_i} \abs{g_s\paren{z}-U_s}$ and $\displaystyle r_i\paren{s}=\inf_{z \in A_i} \abs{g_s\paren{z}-U_s}$. Lemma \ref{gDiamCurve} and Proposition \ref{UtCont} imply that $R_i\paren{s}$ and $r_i\paren{s}$ are continuous functions of $s$.  This proves the first statement.  The proofs of the remaining statements are similar.  %
\end{proof}

We also need a lemma similar to Lemma \ref{gAvoid} for $h_s$.

\begin{lemma}\label{hAvoid}
There exist constants $0<r<R$ and $r'$ depending only on $D$, $\gamma$, and $t_0$ such that for each $0 \leq s \leq t_0$ and $1 \leq i \leq n$ we have
\begin{enumerate}
\item $h_s\paren{A_i} \subset A^+_{r,R}\paren{\varphi_s\paren{U_s}}$
\item $h_s\paren{A_i}\subset \set{z \in \half : \imag\bracket{z}>r'}$.
\end{enumerate}
\end{lemma}

\begin{proof}
Let $r_1$, $R_1$, and $r_2$ be as in Lemma \ref{gAvoid}.  Using Lemma \ref{CPKDeriv0}, we can find an upper bound $M$ for $\abs{\varphi_s'\paren{U_s}}$ that depends only on $r_1$ and $R_1$.  Using Lemma \ref{PKInfinLem} and the (easy) fact that there is a positive lower bound for $H_{{g_s\paren{D_s}}}^{ER}\paren{2R_1e^{i\theta},0}$ that depends only on $\theta$ and $r_2$, we see that there is a lower bound $m>0$ for $\abs{\varphi_s'\paren{U_s}}$ that depends only on $R_1$ and $r_2$.

The Koebe 1/4 theorem implies that there is a constant $r>0$ that depends only on $r_1$ and $m$ such that $B_r^+\paren{\varphi_s\paren{U_s}} \subset \varphi_s\paren{g_s\paren{D_s}}$.  Since $h_s=\varphi_s \circ g_s$, it follows that $B_r^+\paren{\varphi_s\paren{U_s}} \subset h_s\paren{D_s}$.

The Koebe distortion theorem implies that there is an upper bound that depends only on $M$, $r_2$, and $R_1$ for $\abs{\varphi_s'\paren{z}}$ restricted to the boundary of $ B_{2R_1}^+\paren{U_s}$.  As a result, there is a constant $R>0$ that depends only on $M$, $r_2$, and $R_1$ such that $\varphi_s\paren{B^+_{2R_1}\paren{U_s}} \subset B_{R}^+\paren{\varphi_s\paren{U_s}}.$  The first statement of the proposition follows.
 
Similarly, the Koebe distortion theorem implies that there is a lower bound greater than zero for $\abs{\varphi_s'\paren{x}}$ restricted to $\bracket{-2R_1,2R_1}$ that depends only on $m$, $r_2$, and $R_1$.  Combined with Lemma \ref{CPKDeriv1}, this gives a lower bound for $\abs{\varphi_s'\paren{x}}$ restricted to $\R$.  As a result, the Koebe 1/4 theorem implies that there is an $r'>0$ depending only on $m$, $r_2$, and $R_1$ such that for each $x \in \R$, $B_{r'}^+\paren{\varphi_s \paren{x}} \subset \varphi_s \paren{g_s\paren{D_s}}$.  The second statement of the proposition follows.
\end{proof}

We have the tools to prove an analog of Lemma \ref{gDiamCurve} for $h_s$.

\begin{proposition}\label{ERDiamCurve}
There exists a constant $c < \infty$ that depends only on $D$, $\gamma$, and $t_0$ such that if $0 \leq s < t \leq t_0< \infty$, then
\[\diam \bracket{h_s\paren{\gamma \paren{s,t}}}\leq c \sqrt{\osc \paren{\gamma,t-s,t_0}}\]
and
\[\left\|h_s-h_t\right\|_{\infty} \leq c \sqrt[4]{\osc\paren{\gamma,t-s,t_0}},\]
where
\[\osc\paren{\gamma,\delta,t_0}= \sup \set{\abs{\gamma\paren{s}-\gamma\paren{t}} :0\leq s, t\leq t_0;\abs{t-s} \leq \delta}\]
and $h_s - h_t$ is considered as a function on $D_t$.
\end{proposition}

\begin{proof}
Let $r_1$, $R_1$, and $d$ be as in Lemma \ref{gAvoid} and $\tilde{U}_s=\varphi\paren{U_s}$.  Lemma \ref{CPKDeriv0}, combined with Proposition \ref{phiDerivFin}, shows that there is an upper bound for $\varphi_s'\paren{U_s}$ that depends only on $r_1$ and $R_1$.  As a result, Lemma \ref{CPKDeriv2} implies that there is a bound for $\abs{\varphi_s'\paren{z}}$ restricted to $g_s\paren{\gamma_t}$ that depends only on $r_1$, $R_1$, $d$, and $t_0$.  Since $h_s=\varphi_{s}\circ g_s$, the first statement of the proposition follows from Lemma \ref{gDiamCurve}.

Since $h_t=h_{s,t} \circ h_{s}$, to prove the second statement of the proposition it is enough to show that there is a $c<\infty$, depending only on $D$, $\gamma$, and $t_0$, such that
\begin{equation} \label{ERDiamCurve1}
\left\|h_{s,t}\paren{z}-z\right\|_{\infty} \leq c \sqrt[4]{\osc\paren{\gamma,t-s,t_0}}.
\end{equation}

Let $f_{s,t}:=g_{h_s\paren{\gamma_t}}$, $\phi_{s,t}$ be the unique conformal map such that $h_{s,t}=\phi_{s,t}\circ f_{s,t}$, $d_{s,t} = \diam\bracket{h_s\paren{\gamma \paren{s,t}}}$, and $r$, $R$, and $r'$ be as in Lemma \ref{hAvoid}.  Note that by the Schwarz reflection principle, $\phi_{s,t}$ can be extended to a conformal map on
\[\R \cup \set{z: z \text{ or } \overline{z} \text{ is in the image of } f_{s,t}}.\]

For any $z \in \overline{\half} \backslash h_s\paren{\gamma_t}$, Lemma \ref{gDerivR} implies that 
\begin{equation}\label{ERDiamCurve2}
\abs{f_{s,t}\paren{z}-z} \leq 3d_{s,t}.
\end{equation}
It follows that if $d_{s,t}$ is sufficiently small, then the image of $f_{s,t}$ restricted to $B_r^+\paren{\tilde{U}_s}$ contains $B_{4\sqrt{d_{s,t}}}^+\paren{\tilde{U}_s}$ and the image of $f_{s,t}$ restricted to $B_R^+\paren{\tilde{U}_s}$ is contained in a half-disk centered at $\tilde{U}_s$ with radius depending only on $d_{s,t}$ and $R$.  
Since the first part of the proposition shows that $d_{s,t} \rightarrow 0$ as $\abs{t-s} \rightarrow 0$, it follows that there is a $\delta>0$ and constants $4\sqrt{d_{s,t}}<r_{\delta}<R<R_{\delta}$
such that if $\abs{t-s}<\delta$, then $d_{s,t}<1$ and the image of $f_{s,t}$ contains all $z \in \half$ in the complement of $A_{r_{\delta},R_{\delta}}^+\paren{\tilde{U}_s}$. 

For the remainder of the proof, we assume that $\abs{t-s}<\delta$ and let $c>0$ be a (changing) constant that depends only on $r$, $R$, $r'$, and $\delta$.  Lemma \ref{CPKDeriv0}, combined with the Koebe distortion theorem, shows that there is an upper bound $M_{\delta}$ for $\phi_{s,t}'\paren{x}$ restricted to 
\[\bracket{\tilde{U}_s-\frac{10\sqrt{d_{s,t}}}{3},\tilde{U}_s+\frac{10\sqrt{d_{s,t}}}{3}}\]
that depends only on $r_{\delta}$ and $R_{\delta}$.  Lemma \ref{ERPKInfRHull} implies that if $x \in \R$ satisfies $\abs{x-\tilde{U}_s}>d_{s,t}+\sqrt{d_{s,t}}$, then
\begin{equation}\label{ERDiamCurve3}
H^{ER}_{h_s\paren{D_s}}\paren{\infty,x} - H^{ER}_{h_s\paren{D_s}\backslash h_s\paren{\gamma_t} }\paren{\infty,x}< c d_{s,t}.
\end{equation}
Lemma \ref{gDerivOnR} implies that if $x \in \R$ satisfies $\abs{x-\tilde{U}_s}>3d_{s,t}$, then
\begin{equation}\label{ERDiamCurve4}
1-\frac{c d_{s,t}^2}{\paren{x-\tilde{U}_s}^2}  \leq f_{s,t}'\paren{x} \leq 1.
\end{equation}
Since Lemma \ref{gDerivR} implies that
\[f_{s,t}\paren{\bracket{\tilde{U}_s-3\sqrt{d_{s,t}},\tilde{U}_s+3\sqrt{d_{s,t}}}}\subset \bracket{\tilde{U}_s-\frac{10 \sqrt{d_{s,t}}}{3},\tilde{U}_s+\frac{10 \sqrt{d_{s,t}}}{3}},\]
Corollary \ref{ERPKInfin1} and Proposition \ref{PKatInfinCI}, combined with \eqref{ERDiamCurve3} and \eqref{ERDiamCurve4}, show that for all $x \in \R$ such that $\abs{x-\tilde{U}_s}>\frac{10\sqrt{d_{s,t}}}{3}$, we have 
\begin{equation}\label{ERDiamCurve5}
\abs{\phi_{s,t}'\paren{x}-1} \leq c d_{s,t}.
\end{equation}
Using the second part of Lemma \ref{ERPKInfRHull}, we see that if $x \in \R$ satisfies $\abs{x-\tilde{U}_s}>2R$, then
\begin{equation}\label{ERDiamCurve6}
H^{ER}_{h_s\paren{D_s}}\paren{\infty,x} - H^{ER}_{h_s\paren{D_s}\backslash h_s\paren{\gamma_t} }\paren{\infty,x}< \frac{c d_{s,t}^2}{x^2}.
\end{equation}
Since Lemma \ref{gDerivR} implies that
\[f_{s,t}\paren{\bracket{\tilde{U}_s-2R,\tilde{U}_s+2R}}\subset \bracket{\tilde{U}_s-3R,\tilde{U}_s+3R},\]
using \eqref{ERDiamCurve3} and \eqref{ERDiamCurve6}, we see that for all $x \in \R$ such that $\abs{x-\tilde{U}_s}>3R$ we have 
\begin{equation}\label{ERDiamCurve7}
\abs{\phi_{s,t}'\paren{x}-1} \leq \frac{c d_{s,t}^2}{x^2}.
\end{equation}

Let $x\in \R$ and assume without loss of generality that $x>\tilde{U}_s$.  Since 
\[\displaystyle \lim_{y \rightarrow \infty}\paren{\phi_{s,t}\paren{y}-y}=0,\]
using \eqref{ERDiamCurve5} and \eqref{ERDiamCurve7}, we have
\begin{align*}
\abs{\phi_{s,t}\paren{x}-x}&= \lim_{y \rightarrow \infty} \abs{\paren{\phi_{s,t}\paren{y}-y}-\paren{\phi_{s,t}\paren{x}-x}} \\
&=\abs{\int_{x}^{\infty} \paren{\phi_{s,t}'\paren{y}-1}~dy}\\
&\leq \abs{4\sqrt{d_{s,t}}M_{\delta}+ 3Rc d_{s,t}+c d_{s,t}^2 \int_{3R}^{\infty} \frac{dy}{y^2}}\\
&\leq c\sqrt{d_{s,t}}.
\end{align*}
Combined with the first statement of the proposition and \eqref{ERDiamCurve2}, this proves \eqref{ERDiamCurve1} and hence the proposition in the special case that $\abs{s-t}<\delta.$  Since if $s<r<t$, then both $\osc\paren{\gamma,r-s,t_0}$ and $\osc\paren{\gamma,t-r,t_0}$ are less than $\osc\paren{\gamma,t-s,t_0}$, the general case follows from the special case and the triangle inequality.
\end{proof}

\begin{proposition}
For any $t>0$, there is a unique $\tilde{U}_t \in \R$ such that
\begin{equation*}
\displaystyle \lim_{z \rightarrow \gamma\paren{t}} h_t\paren{z}=\varphi_t\paren{U_t}=\tilde{U}_t,
\end{equation*}
where the limit is taken over $z \in \half \backslash \gamma_t$.  Furthermore,
\begin{equation*}
\tilde{U}_t=\lim_{s \rightarrow t^-} h_s\paren{\gamma\paren{t}}
\end{equation*}
and $t \mapsto \tilde{U}_t$ is a continuous map.
\end{proposition}

\begin{proof}
Using Proposition \ref{ERDiamCurve}, the proof is similar to the analogous proof for $g_t$ (see \cite{MR2129588}).
\end{proof}

The final ingredient in the proof of Theorem \ref{SCLE} is to show that for any $z \in D\backslash \gamma$ the map $t \mapsto \mathcal{H}_{h_t\paren{D_t}}^{ER}\paren{h_t\paren{z},\tilde{U}_t}$ is continuous.  We start by proving the analogous fact for $H_{h_t\paren{D_t}}^{ER}\paren{h_t\paren{z},\tilde{U}_t}$.

\begin{lemma}\label{ERPKCont}
Fix $z \in D$ and let $t_0$ be such that $z \notin \gamma_{t_0}$.  Then there are constants $\delta>0$ and $c>0$ that depend only on $\gamma$, $D$, $z$, and $t_0$ such that if $0<s<t<t_0$ and $t-s<\delta$, then
\[\abs{H_{h_s\paren{D_s}}^{ER}\paren{h_s\paren{z},\tilde{U}_s}-H_{h_t\paren{D_t}}^{ER}\paren{h_t\paren{z},\tilde{U}_t}}<c\sqrt[4]{\osc\paren{\gamma,t-s,t_0}}.\]
\end{lemma}

\begin{proof}
Throughout the proof, all constants will depend only on $D$, $\gamma$, $z$, and $t_0$.  Let $r_{s,t}=\osc\paren{\gamma,t-s,t_0}$, $d_{s,t}$ be as in the proof of Proposition \ref{ERDiamCurve}, and 
\[r_z=\min\set{r,\inf \set{\dist\paren{h_s\paren{z},\tilde{U}_s}: 0\leq s \leq t_0}},\]
where $r$ is as in Lemma \ref{hAvoid}.  Using Proposition \ref{ERDiamCurve}, it is easy to see that there is a $\delta_1>0$ such that if $0<t-s<\delta_1$, then $r_{s,t}<1$ and $h_s\paren{\gamma \paren{s,t}}\subset B_{r_z/2}^+\paren{\tilde{U}_s}$.  

Since Proposition \ref{ERDiamCurve} implies that 
\begin{equation}\label{ERPKCont0}
d_{s,t}+\sqrt{d_{s,t}}=O\paren{\sqrt [4] {r_{s,t}}}
\end{equation}
and $\abs{h_{s,t}\paren{x}-x}=O\paren{\sqrt [4] {r_{s,t}}}$, it follows that
\[h_{s,t}\paren{\tilde{U}_s+d_{s,t}+\sqrt{d_{s,t}}}- h_{s,t}\paren{\tilde{U}_s-d_{s,t}-\sqrt{d_{s,t}}} = O\paren{\sqrt [4] {r_{s,t}}}.\]
As a result, there is a $\delta_2>0$ such that if $0<t-s<\delta_2$, then $h_{s,t}\paren{x} \in B_{r_z/2}^+\paren{\tilde{U}_t}$ for all $x$ such that $\abs{x-\tilde{U}_s} \leq d_{s,t}+\sqrt{d_{s,t}}$.  Let $\delta=\min \set{\delta_1,\delta_2}$ and for the remainder of the proof, assume that $0<t-s<\delta$ and $x=\tilde{U}_s+d_{s,t}+\sqrt{d_{s,t}}$.

Lemma \ref{ERPKDeriv} implies that 
\begin{equation}\label{ERPKCont1}
\abs{H_{h_t\paren{D_t}}^{ER}\paren{h_t\paren{z},\tilde{U}_t}-H_{h_t\paren{D_t}}^{ER}\paren{h_t\paren{z},h_{s,t}\paren{x}}}=O\paren{\sqrt [4] {r_{s,t}}}.
\end{equation}
Using Proposition \ref{ERPKCI}, we see that
\begin{equation}\label{ERPKCont2}
\abs{h_{s,t}'\paren{x}}H_{h_t\paren{D_t}}^{ER}\paren{h_t\paren{z},h_{s,t}\paren{x}}=H_{h_s\paren{D_s}\backslash h_s\paren{\gamma_t}}^{ER}\paren{h_s\paren{z},x}.
\end{equation}
Using the chain rule, Proposition \ref{phiDerivFin}, and Proposition \ref{PKatInfinCI}, we see that
\[h_{s,t}'\paren{x}=\pi H_{g_{h_s\paren{\gamma_t}}\paren{h_s\paren{D_s}}}^{ER}\paren{\infty,g_{h_s\paren{\gamma_t}}\paren{x}}g_{h_s\paren{\gamma_t}}'\paren{x}=\pi H_{h_s\paren{D_s}\backslash h_s\paren{\gamma_t}}^{ER}\paren{\infty,x}.\]
Since $\pi H_{h_s\paren{D_s}}^{ER}\paren{\infty,x}=1$, Lemma \ref{ERPKInfRHull} and Proposition \ref{ERDiamCurve} together imply
\[h_{s,t}'\paren{x}=1+O\paren{\sqrt{r_{s,t}}}.\]
Combining this with \eqref{ERPKCont2} and Lemma \ref{uniPKBound}, we conclude that
\begin{equation}\label{ERPKCont3}
H_{h_s\paren{D_s}\backslash h_s\paren{\gamma_t}}^{ER}\paren{h_s\paren{z},x}-H_{h_t\paren{D_t}}^{ER}\paren{h_t\paren{z},h_{s,t}\paren{x}}= O\paren{\sqrt{r_{s,t}}}.
\end{equation}
Next, using Lemma \ref{ERPKEffHull}, we see that 
\begin{equation}\label{ERPKCont4}
H_{h_s\paren{D_s}\backslash h_s\paren{\gamma_t}}^{ER}\paren{h_s\paren{z},x}-H_{h_s\paren{D_s}}^{ER}\paren{h_s\paren{z},x}= O\paren{\sqrt{r_{s,t}}}.
\end{equation}
Finally, using \eqref{ERPKCont0} and arguing as in \eqref{ERPKCont1}, we see that
\begin{equation}\label{ERPKCont5}
H_{h_s\paren{D_s}}^{ER}\paren{h_s\paren{z},\tilde{U}_s}-H_{h_s\paren{D_s}}^{ER}\paren{h_s\paren{z},x}= O\paren{\sqrt[4]{r_{s,t}}}.
\end{equation}
Combining \eqref{ERPKCont1}, \eqref{ERPKCont3}, \eqref{ERPKCont4}, and \eqref{ERPKCont5}, the result follows.
\end{proof}

It is not hard to see that the proof of Lemma \ref{ERPKCont} can be modified to show that the constants $c$ and $\delta$ can be chosen uniformly over all $z$ in a compact set.

\begin{lemma} \label{ERCPKCont}
Fix $z \in D$ and let $t_0$ be such that $z \notin \gamma_{t_0}$.  Then the map 
\[t \mapsto \mathcal{H}_{h_t\paren{D_t}}^{ER}\paren{h_t\paren{z},\tilde{U}_t}\]
is a continuous function on $[0,t_0)$.
\end{lemma}

\begin{proof}
Assume without loss of generality that $\gamma\paren{0}=0$. Let
\[f_t\paren{w}:=\mathcal{H}_{h_t\paren{D_t}}^{ER}\paren{h_t\paren{w},\tilde{U}_t}=u_t\paren{w}+iv_t\paren{w}\]
and if $s<t$, let $f_{s,t}:=f_t-f_s$ and $v_{s,t}:=v_t-v_s$.  
Let $R$ be as in Lemma \ref{hAvoid} and $\tilde{R}=\max \set{\diam \paren{\gamma_{t_0}},R}.$  Finally, let $\tilde{\gamma}$ be a path in $D_{t_0}$ from $z$ to $i2\tilde{R}$. 

Lemma \ref{ERPKCont} and Lemma \ref{HarDerBd} imply that there is a $\delta_1>0$ and $c>0$ such that if $\abs{t-s}<\delta_1$, then the partial derivatives of $v_{s,t}$ restricted to $\tilde{\gamma}$ are bounded in absolute value by $\frac{c\sqrt[4]{\osc\paren{\gamma, t-s,t_0}}}{d}$, where $d=\dist\paren{\tilde{\gamma},\partial D_{t_0}}$.  As a result, there is a constant $c>0$ that depends only on $\gamma$, $D$, $z$ and $t_0$ such that if $\abs{t-s}<\delta_1$, then $\abs{f_{s,t}'\paren{w}}<c\sqrt[4]{\osc\paren{\gamma, t-s,t_0}}$ for all $w \in \tilde{\gamma}$.  It follows that if $\abs{t-s}<\delta_1$, then
\begin{equation}\label{ERCPKCont0}
\abs{f_{s,t}\paren{z}-f_{s,t}\paren{i2\tilde{R}}}< cl\sqrt[4]{\osc\paren{\gamma, t-s,t_0}},
\end{equation}
where $l$ is the length of $\tilde{\gamma}$.

If $y>2\tilde{R}$, \eqref{PKHalfRMHalfDisk} implies that there is a $c>0$ such that the probability a Brownian motion in $\half$ started at $y$ leaves $\half \backslash \tilde{R}\D$ on $\partial B^+_{\tilde{R}}\paren{0}$ is less than $\frac{c \tilde{R}}{y}$.  Lemma \ref{ERPKCont} implies that there is a $\delta_2>0$ and $c>0$ such that if $\abs{t-s}<\delta_2$, then the maximum value of $\abs{v_{s,t}\paren{z}}$ restricted to $\partial B_{\tilde{R}}^+\paren{0}$ is less than $c\sqrt[4]{\osc\paren{\gamma, t-s,t_0}}$. It follows that there is a $c>0$ such that if $\abs{t-s}<\delta_2$, then
\begin{equation*}
v_{s,t}\paren{y}< \frac{c\sqrt[4]{\osc\paren{\gamma, t-s,t_0}}}{y}.
\end{equation*}
As a result, Lemma \ref{HarDerBd} implies that there is a $c>0$ such that if $\abs{t-s}<\delta_2$, then the partial derivatives of $v_{s,t}\paren{iy}$ are bounded in absolute value by $\frac{c\sqrt[4]{\osc\paren{\gamma, t-s,t_0}}}{y^2}$.  We conclude that there is a $c>0$ such that if $\abs{t-s}<\delta_2$, then
\begin{equation*}
\abs{f_{s,t}'\paren{iy}}<\frac{c\sqrt[4]{\osc\paren{\gamma, t-s,t_0}}}{y^2}
\end{equation*}
for all $y>2\tilde{R}$.  Since $\displaystyle \lim_{y \rightarrow \infty} f_{s,t}\paren{iy}=0$, it follows that if $\abs{t-s}<\delta_2$, then 
\begin{equation}\label{ERCPKCont1}
\abs{f_{s,t}\paren{i2\tilde{R}}}\leq \int_{2\tilde{R}}^{\infty} \abs{f_{s,t}'\paren{iy}}~dy \leq \frac{c\sqrt[4]{\osc\paren{\gamma, t-s,t_0}}}{2\tilde{R}}.
\end{equation}
Combining this with \eqref{ERCPKCont0}, the result follows.
\end{proof}

We have everything we need in order to prove Theorem \ref{SCLE}.

\begin{proof}[Proof of Theorem \ref{SCLE}]
Let $f$ be the conformal map such that 
\[h_{s,s+\epsilon}\paren{z}=f\paren{z-\tilde{U}_s}+\tilde{U}_s.\]  
Applying Proposition \ref{LEat0} to $f$, we see that for sufficiently small $\epsilon >0$,
\begin{align*}
h_{s+\epsilon}\paren{z}-h_s\paren{z}=&h_{s,s+\epsilon}\paren{h_s\paren{z}}-h_s\paren{z} \\
=&-\paren{b\paren{s+\epsilon}-b\paren{s}}\mathcal{H}_{h_s\paren{D_s}}^{ER}\paren{h_s\paren{z},\tilde{U}_s}\\
&+\diam \bracket{\gamma\paren{s,s+\epsilon}}\bracket{ b\paren{ s+\epsilon}-b\paren{s}}O\paren{1}.
\end{align*}
Dividing this by $\epsilon$ and taking the limit as $\epsilon \rightarrow 0$, we see that $h_t\paren{z}$ has right derivative at $s$ equal to
\begin{equation*}
-\dot{b}\paren{s}\mathcal{H}_{h_s\paren{D_s}}^{ER}\paren{h_s\paren{z},\tilde{U}_s}.
\end{equation*}
Using Lemma \ref{conRDeriv} and Lemma \ref{ERCPKCont}, the result follows.
\end{proof}

Up until now we have assumed that the curve $\gamma$ is parametrized such that $b\paren{t}$ is $C^1$.  While this was the most convenient parametrization to use when formulating and proving Theorem \ref{SCLE}, in applications we will often start with a curve that is only assumed to be parametrized such that $a\paren{t}$ is $C^1$.  This will not pose a problem though because, as we now prove, the ER half-plane capacity is $C^1$ if and only if the usual half-plane capacity is.  For the remainder of this section we will assume $D \in \mathcal{Y}^*$ and $\gamma\paren{0} \in \partial D \cap \R$.

\begin{lemma}\label{CapEqZr}
$\dot{b}\paren{0}$ exists if and only if $\dot{a}\paren{0}$ exists.  If both quantities exist, then
\[\dot{b}\paren{0}=\pi H_D^{ER}\paren{\infty,0}\dot{a}\paren{0}.\]
In particular, if $D$ is a chordal standard domain, then $\dot{a}\paren{0}=\dot{b}\paren{0}$.
\end{lemma}

\begin{proof}
Assume without loss of generality that $\gamma\paren{0}=0$ and $z \in D$ for all $z \in \half$ such that $\abs{z}<1$.  Let $r_t=\rad\paren{\gamma_t}$ and
\[\tilde{D}_t=\set{z\in D_t: \abs{z}<1}. \]
Define $\tau_1^t$ to be the first time a Brownian motion in $\half$ exits $\half \backslash \gamma_t$, $\tau_2^t$ to be the first time an ERBM in $D$ exits $D_t$, $X_1^t=\imag \bracket{B_{\tau_1^t}}$, and $X_2^t=\imag \bracket{B_D^{ER}\paren{\tau_2^t}}$.  Finally, define 
\[M_1\paren{t}=\int_0^{\pi} \ev{r_t e^{i\theta}}{X_2^t} \sin\theta~d\theta \]
and observe that Lemma \ref{hcapRadAEst} implies
\begin{equation}\label{CapEqZr0}
b\paren{t}=2r_t H_D^{ER}\paren{\infty,0}M_1\paren{t} \bracket{1+O\paren{r_t}},~~~r_t \rightarrow 0.
\end{equation}
As a result, $\dot{b}\paren{t}$ exists if and only if 
\begin{equation}\label{CapEqZr1}
\displaystyle \lim_{t \rightarrow 0} \frac{r_t M_1\paren{t}}{t}
\end{equation}
exists.  

Let $E_t^z$ be the event that a Brownian motion started at $z \in \tilde{D}_t$ does not leave $\tilde{D}_t$ on $\set{z \in \half:\abs{z}=1}$ and define
\[M_2\paren{t} =\int_0^{\pi} \ev{r_t e^{i\theta}}{X_2^t;E_t^{r_t e^{i\theta}}} \sin\theta~d\theta.\]
We claim that the limit in \eqref{CapEqZr1} exists if and only if
\begin{equation}\label{CapEqZr2}
\displaystyle \lim_{t \rightarrow 0} \frac{r_t M_2\paren{t}}{t}
\end{equation}
exists and if both limits exist, then they are equal.  Observe that
\begin{equation}\label{CapEqZr3}
\ev{r_t e^{i\theta}}{X_2^t}=\ev{r_t e^{i\theta}}{X_2^t;E_t^{r_t e^{i\theta}}}+\int_0^{\pi} \ev{e^{i\theta_1}}{X_2^t} H_{\tilde{D}_t}\paren{r_t e^{i\theta},e^{i\theta_1}}~d\theta_1
\end{equation}
and that, using Proposition \ref{ERPKEstimate},
\begin{equation}\label{CapEqZr4}
\ev{e^{i\theta_1}}{X_2^t}=2r_t H_{D}^{ER}\paren{e^{i\theta_1},0}M_1\paren{t}\bracket{1+O\paren{r_t}},
\end{equation}
for all $r_t<1/2$.  It follows that if the limit in \eqref{CapEqZr1} exists, then 
\[\lim_{t\rightarrow 0} \frac{\ev{e^{i\theta_1}}{X_2^t}}{t}\]
exists and is a continuous function of $\theta_1$.  In particular, there is an upper bound for 
\[\frac{\ev{e^{i\theta_1}}{X_2^t}}{t}\] 
that is uniform over all $0\leq \theta_1 \leq \pi$ and $t$ sufficiently small.  Since the remark following \eqref{PKHD} implies that 
\[\int_0^{\pi}H_{\tilde{D}_t}\paren{r_t e^{i\theta},e^{i\theta_1}}~d\theta_1\]
is comparable to $r_t \sin \theta$, it follows that
\begin{equation}\label{CapEqZr5}
\lim_{t \rightarrow 0} \frac{\int_0^{\pi}H_{\tilde{D}_t}\paren{r_t e^{i\theta},e^{i\theta_1}}\ev{e^{i\theta_1}}{X_2^t}~d\theta_1}{t}=0.
\end{equation}
Combining this with \eqref{CapEqZr3}, it is easy to check that the limit in \eqref{CapEqZr2} exists and is equal to the limit in \eqref{CapEqZr1}.

If the limit in \eqref{CapEqZr2} exists, then using \eqref{CapEqZr3}, \eqref{CapEqZr4}, Lemma \ref{uniPKBound}, and the remark following \eqref{PKHD}, we see that there is a $c>0$ independent of $t$ such that
\begin{align*}
M_1\paren{t}-M_2 \paren{t}&=\int_0^{\pi}\bracket{\int_0^{\pi}H_{\tilde{D}_t}\paren{r_t e^{i\theta},e^{i\theta_1}}\ev{e^{i\theta_1}}{X_2^t}~d\theta_1} \sin \theta~d\theta\\
&\leq  c M_1\paren{t}r_t,
\end{align*}
for all $r_t<1/2$.  Thus, for sufficiently small $t$, $M_1\paren{t} \leq 2M_2\paren{t}$ and as a result 
\[\limsup_{t \rightarrow 0} \frac{r_t M_1}{t}\leq 2\lim_{t \rightarrow 0} \frac{r_t M_2\paren{t}}{t}<\infty.\]
Using this, we can argue as before to show that \eqref{CapEqZr5} holds, from which it is easy using \eqref{CapEqZr3} to show that the limit in \eqref{CapEqZr1} exists and is equal to the limit in \eqref{CapEqZr2}.

Using \eqref{CapEqZr0} and our claim, it follows that $\dot{b}\paren{0}$ exists if and only if the limit in \eqref{CapEqZr2} exists and in that case,
\begin{equation}\label{CapEqZr6}
\dot{b}\paren{0}=2H_D^{ER}\paren{\infty,0}\lim_{t \rightarrow 0}  \frac{r_t M_2\paren{t}}{t}.
\end{equation}
Using Proposition \ref{HCapEqualities}, a similar argument as the one used to prove the claim shows that $\dot{a}\paren{0}$ exists if and only if
\[\lim_{t \rightarrow 0} \frac{\int_0^{\pi} \ev{r_t e^{i\theta}}{X_1^t;E_t^{r_t e^{i\theta}}} \sin\theta~d\theta}{t}\]
exists.  Since Brownian motion and ERBM have the same distribution in $\tilde{D}_t$, this limit is the same as the one in \eqref{CapEqZr2}.  It follows that $\dot{a}\paren{0}$ exists if and only if the limit in \eqref{CapEqZr2} exists and in that case,
\begin{equation}\label{CapEqZr7}
\dot{a}\paren{0}=\frac{2}{\pi}\lim_{t \rightarrow 0}  \frac{r_t M_2\paren{t}}{t}.
\end{equation}
Combining \eqref{CapEqZr6} and \eqref{CapEqZr7}, the result follows.
\end{proof}

\begin{proposition}
$\dot{b}\paren{t}$ exists if and only if $\dot{a}\paren{t}$ exists.  If both quantities exist, then
\[\dot{b}\paren{t}=\varphi_t'\paren{U_t}^2 \dot{a}\paren{t}.\]
\end{proposition}

\begin{proof}
Recall that $\gamma^s\paren{t}=h_s\paren{\gamma\paren{s+t}}$ and define $\alpha\paren{t}= \hcap\paren{\gamma^s\paren{t}}$ and $\beta\paren{t}=\hcap^{ER}\paren{\gamma^s\paren{t}}.$  Since $\hcap^{ER}\paren{\gamma^s\paren{t}}=b\paren{s+t}-b\paren{t}$, $\dot{\beta}\paren{0}$ exists if and only if $\dot{b}\paren{t}$ exists and if they both exist, then they are equal.  Lemma \ref{CapEqZr} implies that $\dot{\alpha}\paren{0}$ exists if and only if $\dot{\beta}\paren{0}$ exists and, in that case, they are equal.  Finally, since $\gamma^s\paren{t}$ is the image under $\varphi_s$ of $g_s\paren{\gamma\paren{s+t}}$, Proposition \ref{LRConHCap} implies that $\dot{\alpha}\paren{0}$ exists if and only if $\dot{a}\paren{s}$ exists and if they both exist, then $\dot{\alpha}\paren{0}=\varphi_s'\paren{U_s}^2\dot{a}\paren{t}$.  The result follows.
\end{proof}

\bibliography{paper}
\bibliographystyle{amsplain}
\end{document}